\documentclass{amsart}

\usepackage{amssymb}
\usepackage{color}
\usepackage{cancel}
\usepackage[titletoc]{appendix}

\newcommand{\rmd}{\mathrm{d}}
\newcommand{\rmi}{\mathrm{i}}
\newcommand{\rme}{\mathrm{e}}

\newcommand{\grad}{\mathrm{grad}}

\newcommand{\I}{\mathcal{I}}
\newcommand{\J}{\mathcal{J}}
\newcommand{\K}{\mathcal{K}}
\newcommand{\R}{\mathcal{R}}
\newcommand{\B}{\mathcal{B}}
\newcommand{\M}{\mathcal{M}}

\newcommand{\kFr}{\mathfrak{k}}

\newcommand{\mFr}{\mathfrak{m}}
\newcommand{\nFr}{\mathfrak{n}}

\newcommand{\cFr}{\mathfrak{c}}
\newcommand{\rFr}{\mathfrak{r}}
\DeclareMathOperator{\modR}{mod}
\DeclareMathOperator{\Ord}{Ord}

\newcommand{\Jac}{\mathrm{Jac}}

\DeclareMathOperator{\Complex}{\mathbb{C}}
\DeclareMathOperator{\Real}{\mathbb{R}}
\DeclareMathOperator{\Rational}{\mathbb{Q}}
\DeclareMathOperator{\Integer}{\mathbb{Z}}

\DeclareMathOperator{\rank}{rank}
\DeclareMathOperator{\card}{card}
\DeclareMathOperator{\spanOp}{span}

\newtheorem{lemma}{Lemma}
\newtheorem{teo}{Theorem}
\newtheorem{prop}{Proposition}

\newtheorem{conj}{Conjecture}

\theoremstyle{definition}
\newtheorem{rem}{Remark}
\newtheorem{exam}{Example}

\theoremstyle{plain}

\title[]{New generalisation \\ of Jacobi's derivative formula}
\author{J Bernatska}
\address{}
\email{jbernatska@gmail.com, julia.bernatska@uconn.edu}
\date{\today}

\allowdisplaybreaks[4]

\begin{document}
 
\maketitle 
\begin{abstract}
A stream of new theta relations is obtained. They follow from the general Thomae formula,
which is a new result giving expressions for 
theta derivatives  (the zero values of the lowest non-vanishing derivatives
of theta functions with singular half-period characteristics)
in terms of branch points and the period matrix of a hyperelliptic Riemann surface.
The new theta relations contain 
(i) linear relations on the vector space of first order theta derivatives which are arranged in gradients,
(ii) relations between second order theta derivatives and symmetric bilinear forms on the vector space of the gradients,
(iii) relations between third order theta derivatives and symmetric trilinear forms on the vector space of the gradients,
and (iv) a conjecture regarding higher order theta derivatives.
It is shown how the Schottky identity (in the hyperelliptic case) is derived from the obtained relations. 
\end{abstract}

\section{Introduction}
The recently discovered general Thomae formula  \cite{BerTF2020}
gives rise to completely new relations between theta constants and theta derivatives. 
The zero values of high order derivatives of theta functions with singular half-period characteristics 
are expressed in terms of theta constants and the zero values of first order derivatives of the theta functions
with non-singular odd characteristics.
In the hyperelliptic case explicit formulae are obtained for the zero values of second and third order derivatives
of the theta functions with characteristics of multiplicities $2$ and $3$. Also formulae for the zero values of 
higher order derivatives are suggested.
On the other hand, known relations  can be derived from these new relations,
for example the Schottky relation and its generalisations.

Start with a brief review of known generalisations of Jacobi's derivative formula
\begin{gather}\label{JDF}
 \theta' \big[\begin{smallmatrix} 1 \\ 1 \end{smallmatrix}\big] = 
 - \pi \theta \big[\begin{smallmatrix} 0 \\ 0 \end{smallmatrix}\big] 
 \theta \big[\begin{smallmatrix} 0 \\ 1 \end{smallmatrix}\big]
 \theta \big[\begin{smallmatrix} 1 \\ 0 \end{smallmatrix}\big].
\end{gather}
Its first generalisation is known as the Rosenhain formulae, see \cite{Ros1895}
and a revision in \cite[13.14 and 13.15]{BEL14}, which express the zero values of determinants of Jacobian matrices 
in terms of theta constants in genus two.
Later Frobenius \cite{Fro} gave similar formulas for genera three and four, and
Riemann examined cases up to genus seven~\cite{Rie}.

This type of generalisation in the hyperelliptic case is known as the Riemann-Jacobi derivative formula.
Let $\delta_i$, $i=1$, \ldots, $g$ denote odd characteristics,
and $\varepsilon_j$, $j=0$, \ldots, $g+1$ denote even characteristics as described in \cite[Theorem 3.4, pp.\,1015--1016]{ER2008}.
Here and in what follows $g$ denotes the genus of a curve. The
Riemann-Jacobi derivative formula has the form
\begin{gather*}
\det \big\| \partial_v \theta [\delta_1], \partial_v \theta [\delta_2], \dots, \partial_v \theta [\delta_g] \big\| =
\pm \pi^g \theta [\varepsilon_0] \theta [\varepsilon_1] \dots \theta [\varepsilon_{g+1}],
\end{gather*}
where $\partial_v \theta [\delta_i]$ denotes the gradient of $\theta [\delta_i]$.

Then Riemann and Weil posed the following questions, see \cite[p.\,170]{Ig1982Bull}.
Let $D_\text{odd}$ denote
$\det \big\| \partial_v \theta [\delta_1], \partial_v \theta [\delta_2], \dots, \partial_v \theta [\delta_g] \big\|$.
\begin{enumerate}
\item[(W-1)]  Is $D_\text{odd}$ always a polynomial in the theta constants with integral rational coefficients?
\item[(W-2)] Can $D_\text{odd}$  be expressed as a quotient of two such polynomials?
\end{enumerate}

In \cite{Ig1979} it was shown that $D_\text{odd}$ is always a rational function of the theta constants 
with rational coefficients.
In \cite{Ig1980} a general theory (including the non-hyperelliptic case) 
was developed which gives the answer to the question 
when an expression for $D_\text{odd}$ becomes a polynomial.

Another well-developed direction of generalisation of \eqref{JDF} 
relates to the elliptic theta functions with rational 
characteristics\footnote{The theta function with a rational characteristic is defined 
for all $v\in\Complex$ and $\tau \in\Complex$, $\Im \tau >0$ by
\begin{gather*}
\theta \big[\begin{smallmatrix} a \\ b \end{smallmatrix}\big](v;\tau)
= \sum_{n\in\Integer} \exp \big(\rmi \pi (n+a)^2 \tau + 2 \rmi \pi (n+a)(v+b) \big),
\end{gather*}
where $a$, $b\in \Rational$.
}. In \cite[p.\,117]{Mum} Mumford posed the question about generalisation of Jacobi’s formula to the form
\begin{gather}
 \frac{\partial }{\partial v} \theta \big[\begin{smallmatrix} a \\ b \end{smallmatrix}\big](v;\tau)\big|_{v=0} 
 = \{ \text{cubic polynomial in } \theta \big[\begin{smallmatrix} c \\ d \end{smallmatrix}\big]\text{'s}\}
\end{gather}
for all $a$, $b\in \Rational$, where $c$, $d \in \Rational$. 
Are there similar generalisations of Jacobi’s formula with higher order
derivatives? In \cite{Fra1}  the following theta constant identity was derived:
\begin{gather*}
\frac{6 \theta' \big[\begin{smallmatrix} 1 \\ 1/3 \end{smallmatrix}\big](0;\tau)}
{\rme^{\rmi\pi/3} \theta^3 \big[\begin{smallmatrix} 1/3 \\ 1/3 \end{smallmatrix}\big](0;\tau)
+ \theta^3 \big[\begin{smallmatrix} 1/3 \\  1 \end{smallmatrix}\big](0;\tau)
+ \rme^{5\rmi\pi/3} \theta^3 \big[\begin{smallmatrix} 1/3 \\ 5/3 \end{smallmatrix}\big](0;\tau)} 
= 2\pi \rmi \frac{\theta \big[\begin{smallmatrix} 1 \\ 1/3 \end{smallmatrix}\big](0;9\tau)}
{\theta \big[\begin{smallmatrix} 1 \\ 1/3 \end{smallmatrix}\big](0;3\tau)}.
\end{gather*}
That is, $\theta' \big[\begin{smallmatrix} 1 \\ 1/3 \end{smallmatrix}\big](0,\tau)$
is given by a rational expression of theta constants, and so 
can be viewed as an analogue of Jacobi’s derivative formula.
Note that some theta constants have the second argument equal to $a\tau$ with integer $a$.

With the help of Jacobi's triple product identity (for more details on Jacobi's triple and quintuple product 
identities see \cite[Ch.\,2 \S\,8]{FK2001})
the following results are derived.
In \cite{Mat2016} expressions for $\theta'  \big[\begin{smallmatrix} 1 \\ 1/2 \end{smallmatrix}\big]$, 
$\theta'  \big[\begin{smallmatrix} 1 \\ 1/4 \end{smallmatrix}\big]$ and 
$\theta'  \big[\begin{smallmatrix} 1 \\ 3/4 \end{smallmatrix}\big]$ as cubic polynomials in the theta constants with
rational characteristics are obtained. 
In \cite{Mat2017}  rational expressions in terms of the theta constants 
are found for $\theta'  \big[\begin{smallmatrix} 1 \\ \varepsilon'  \end{smallmatrix}\big]$ 
with $\varepsilon' = \frac{1}{3}$, $\frac{2}{3}$, $\frac{1}{4}$, $\frac{3}{4}$, $\frac{1}{5}$, $\frac{2}{5}$, 
$\frac{3}{5}$, and $\frac{4}{5}$.
A proof of this uses the elementary fact that a holomorphic elliptic function is a constant. 
This method was proposed in \cite{FK2}, and is applicable to arbitrary rational characteristics.
Further progress in this direction was made in \cite{Zemel2019}, 
where polynomial expressions for
theta derivatives with the characteristics from $\frac{1}{4}\Integer$ and  $\frac{1}{3}\Integer$
were found.




The third direction of generalisation of Jacobi’s derivative formula 
involves higher order derivatives of theta functions. 
For example 
\begin{gather*}
\det \begin{pmatrix}  \theta \big[\begin{smallmatrix} 0 \\ 0 \end{smallmatrix}\big] (0;\tau)
& \theta \big[\begin{smallmatrix} 1 \\ 0 \end{smallmatrix}\big](0;\tau)  \\
 \frac{\partial^2}{\partial v^2} \theta \big[\begin{smallmatrix} 0 \\ 0 \end{smallmatrix}\big](v;\tau)\big|_{v=0} 
 &  \frac{\partial^2}{\partial v^2}  \theta \big[\begin{smallmatrix} 1 \\ 0 \end{smallmatrix}\big](v;\tau)\big|_{v=0} 
 \end{pmatrix} \\
  = 4 \pi \rmi \det 
 \begin{pmatrix}  \theta \big[\begin{smallmatrix} 0 \\ 0 \end{smallmatrix}\big] (0;\tau)
& \theta \big[\begin{smallmatrix} 1 \\ 0 \end{smallmatrix}\big](0;\tau)  \\
 \frac{\partial}{\partial \tau} \theta \big[\begin{smallmatrix} 0 \\ 0 \end{smallmatrix}\big](0;\tau) 
 &  \frac{\partial}{\partial \tau}  \theta \big[\begin{smallmatrix} 1 \\ 0 \end{smallmatrix}\big](0;\tau)
 \end{pmatrix}
\end{gather*}
can be expressed as a polynomial in theta constants and 
theta derivatives, the two determinants are equal by the heat equation. 
This expression follows from the identity
\begin{gather}\label{detTheta}
\det \begin{pmatrix}  \Theta \big[\begin{smallmatrix} 0 \\ 0 \end{smallmatrix}\big] (0;\tau)
& \Theta \big[\begin{smallmatrix} 1 \\ 0 \end{smallmatrix}\big](0;\tau)  \\
 \frac{\partial}{\partial \tau} \Theta \big[\begin{smallmatrix} 0 \\ 0 \end{smallmatrix}\big](0;\tau) 
 &  \frac{\partial}{\partial \tau}  \Theta \big[\begin{smallmatrix} 1 \\ 0 \end{smallmatrix}\big](0;\tau)
 \end{pmatrix} = \frac{\rmi}{4\pi} 
 \Big(  \frac{\partial}{\partial v} \theta \big[\begin{smallmatrix} 1 \\ 1 \end{smallmatrix}\big](v;\tau)\big|_{v=0} \Big)^2,
\end{gather}
where $\Theta[\varepsilon] (v;\tau)$ stands for 
$\theta [\varepsilon] (2v;2\tau)$ and denotes the second order\footnote{With a positive integer $N$, 
an $N$-th order theta function with characteristic $[\varepsilon]=\big[\begin{smallmatrix} \varepsilon\phantom{'} \\
\varepsilon' \end{smallmatrix}\big]$, $\varepsilon$, $\varepsilon'\in\Real$, is such a function $\varphi$
which satisfies the following functional equations 
\begin{gather*}
\varphi(v+1) = \rme^{\pi \rmi \varepsilon} \varphi(v),\\
\varphi(v+\tau) = \rme^{-\pi \rmi (\varepsilon' + 2 N v + N \tau)} \varphi(v),\quad \forall v\in\Complex.
\end{gather*}
} theta function with characteristic $[\varepsilon]$. 
The identity \eqref{detTheta}
is deduced from \cite[Ch.\,2 Theorem 5.3]{FK2001} by applying Jacobi’s triple product identity and 
changing to the argument $\tau/2$. In \cite{GM2005}  this identity
was generalised to a higher genus.

Finally, recall one more generalisation, which is the closest to the one proposed in this paper.
In \cite{Gra1988} the theta derivative with non-singular characteristic equal to the vector of Riemann constants
in genus $2$ is expressed in terms of theta constants. A proof is based on
the addition formula for sigma function and relations between $\wp$  functions. 
This result arises as a particular case in \cite[Eqs (31), (30)]{BerTF2020}, where 
it is derived with the help of other technique.

The present paper is a continuation and development of \cite{BerTF2020},
where the general Thomae formula in hyperelliptic case  is obtained. This new formula
expresses theta derivative of an arbitrary order in terms of branch points and period matrix
of a hyperelliptic Riemann surface. 
By an $m$-th \emph{order theta derivative} we call the zero value
of order~$m$ partial derivative of theta function
with a singular characteristic of multiplicity~$m$. 
Zero values of first order derivatives of theta functions 
with odd non-singular characteristics are called \emph{first order theta derivatives} or simply \emph{theta derivatives}.
Zero values of theta functions with even non-singular characteristics are called
\emph{theta constants} as usual.

Applying the second Thomae formula one can derive relations between theta derivatives
and theta constants from the general Thomae formula. Furthermore, these relations show a clear structure.
Let $V$ be the vector space of theta derivatives,
namely:
\begin{gather*}
V = \spanOp \{\partial_v \theta[\varepsilon]\mid \varepsilon 
\text{ are odd non-singular characteristics of half-periods}\},
\end{gather*}
where $\partial_v \theta[\varepsilon]$ is a gradient of $\theta[\varepsilon](v;\tau)$ at $v=0$.
Every vector $\partial_v \theta[\varepsilon]$ has $g$ components. And $V$  
has dimension $g$ indeed, that follows from linear relations between gradients of theta functions
with different odd non-singular characteristics, see section~\ref{ss:FDT}.
Second order theta derivatives are expressed through symmetric bilinear forms on  $V$,
see section  \ref{ss:2derThetaC}. 
Third order theta derivatives are represented as symmetric trilinear forms on $V$, 
see section  \ref{ss:3derThetaC}, and so on.
All these relations can be viewed as 
\emph{generalisations of Jacobi's derivative formula} since they relate high order theta derivatives 
and theta constants, similar to~\eqref{JDF}. 
These basic relations can be used to produce further relations. As an example the Schottky identity is derived below.

The paper is organised as follows. Section 2 contains the minimal background: definitions and notation
regarding theta functions, Thomae theorems with corollaries, and some auxiliary lemmas. 
Section 3 is devoted to the main result: deriving relations between theta derivatives and
 theta constants.  Relations with first, second and third order theta derivatives are described in detail,
and a conjecture about higher order theta derivatives is made. 
In Section 4 Schottky type relations are derived from the relations with second order theta derivatives.
A variety of explicit relations can be found in appendices.

\section{Preliminaries}
\subsection{Hyperelliptic curve and its Jacobian}\label{ss:HypC}
In the paper we deal with hyperelliptic curves $\mathcal{C}$ defined by
\begin{gather}\label{Cg}
 0 = f(x,y) = -y^2 + \prod_{j=1}^{2g+1} (x-e_j),
\end{gather}
where $\{(e_j,0)\}$ are finite branch points and $(x,y)\in \Complex^2$. 
In what follows, for the sake of brevity notation $e_i$ is employed both for branch point $(e_i,0)$ 
and its $x$-coordinate.
All branch points are distinct, so the curves are non-degenerate. 
One more branch point is located at infinity, and serves as the base point.

A homology basis is adopted from Baker \cite[p.\,303]{bak898}.
Imagine a continuous path through all branch points which ends at infinity.
The branch points are denoted by $e_1$, $e_2$, \ldots, $e_{2g+1}$ along the path, infinity is denoted by $e_{0}$.
Cuts are made between points $e_{2k-1}$ and $e_{2k}$ with $k$ from $1$ to $g$ and between $e_{2g+1}$ and $e_0$. 
With $k$ running from $1$ to $g$
an $\mathfrak{a}_k$-cycle encircles the cut $(e_{2k-1},e_{2k})$ clockwise,
and $\mathfrak{b}_k$-cycle comes out from the cut $(e_{2k-1},e_{2k})$ and  enters into the cut $(e_{2g+1},e_{0})$.
These cycles constitute a canonical homology basis on a curve \eqref{Cg}.

Standard holomorphic differentials are employed, see for example \cite[p.\,306]{bak898},
\begin{align}\label{HDifCg}
\rmd u_{2n-1} = \frac{x^{g-n} \rmd x}{-2y} ,\qquad n=1,\,\dots,\,g.  
\end{align} 
All together they are denoted by $\rmd u = (\rmd u_1$, $\rmd u_3$, $\dots$, $\rmd u_{2g-1})^t$.
The differentials are labeled by Sato weights in subscripts. The weight shows
the exponent of the first term in the expansion of the corresponding integral $u_{2n-1}$ 
about infinity in the local parameter $\xi$ which is introduced by $x=\xi^{-2}$.

Integrals of the holomorphic differentials along the canonical homology cycles give
first kind periods
\begin{align*}
 &\omega = (\omega_{ij})= \bigg( \int_{\mathfrak{a}_j} \rmd u_i\bigg),&
 &\omega' = (\omega'_{ij}) = \bigg(\int_{\mathfrak{b}_j} \rmd u_i \bigg).&
\end{align*}
The both first kind periods form the matrix $(\omega,\omega')$ of non-normalized periods.
Normalized period matrix is $(1_g,\tau)$, where $1_g$ denotes the identity matrix of size $g$, 
and $\tau = \omega^{-1}\omega'$. Matrix $\tau$ is symmetric with a positive imaginary part: 
$\tau^t=\tau$, $\Im \tau >0$,
that is $\tau$ belongs to the Siegel upper half-space. 
Normalised holomorphic differentials are defined  by
\begin{gather*}
 \rmd v = \omega^{-1} \rmd u,
\end{gather*}
and Abel's map $\mathcal{A}$ with respect to these differentials is
\begin{gather}\label{AbelM}
 \mathcal{A}(P) = \int_{\infty}^P \rmd v,\qquad P=(x,y)\in \mathcal{C}.
\end{gather}
Abel’s map of a positive divisor $\mathcal{D} =  \sum_{i =1}^n P_i$ on $\mathcal{C}$ is defined by
\begin{gather}\label{AbelMDiv}
 \mathcal{A}(\mathcal{D}) = \sum_{i =1}^n \int_{\infty}^{P_i} \rmd v.
\end{gather}
$\mathcal{A}$ maps a genus $g$ curve to its Jacobian variety 
$\Jac(\mathcal{C})=\Complex^g\backslash \mathfrak{P}$,
which is the quotient space of $\Complex^g$ by the lattice $\mathfrak{P}$ of periods.
The lattice is constructed from columns of normalised period matrix $(1_g,\tau)$.
We denote points of Jacobian by $v$ with coordinates $(v_1,v_2,\dots,v_g)^t$.
The map is one-to-one on the non-special locus of $g$-th symmetric power of the curve.

\subsection{Theta functions}
The Riemann theta function related to a curve $\mathcal{C}$ is defined
on its Jacobian $\Jac(\mathcal{C})$ by Fourier series
\begin{gather}\label{ThetaDef}
 \theta(v;\tau) = \sum_{n\in \Integer^g} \exp \big(\rmi \pi n^t \tau n + 2\rmi \pi n^t v\big),
\end{gather}
see for example \cite[p.\,118]{Mum}.

Values of theta function at half periods are of importance. 
Relations between these values are an objective of the paper. 
Therefore, it is suitable to introduce \emph{theta function}
 \emph{ with half-period characteristic} as 
in \cite[Def\;3 p.\,4]{RF1974}, namely
\begin{multline}\label{ThetaDefChar}
 \theta[\varepsilon](v;\tau) = \exp\big(\rmi \pi (\varepsilon'{}^t/2) \tau (\varepsilon'/2)
 + 2\rmi \pi  (v+\varepsilon/2)^t \varepsilon'/2\big) \times \\ \times \theta(v+\varepsilon/2 + \tau \varepsilon'/2;\tau),
\end{multline}
where $g$-component vectors $\varepsilon$ and $\varepsilon'$ consist of entries $0$ and $1$.
Characteristic $[\varepsilon]$ is a $2\times g$ matrix of the form
\begin{gather*}
 [\varepsilon] = \begin{bmatrix} \varepsilon'{}^t \\ 
\varepsilon^t \end{bmatrix}.
\end{gather*}
Each branch point $e$ of a hyperelliptic curve \eqref{Cg} is identified with a half-period
\begin{gather}\label{AbelMBranchP}
 \mathcal{A}(e) = \int_{\infty}^{e} \rmd v = \varepsilon/2 + \tau \varepsilon'/2,
\end{gather}
where vectors $\varepsilon$ and $\varepsilon'$ form characteristic $[\varepsilon]$,
see \cite[\S\,202 p.\,300-301]{bak897}.

One can add characteristics by the rule: $[\varepsilon]+[\delta] = ([\varepsilon]+[\delta]) \modR 2$.
A characteristic $[\varepsilon]$ is odd whenever $\varepsilon^t \varepsilon' \modR 2 = 1$, 
and even whenever $\varepsilon^t \varepsilon' \modR 2 = 0$. Theta function with characteristic
has the same parity as its characteristic.

\subsection{Characteristics in hyperelliptic case}\label{ss:CharHyper}
A method of consructing characteristics in the hyperelliptic case are employed from \cite[p.\,1012]{ER2008}.
It is based on the definition \eqref{AbelMBranchP} of a half-period characteristic  with the help of 
Abel's map \eqref{AbelM}.
We denote characteristics of branch points $e_k$ by $[\varepsilon_k]$;
and $[\varepsilon_{0}]=0$. Guided by the picture of canonical homology cycles, we obtain
\begin{align*}
 &\mathcal{A}(e_{2g+1}) = \mathcal{A}(e_{0}) - \sum_{k=1}^g \int_{e_{2k-1}}^{e_{2k}} \rmd v &
 &[\varepsilon_{2g+1}] = \big[ {}^{00\dots 00}_{11\dots 11} \big], & \\
 &\mathcal{A}(e_{2g}) = \mathcal{A}(e_{2g+1}) + \int_{e_{2g}}^{e_{2g+1}} \rmd v = &
 &[\varepsilon_{2g}] = \big[ {}^{00\dots 01}_{11\dots 11} \big], & \\
 &\mathcal{A}(e_{2g-1}) = \mathcal{A}(e_{2g}) + \int^{e_{2g}}_{e_{2g-1}} \rmd v&
 &[\varepsilon_{2g-1}] = \big[ {}^{00\dots 01}_{11\dots 10} \big], &
\intertext{for $k$ from $g-1$ to $2$}
 &\mathcal{A}(e_{2k}) = \mathcal{A}(e_{2k+1}) + \int_{e_{2k}}^{e_{2k+1}} \rmd v = &
 &[\varepsilon_{2k}] = \big[ \overbrace{{}^{00\dots 0}_{11\dots 1}}^{k-1}\!{}^{10 \dots 0}_{10 \dots 0} \big], & \\
 &\mathcal{A}(e_{2k-1}) = \mathcal{A}(e_{2k}) + \int^{e_{2k-1}}_{e_{2k}} \rmd v&
 &[\varepsilon_{2k-1}] = \big[ \overbrace{{}^{00\dots 0}_{11\dots 1}}^{k-1}\!{}^{10 \dots 0}_{00 \dots 0} \big], &
\intertext{and finally}
 &\mathcal{A}(e_{2}) = \mathcal{A}(e_{3}) + \int_{e_{2}}^{e_{3}} \rmd v &
 &[\varepsilon_{2}] = \big[ {}^{10\dots 0}_{10\dots 0} \big], & \\
 &\mathcal{A}(e_{1}) = \mathcal{A}(e_{2}) + \int^{e_{1}}_{e_{2}} \rmd v &
 &[\varepsilon_{1}] = \big[ {}^{10\dots 0}_{00\dots 0} \big]. &
 \end{align*}

Characteristic $[K]$ of the vector of Riemann constants equals 
the sum of all odd characteristics of branch points, there are $g$ such characteristics
according to \cite[\S\,200 p.\,297, \S\,202 p.\,301]{bak897}, and \cite[\S\,VII.1.2 p.\,305]{FK1980}.
Actually,
\begin{gather*}
 [K] = \sum_{k=1}^g [\varepsilon_{2k}].
\end{gather*}

\subsection{Characteristics and partitions}\label{ss:CharPart}
Let $\I\cup \J$ be a partition of the set of indices of all branch points $\{e_0,\, e_1,\dots,\, e_{2g+1}\}$ of a curve.
Denote by $[\varepsilon(\I)] = \sum_{i\in\I} [\varepsilon_i]$ the characteristic of 
\begin{gather}
  \mathcal{A}(\I) = \sum_{i\in\I}  \mathcal{A}(e_i) = \varepsilon(\I) /2 + \tau \varepsilon'(\I) /2.
\end{gather}

According to \cite[p.\,13]{fay973} (for more details see \cite[\S\,202 p.\,301]{bak897}) 
there is one-to-one correspondence between
$2^{2g}$ characteristics and $2^{2g}$ partitions $\I_\mFr \cup \J_\mFr$ with $\I_\mFr = \{i_1,\,\dots,\, i_{g+1-2\mFr}\}$
and $\J_\mFr = \{j_1,\,\dots,\, j_{g+1+2\mFr}\}$ of the set of $2g+2$ 
indices of branch points of a hyperelliptic curve, where $\mFr$ is between $0$ and $[(g+1)/2]$. 
Here $[z]$ means the integer part of $z$. 
Number $\mFr$ is called \emph{multiplicity}. Thus, characteristics of $2g+2$ branch points of 
a hyperelliptic curve \eqref{Cg} 
serve as a basis to construct
all $2^{2g}$ half-period characteristics.

The index of infinity is usually omitted, it belongs to the part $\I_\mFr$ or $\J_\mFr$ 
with the number of indices less than $g+1-2\mFr$ or $g+1+2\mFr$, respectively. 
Notation $\I_\mFr$ is used for the part of less cardinality.

Introduce also characteristic $[\I_\mFr] = [\varepsilon(\I_\mFr)] + [K]$ of $\mathcal{A}(\I_\mFr) + K$
which corresponds to a partition $\I_\mFr \cup \J_m$. Here $K$ denotes the vector of Riemann constants,
and $[K]$ its characteristic. Note that $[\J_\mFr]$ represents the same characteristic as $[\I_\mFr]$.
Characteristics $[\I_\mFr]$ of even multiplicity $\mFr$ are even, and of odd $\mFr$ are odd.
According to the Riemann vanishing theorem 
$\theta(v+\mathcal{A} (\I_\mFr)+K)$ vanishes to order~$\mFr$ at $v=0$, see \cite[p.\,13]{fay973}.
Characteristics of multiplicity $0$ are called non-singular even characteristics,
there are $\binom{2g+1}{g}$ such characteristics.
There exist $\binom{2g+2}{g-1}$ characteristics of multiplicity $1$, which are called non-singular odd.
The number of characteristics of multiplicity $\mFr > 1$ is $\binom{2g+2}{g+1-2\mFr}$.

Below characteristic $[K]$ is also denoted by $[\emptyset]$ since it corresponds
 to partition $\emptyset\cup\{1,2,\dots,2g+1\}$, 
which is always unique. At that  $\theta[K](v)$ vanishes to the maximal order $[(g+1)/2]$ at $v=0$, 
see for example \cite[\S\,VII.1.4 p.\,306]{FK1980}. 
The maximal order is greater than~$1$ if $g>2$.
If the genus $g$ of a curve is odd the mentioned partition contains the empty set: $\I_{(g+1)/2}=\emptyset$, 
because the omitted index of infinity belongs to $\J_{(g+1)/2}$.
It means that $\theta[K](v)$ is the only theta function with characteristic of the maximal multiplicity.
If $g$ is even, part $\I_{g/2}=\emptyset$ of the partition actually contains the omitted index of infinity.
In this case $\theta[K](v)$ is not the only theta function with characteristic of the maximal multiplicity,
there exist $2g+2$ such functions. Vanishing to the order equal to multiplicity 
is the distinctive property of hyperelliptic curves, see \cite[\S\,VII.1.5, 7--8 p.\,308--309]{FK1980}.

Characteristics of multiplicity $0$ are usually described by the partitions whose part~$\I_0$ 
contains $g$ non-zero indices and the omitted infinity index. 
Partitions corresponding to characteristics of multiplicity $1$ can be obtained from 
$\I_0\cup \J_0$ by moving two indices from $\I_0$ to $\J_0$. These two indices can be of finite points the both
or one index of a finite point and the index of infinity. 
The former case is described by dropping two indices from $\I_0$, 
so cardinality of the obtained $\I_1$ becomes $g-2$, then $\J_1$ contains $g+3$ indices
and the infinity index is located in $\I_1$. 
In the latter case only one index drops from $\I_0$,
so the cardinality of $\I_1$ is $g-1$, and the infinity index moves into $\J_1$ which has cardinality $g+2$.
(Note that computing cardinality of a set we always omit the zero index of infinity.)
The same occurs in higher multiplicities.
Regarding characteristics $[\I_\mFr]$, it is convenient to distinguish between  
ones with the index of infinity in part $\J_\mFr$ ($\card \I_\mFr = g+1-2\mFr$),
and ones with the index of infinity in part $\I_\mFr$ (cardinality of $\I_m$ is $g-2\mFr$). 

In what follows we also use notation $\B^{[r]}$ for a set of cardinality $r$.

\subsection{Notations}
In what follows $\partial_{v_i}$ stands for a partial derivative with respect to variable $v_i$, and 
argument $\tau$ of a theta function is usually omitted, so 
\begin{align*}
 &\partial_{v_i} \theta[\varepsilon](v) \equiv \frac{\partial}{\partial v_i}  \theta[\varepsilon](v;\tau),&\\
 &\partial^2_{v_i,v_j} \theta[\varepsilon](v) 
\equiv \frac{\partial^2}{\partial v_i\partial v_j}  \theta[\varepsilon](v;\tau),&\\
&\text{etc.}&
\end{align*}
The gradient of a theta function is denoted by $\partial_v \theta[\varepsilon](v)$.

The theta constant with characteristic $[\I_0]$ corresponding to a
partition $\I_0 \cup\J_0$ is denoted by $\theta[\I_0]$. 
The lowest non-vanishing derivative of $\theta[\I_\mFr](v;\tau)$ at $v=0$ 
is a tensor of order $\mFr$ denoted by $\partial_v^\mFr  \theta[\I_\mFr]$. Such a tensor 
consists of zero values of partial derivatives  $\partial_{v_{i_1}} \cdots \partial_{v_{i_\mFr}} \theta[\I_\mFr]$ 
of order $\mFr$ with respect to all combinations constructed from $g$ components of~$v$.
These zero values of derivatives are called  $\mFr$-th order theta derivatives, as mentioned above. 
 The following notation is used:
\begin{align*}
 &\theta[\I_0] \equiv \theta[\I_0](0;\tau),&\\
 &\partial_{v_i} \theta[\I_1] \equiv \frac{\partial}{\partial v_i}  \theta[\I_1](0;\tau),&\\
 &\partial^2_{v_i,v_j} \theta[\I_2] 
 \equiv \frac{\partial^2}{\partial v_i\partial v_j}  \theta[\I_2](0;\tau),&\\
&\text{etc.}&
\end{align*}
Another notation like $\theta^{\{i_1,i_2,i_3\}}$ is used for $\theta[\I]$ with $\I=\{i_1,i_2,i_3\}$
in the case when $\I$ is described by its elements.
Recall that theta constants are actually functions of normalised period matrix $\tau$ 
also called the Riemann period matrix.

Throughout the paper representation of characteristics in terms of partitions is used, 
because this makes clear which order of vanishing theta functions have at $v=0$.
Representation of characteristics in terms of matrices with entries $0$ and $1$, which is used in Appendices,
is derived from characteristics of branch points given in Subsection~\ref{ss:CharHyper}.

Let $\{e_i\mid i\in \mathcal{I}\}$ be a collection of branch points 
(here and below we use only $x$-coordinates to specify branch points)
corresponding to a partition $\I \cup \J$.
By $s_n(\I)$ an elementary symmetric polynomial of degree $n$ in  $\{e_i\mid i\in\I\}$
 is denoted, at that elementary symmetric polynomials $s_n$ are defined by
\begin{gather*}
\sum_{n\geqslant 0} t^n s_n = \prod_{i \in\I} (1+ e_i t).
\end{gather*}
And $\Delta(\I)$ denotes the Vandermonde determinant in  $\{e_i\mid i\in\I\}$, namely:
\begin{gather*}
 \Delta(\I) = \prod_{\substack{i > l \\ i,l \in \I}} (e_i-e_l).
\end{gather*}
By $\Delta$ the Vandermonde determinant over the set of all finite branch points of the curve
is denoted.

\subsection{First Thomae formula}
The first Thomae theorem and its corollaries below are given in the form proposed in \cite[p.\,1014]{ER2008}
\newtheorem*{FTteo}{First Thomae theorem}
\begin{FTteo}
 Let $\I_0\cup \J_0$ with $\I_0=\{i_1$, \ldots, $i_{g}\}$ 
 and $\J_0 = \{j_1$, \ldots, $j_{g+1}\}$  be a partition
 of the set $\{1,2,\dots,2g+1\}$ of indices of finite branch points
 of a hyperelliptic curve \eqref{Cg}, 
 and $[\I_0]$ denotes the non-singular even characteristic
 corresponding to $\mathcal{A} (\I_0) + K$. Then
 \begin{gather}\label{thomae1}
 \theta[\I_0] = \epsilon \bigg(\frac{\det \omega}{\pi^g}\bigg)^{1/2} \Delta(\I_0)^{1/4} \Delta(\J_0)^{1/4}.
\end{gather}
where $\epsilon$ satisfies $\epsilon^8=1$, then 
 $\Delta(\I_0)$ and $\Delta(\J_0)$ denote the Vandermonde determinants 
 built from $\{e_i\mid i\in \I_0\}$ and $\{e_j\mid j\in \J_0\}$. 
\end{FTteo}
For a proof see \cite[Proposition\;3.6, p.46]{fay973}.
\newtheorem{FTcor}{FTT Corollary}
\begin{FTcor}\label{C:eklm}
 Let $\I = \{i_1,\,\dots,\,i_{g-1}\}$ and $\J = \{j_1,\,\dots,\,j_{g-1}\}$ be two disjoint sets 
 picked out from  $2g+1$ indices of finite branch points of a hyperelliptic curve \eqref{Cg}, 
 let $k$, $m$, $n$ be the remaining indices.
 Then the following formula is valid
 \begin{gather}\label{eklm}
  \frac{e_{k}-e_{m}}{e_{k}-e_{n}}
  = \epsilon \frac{\theta[\{n\}\cup\I]^2 \theta[\{n\}\cup\J]^2}
  {\theta[\{m\}\cup\I]^2 \theta[\{m\}\cup\J]^2},
 \end{gather}
where $\epsilon^4=1$.
\end{FTcor}
\begin{FTcor}\label{C:eJI}
 Let $\I_0=\{i_1,\,\dots,\,i_{g}\}$ and $\J_0 = \{j_1,\,\dots,\,j_{g+1}\}$ form a partition of 
 $2g+1$  indices of finite branch points  of a hyperelliptic curve \eqref{Cg}.
 Choose $i_k$, $i_l \in \I_0$, and $j_n$, $j_m \in \J_0$. Then
 \begin{gather}\label{eJI}
  \frac{\prod_{j\in \J_0}(e_{i_k}-e_{j})}{(e_{i_k}-e_{i_l})^2\prod_{\substack{i\in \I_0 \\ i\neq i_k}}(e_{i_k}-e_{i})} 
  = \pm \frac{ \theta[\I_0^{(i_k\to j_n)}]^4 \theta[\I_0^{(i_k\to j_m)}]^4\theta[\J_0^{(j_n,j_m \to i_l)}]^4}
  {\theta[\I_0^{(i_k,i_l \to j_n,j_m)}\}]^4 \theta[\J_0^{(j_m)}]^4 \theta[\J_0^{(j_n)}]^4},
 \end{gather}
 where  $\J_0^{(j)}$ stands for  $\J_0 \backslash \{j\}$, 
 and $\I_0^{(i \to j)}$ denotes that index $i$ is replaced by $j$ in $\I_0$, 
the same refers to $\J_0^{(j\to l)}$.
\end{FTcor}
Note that the right hand side does not depend on the choice of $j_n$, $j_m \in \J_0$.

\subsection{Second Thomae formula}
The second Thomae theorem is also taken from \cite[p.\,1015]{ER2008}.
\newtheorem*{STteo}{Second Thomae theorem}
\begin{STteo}
 Let $\I_1\cup \J_1$ with $\I_1=\{i_1$, \ldots, $i_{g-1}\}$ and $\J_1 = \{j_1$, \ldots, $j_{g+2}\}$ 
 be a partition of the set $\{1,2,\dots,2g+1\}$ of indices of finite branch points of a hyperelliptic curve \eqref{Cg}, and
 $[\I_1]$ denote the non-singular odd characteristic  
 corresponding to $\mathcal{A} (\I_1) + K$. Then
 \begin{multline}\label{thomae2}
 \frac{\partial }{\partial v_n} \theta[\I_1](v)\big|_{v=0} \\
 = \epsilon \bigg(\frac{\det \omega}{\pi^g}\bigg)^{1/2} \Delta(\I_1)^{1/4} \Delta(\J_1)^{1/4} 
 \sum_{j=1}^g  (-1)^{j-1} s_{j-1}(\I_1) \omega_{j,n},
\end{multline}
where  $\epsilon$ satisfies $\epsilon^8=1$, then $\Delta(\I_1)$ and $\Delta(\J_1)$ denote 
the Vandermonde determinants 
 built from $\{e_i\mid i\in \I_1\}$ and $\{e_j\mid j\in \J_1\}$.
\end{STteo}
This result is nicely presented in the matrix form 
\begin{multline}\label{thomae2MF}
 \begin{pmatrix} \partial_{v_1} \\ \partial_{v_2}  \\ \vdots \\ \partial_{v_g}  \end{pmatrix} 
 \theta[\I_1] (v)\big|_{v=0} \\
 = \epsilon \bigg(\frac{\det \omega}{\pi^g}\bigg)^{1/2} \Delta(\I_1)^{1/4} \Delta(\J_1)^{1/4}
 \omega^t  \begin{pmatrix} s_0(\I_1) \\ - s_1(\I_1) \\ \vdots \\ (-1)^{g-1} s_{g-1}(\I_1) \end{pmatrix}.
\end{multline}

\subsection{General Thomae formula}
Proposed and proven in \cite{BerTF2020}.
\begin{teo}[The general Thomae theorem]\label{T:ThN}
 Let $\I_\mFr\cup \J_\mFr$ with $\I_\mFr = \{i_1$, \ldots, $i_{g+1-2\mFr}\}$ and 
 $\J_\mFr = \{j_1$, \ldots, $j_{g+1+2\mFr}\}$ 
 be a partition of the set $\{0,1,\dots,2g+1\}$ of indices of all branch points of a hyperelliptic curve \eqref{Cg}, and
 $[\I_\mFr]$ denotes the singular characteristic of multiplicity $\mFr$ 
 corresponding to $\mathcal{A} (\I_\mFr) + K$. 
  Then for any $n_1$, \ldots, $n_\mFr\in \{1,\dots,g\}$ and any set $\K \subset \J_\mFr$ of cardinality
 $\kFr = 2\mFr-1$ or $2\mFr$ the following relation holds
 \begin{multline}\label{thomaeN}
 \frac{\partial }{\partial v_{n_1}} \cdots 
 \frac{\partial }{\partial v_{n_\mFr}} \theta[\I_\mFr](v)\big|_{v=0} 
 = \epsilon \bigg(\frac{\det \omega}{\pi^g}\bigg)^{1/2} \Delta(\I_\mFr)^{1/4} \Delta(\J_\mFr)^{1/4} \times \\ \times
  \sum_{\substack{p_1,\dots,p_\mFr \in \K \\ \text{all different}}}
 \prod_{i=1}^\mFr \frac{\sum_{j=1}^g  (-1)^{j-1} s_{j-1}(\I_\mFr \cup \K^{(p_i)})\omega_{jn_i}}
 {\prod_{k\in\K \backslash\{p_1,\dots,p_\mFr \}} (e_{p_i} - e_k)},
\end{multline}
where $\K^{(p_i)}=\K\backslash \{p_i\}$,  $\epsilon$ satisfies $\epsilon^8=1$, then
$\Delta(\I_\mFr)$ and $\Delta(\J_\mFr)$ denote the Vandermonde determinants 
 built from $\{e_i\mid i\in \I_\mFr\}$ and $\{e_j\mid j\in \J_\mFr\}$, and $s_{j}(\I)$
denotes the elementary symmetric polynomial of degree $j$ in $\{e_i\mid i\in \I\}$
at that index $0$ of infinity is omitted when it occurs in sets.
The relation does not depend on the choice of $\K$.
\end{teo} 
\begin{lemma}\label{L:GenThomaeF}
Under the assumptions of the general Thomae theorem (Theorem~\ref{T:ThN}) with $\I_0=\I_\mFr \cup \K$ the following holds
\begin{multline}\label{thomaeNR}
 \frac{\partial_{v_{n_1}}\cdots \partial_{v_{n_\mFr}} \theta[\I_\mFr]}{\theta[\I_0]} 
 = \epsilon \prod_{\kappa \in \K} \frac{\big(\prod_{j\in\J_0}(e_\kappa - e_j)\big)^{1/4}}
 {\big(\prod_{\iota\in \I_\mFr} (e_\kappa-e_\iota)\big)^{1/4}} \times \\ \times
  \sum_{\substack{p_1,\dots,p_\mFr \in \K \\ \text{all different}}}
 \prod_{i=1}^\mFr \frac{\sum_{j=1}^g  (-1)^{j-1} s_{j-1}(\I_\mFr \cup \K^{(p_i)})\omega_{jn_i}}
 {\prod_{k\in\K \backslash\{p_1,\dots,p_\mFr \}} (e_{p_i} - e_k)}.
\end{multline}
\end{lemma}
\begin{proof}
Formula \eqref{thomaeNR} follows immediately from \eqref{thomaeN} 
and the first Thomae formula in the form
\begin{gather*}
 \bigg(\frac{\det \omega}{\pi^g}\bigg)^{1/2} \Delta(\I_\mFr)^{1/4} \Delta(\J_\mFr)^{1/4}  =
 \epsilon \prod_{\kappa \in \K} \frac{\big(\prod_{j\in\J_0}(e_\kappa - e_j)\big)^{1/4}}
 {\big(\prod_{\iota\in \I_\mFr} (e_\kappa-e_\iota)\big)^{1/4}} \theta[\I_0],
\end{gather*}
where $\I_\mFr = \I_0\backslash \K$ and $\J_\mFr = \J_0\cup \K$, and the following relations are applied
\begin{align*}
 &\Delta(\I_0) = \Delta(\K) \Delta(\I_\mFr) 
\prod_{\kappa \in \K} \prod_{\iota\in\I_\mFr}(e_\kappa - e_\iota),\\
 &\Delta(\J_\mFr) = \Delta(\K) \Delta(\J_0) 
\prod_{\kappa \in \K} \prod_{j \in\J_0}(e_\kappa - e_j).
\end{align*}
\end{proof}

Also recall the following Theorem proven in \cite{BerTF2020}.
\begin{teo}\label{T:DetD2theta}
For hyperelliptic curves of genera $g\geqslant 3$
rank of every matrix of second order theta derivativets equals three, that is 
\begin{gather*}
 \rank \big(\partial^2_v \theta[\I_2] \big) = 3.
\end{gather*}
Therefore,  $\det\big(\partial^2_v \theta[\I_2]\big)=0$ in genera $g>3$.
\end{teo}

\begin{rem}\label{R:eOrd}
 We assume that branch points in all factors $(e_i-e_l)$ are ordered in such a way that $i>l$, 
 we call this the \emph{normal order}.
 Such an order was introduced by Baker \cite[p.\,307, 346]{bak898}, in the case of real values of $e_k$ 
 all factors are guaranteed to be positive.
 This allows to avoid the multiplier $\epsilon$ which arises in Thomae formulas and their corollaries  since 
  $\epsilon$ is the same for all relations in the class of a fixed multiplicity.
 
The same is true for curves with branch points not necessary real, but arbitrary complex. Once a homology
basis as described in Subsection~\ref{ss:HypC} is chosen, the normal order of branch points is defined.
Note that  period matrices reflect this order. 

In fact, the first Thomae formula displays the same multiplier $\epsilon=1$ 
for all partitions $\I_0\cup \J_0$  if the arguments of $\Delta(\I_0)^{1/4}$ and $\Delta(\J_0)^{1/4}$
are computed as follows
\begin{equation}\label{argCalc}
\arg \Delta(\I)^{1/4} = \frac{1}{4} \sum_{\substack{i > l \\ i,l \in \I}} \arg (e_i-e_l),
\end{equation}
where $\I$ is $\I_0$ or $\J_0$, the normal order is applied, and the range of $\arg$ is
$(-\pi,\pi]$. The normal order leads to the same 
8-th root of unity  in  FTT Corollary~\ref{C:eklm} if one takes the square root from both sides of the equality,
and FTT Corollary~\ref{C:eJI} if one takes the 4th root from both sides of the equality.

To underline the fact that elements $\{e_k\}$ are normally ordered the operator of order $\Ord$ is introduced.
Then the first Thomae formula and FTT Corollaries 1 and 2 read as
 \begin{gather}\tag{\ref{thomae1}'}\label{thomae1Ord}
 \theta[\I_0] =  \bigg(\frac{\det \omega}{\pi^g}\bigg)^{1/2} 
 \Ord \Delta(\I_0)^{1/4} \Ord\Delta(\J_0)^{1/4},
\end{gather}
\begin{gather}\tag{\ref{eklm}'}\label{eklmOrd}
   \bigg(\Ord \frac{e_{k}-e_{m}}{e_{k}-e_{n}}\bigg)^{1/2}
  =  \frac{\theta[\{n\}\cup\I] \theta[\{n\}\cup\J]}
  {\theta[\{m\}\cup\I] \theta[\{m\}\cup\J]},
 \end{gather}
 \begin{gather}\tag{\ref{eJI}'}\label{eJIOrd}
   \bigg( \Ord \frac{\prod_{j\in \J_0}(e_{i_k}-e_{j})}
  {(e_{i_k}-e_{i_l})^2\prod_{\substack{i\in \I_0 \\ i\neq i_k}}(e_{i_k}-e_{i})} \bigg)^{1/4}
  =  \frac{ \theta[\I_0^{(i_k\to j_n)}] \theta[\I_0^{(i_k\to j_m)}] \theta[\J_0^{(j_n,j_m \to i_l)}]}
  {\theta[\I_0^{(i_k,i_l \to j_n,j_m)}\}] \theta[\J_0^{(j_m)}] \theta[\J_0^{(j_n)}]}.
 \end{gather}
 
The second Thomae formula also displays the same multiplier for all partitions $\I_1\cup \J_1$,
namely: $\epsilon=1$ if the genus $g$ of a curve is even, and 
$\epsilon=-1$ if $g$ is odd, at that
 expression \eqref{argCalc} is used for computing arguments of
$\Delta(\I_1)^{1/4}$ and $\Delta(\J_1)^{1/4}$. That is, the second Thomae formula acquires the form
\begin{multline}\tag{\ref{thomae2}'}\label{thomae2Ord}
 \frac{\partial }{\partial v_n} \theta[\I_1](v)\big|_{v=0} 
 = (-1)^g \bigg(\frac{\det \omega}{\pi^g}\bigg)^{1/2} \Ord \Delta(\I_1)^{1/4} \Ord \Delta(\J_1)^{1/4} 
 \times \\ \times
 \sum_{j=1}^g  (-1)^{j-1} s_{j-1}(\I_1) \omega_{j,n}.
\end{multline}

The general Thomae formula acquires the form  
\begin{multline}\tag{\ref{thomaeN}'}\label{thomaeNOrd} 
 \frac{\partial }{\partial v_{n_1}} \cdots 
 \frac{\partial }{\partial v_{n_\mFr}} \theta[\I_\mFr](v)\big|_{v=0} 
 = \epsilon \bigg(\frac{\det \omega}{\pi^g}\bigg)^{1/2}  \Ord \Delta(\I_1)^{1/4} \Ord \Delta(\J_1)^{1/4} 
 \times \\ \times
  \sum_{\substack{p_1,\dots,p_\mFr \in \K \\ \text{all different}}}
 \prod_{i=1}^\mFr \frac{\sum_{j=1}^g  (-1)^{j-1} s_{j-1}(\I_\mFr \cup \K^{(p_i)})\omega_{jn_i}}
 {\prod_{k\in\K \backslash\{p_1,\dots,p_\mFr \}} (e_{p_i} - e_k)}.
\end{multline}
Note that $\Ord$ is not applied to factors $(e_{p_i} - e_k)$ in the denominator under the product. 
The general Thomae formula displays the same multiplier $\epsilon$ in each class of relations 
with a fixed multiplicity. Actually, $\epsilon=-1$ at multiplicity $\mFr=2$, and $\epsilon=-(-1)^g$
at $\mFr=3$. As a conjecture, $\epsilon=(-1)^{g \mFr + [\mFr/2]}$ at multiplicity $\mFr$.
\end{rem}

\subsection{Verification}
Formulas \eqref{thomae1Ord}--\eqref{thomae2Ord} as well as
all the relations presented below 
had been verified by direct computation of the left and right hand sides
for all possible partitions.
Curves with real and complex branch points had been used.
Period matrices $\omega$, $\omega'$ had been computed explicitly, as well as matrix $\tau$
for each curve. Hyperelliptic curves of genera $3$, $4$, $5$ and $6$
(four curves with different sets of branch points in each genus)  had been taken.

\section{Generalization of Jacobi's formula}\label{s:ThetaRels}
This section consists of three subsections, corresponding to the cases of multiplicities $1$, $2$ and $3$.
In the case of multiplicity $1$ relations between gradient vectors $\partial_v \theta[\I_1]$ 
with different characteristics are found (Proposition~\ref{P:thomaeK2}), 
and the question about rank of a collection of such vectors is elucidated
(Propositions~\ref{P:D1thetaLD3}, \ref{P:D1thetaLD4}, Conjecture~\ref{C:D1thetaLDH}, and Theorem~\ref{T:D1thetaLD}).
Then a representation of second order theta derivatives $\partial^2_v \theta[\I_2]$ (the Hesse matrix)
in terms of theta constants and first order theta derivatives is proposed (Theorem~\ref{T:D2thetaGradRepr}).
A similar result is obtained in the case of multiplicity~$3$ (Theorem~\ref{T:D3thetaGradRepr}),
and can be generalised to higher multiplicities (Conjecture~\ref{C:DNthetaGradRepr}).

\subsection{First order theta derivatives}\label{ss:FDT}
\begin{prop}\label{P:thomaeK2}
 Let $\I_0\cup \J_0$ with $\I_0 =\{i_1$, \ldots, $i_{g}\}$ and
 $\J_0 = \{j_1$, \ldots, $j_{g+1}\}$ 
 be a partition of the set $\{1$, $2$, $\dots$, $2g+1\}$ of indices of finite branch points of a genus $g$ hyperelliptic curve.  
 Then with arbitrary $\kappa_1$, $\kappa_2 \in \I_0$, $\kappa_1<\kappa_2$ and any $j_m$, $j_k\in\J_0$
 \begin{multline}\label{ThGInf}
 \partial_{v_n} \theta[\I_0 \backslash \{\kappa_1,\kappa_2\}] 
 = \big(\theta[\I_0^{(\kappa_1,\kappa_2 \to j_m,j_k)}]\theta[\J_0^{(j_m)}]\theta[\J_0^{(j_k)}]\big)^{-1} \times \\
 \times \Big( \theta[\I_0^{(\kappa_1\to j_m)}]\theta[\I_0^{(\kappa_1\to j_k)}]\theta[\J_0^{(j_m,j_k \to \kappa_2)}]
 \partial_{v_n} \theta[\I_0^{(\kappa_2)}] \\
 - \theta[\I_0^{(\kappa_2\to j_m)}]\theta[\I_0^{(\kappa_2\to j_k)}]\theta[\J_0^{(j_m,j_k \to \kappa_1)}]
 \partial_{v_n} \theta[\I_0^{(\kappa_1)}] \Big),
\end{multline}
where $\I_0^{(\kappa_1,\kappa_2 \to j_m,j_k)}$ denotes the set $\I_0$ with indices $\kappa_1$, $\kappa_2$
replaced by $j_m$, $j_k$ etc.
The equality does not depend on the choice of $j_m$, $j_k$.
\end{prop}
\begin{proof}
Formula \eqref{thomaeNR} from Lemma~\ref{L:GenThomaeF}  with $\mFr=1$ and $\card \K =2$ gives
\begin{multline}\label{thomaeK2R}
\frac{\partial_{v_{n}} \theta [\I_0 \backslash \{\kappa_1,\kappa_2\}]}  {\theta[\I_0]}  
  = \epsilon \frac{\big(\prod_{j\in \J_0} (e_{\kappa_1}-e_j)(e_{\kappa_2}-e_j) \big)^{1/4}}
 {\big( \prod_{\iota \in \I_0 \backslash \{\kappa_1,\kappa_2\}} 
 (e_{\kappa_1} - e_{\iota})(e_{\kappa_2} - e_{\iota}) \big)^{1/4}} \times \\ \times
 \bigg(\frac{\sum_{j=1}^g(-1)^{j-1} s_{j-1}\big(\I_0^{(\kappa_1)}\big) \omega_{jn}}{(e_{\kappa_1} - e_{\kappa_2})}
  + \frac{\sum_{j=1}^g(-1)^{j-1} s_{j-1}\big(\I_0^{(\kappa_2)}\big) \omega_{jn}}{(e_{\kappa_2} - e_{\kappa_1})} \bigg).
\end{multline}
Then second Thomae theorem in the form
\begin{gather}\label{Th2Inv}
 \sum_{j=1}^g  (-1)^{j-1} s_{j-1}(\I_0^{(p)}) \omega_{j,n} = \widetilde{\epsilon} 
 \frac{\big(\prod_{\iota \in \I_0^{(p)}} (e_p - e_{\iota}) \big)^{1/4}}{\big(\prod_{j \in \J_0} (e_p - e_j)\big)^{1/4}}
 \frac{\partial_{v_n} \theta[\I_0^{(p)}]}{\theta[\I_0]}
\end{gather}
is applied to the right hand side of \eqref{thomaeK2R}, so
\begin{multline*}
 \partial_{v_{n}} \theta \big[\I_0 \backslash \{\kappa_1,\kappa_2\}\big]  
  =  \Ord \frac{\big(\prod_{j\in \J_0} (e_{\kappa_2}-e_j) \big)^{1/4}
  \partial_{v_n} \theta[\I_0^{(\kappa_1)}]}
 {\big(-(e_{\kappa_1} - e_{\kappa_2})^2 \prod_{\iota \in \I_{0}^{(\kappa_2)}} (e_{\kappa_2} - e_{\iota}) \big)^{1/4}} \\
   - \Ord \frac{\big(\prod_{j\in \J_0} (e_{\kappa_1}-e_j) \big)^{1/4}
   \partial_{v_n} \theta[\I_0^{(\kappa_2)}]}
 {\big(-(e_{\kappa_2} - e_{\kappa_1})^2 \prod_{\iota \in \I_{0}^{(\kappa_1)}} (e_{\kappa_1} - e_{\iota}) \big)^{1/4}}.
\end{multline*}
Recall that $\widetilde{\epsilon}$ for all characteristics  of multiplicity $1$ is the same if branch points 
in all factors $(e_i-e_l)^{1/4}$ are ordered normally, that is $i>l$.
Using FTT Corollary~\ref{C:eJI} with the normal order in all factors $(e_i-e_l)$, relation \eqref{ThGInf} is obtained. 
At that relation \eqref{ThGInf} does not depend on the choice of $j_m$, $j_k$ due to FTT Corollary~\ref{C:eklm}.
\end{proof}

\begin{exam}\label{E:ThetaRelG2}
In genus $2$ case let  $\I_0=\{\kappa_1$, $\kappa_2\}$, 
$\kappa_1<\kappa_2$, and $\J_0 = \{j_1$, $j_2$, $j_3\}$ 
form a partition of the set $\{1,2,3,4,5\}$ of finite branch points.
Then \eqref{ThGInf} reads as 
 \begin{multline}\label{ThetaRelG2}
 \partial_v \theta^{\emptyset}  =  \big(\theta^{\{j_1,j_2\}}\theta^{\{j_1,j_3\}}\theta^{\{j_2,j_3\}}\big)^{-1} 
 \Big( \theta^{\{\kappa_2,j_1\}} \theta^{\{\kappa_2,j_2\}} \theta^{\{\kappa_2,j_3\}} \partial_v \theta^{\{\kappa_1\}} \\
- \theta^{\{\kappa_1,j_1\}} \theta^{\{\kappa_1,j_2\}} \theta^{\{\kappa_1,j_3\}} \partial_v \theta^{\{\kappa_2\}} \Big),
\end{multline}
where $\theta^{\{i,j\}}$ denotes the theta constant with characteristic corresponding to point $\mathcal{A} (\{i,j\}) + K$.
There exist $10$ partitions $\I_0\cup \J_0$ of this type and $10$ right-hand side expressions 
all equal to $\partial_v \theta^{\emptyset}$.
Some relations are given in Appendix~\ref{A:G2ThetaRelK2}, where  characteristics of theta functions
are presented in two forms: in terms of partitions and in the standard form.
\end{exam}

\begin{rem}
 Relations \eqref{ThetaRelG2} are mentioned in \cite[p.\,1021]{ER2008} as
 `derived from addition theorems (e.g. Baker 1897, p. 342)'. 
 Here an alternative proof is proposed.
\end{rem}

\begin{exam}\label{E:ThetaRelG3}
In genus $3$ case let $\I_0=\{\iota$, $\kappa_1$, $\kappa_2\}$ with $\kappa_1<\kappa_2$.
For each $\iota$ one obtains $\binom{2g}{2}=15$ relations of the form
 \begin{multline}\label{ThGInfG3}
 \partial_v \theta^{\{\iota\}}
 = \big(\theta^{\{\iota,j_1,j_2\}}\theta^{\{j_2,j_3,j_4\}}\theta^{\{j_1,j_3,j_4\}}\big)^{-1} \times \\ \times 
 \Big(\theta^{\{\iota,\kappa_2,j_1\}}\theta^{\{\iota,\kappa_2,j_2\}}\theta^{\{j_3,j_4,\kappa_2)\}}
 \partial_v \theta^{\{\iota,\kappa_1\}} \\
 - \theta^{\{\iota,\kappa_1,j_1\}}\theta^{\{\iota,\kappa_1,j_2\}}\theta^{\{j_3,j_4,\kappa_1)\}}
 \partial_v \theta^{\{\iota,\kappa_2\}} \Big).
\end{multline}
Some relations with $\partial_v \theta^{\{1\}}$ are given in Appendix~\ref{A:G2ThetaRelK2},
where characteristics of theta functions
are presented both in terms of partitions and in the standard form.
\end{exam}

Proposition~\ref{P:thomaeK2} implies an obvious observation: every vector
$\partial_v \theta[\I_0\backslash\{\kappa_1,\kappa_2\}]$ is a linear combination of 
$\partial_v \theta[\I_0\backslash\{\kappa_1\}]$ and $\partial_v \theta[\I_0\backslash\{\kappa_2\}]$.
Note that the intersection of $\I_0\backslash\{\kappa_1,\kappa_2\}$, $\I_0\backslash\{\kappa_1\}$,
and $\I_0\backslash\{\kappa_2\}$ has cardinality $g-2$. Presumably, 
every three $g$-component vectors  $\partial_v \theta[\B_1]$, $\partial_v \theta[\B_2]$,
$\partial_v \theta[\B_3]$  such that  $\B_1 \cap \B_2 \cap \B_3$ is of cardinality $g-2$ 
contain only two linearly independent vectors.

\begin{prop}\label{P:D1thetaLD3}
Let $\I \cup \{\kappa_1,\kappa_2,\kappa_3\} \cup \J$ with $\I=\{i_1$, $\dots$, $i_{g-2}\}$ 
and $\J=\{j_1$, $\dots$, $j_{g}\}$
be a partition of the set $\{1$, $2$, $\dots$, $2g+1\}$ of indices of finite branch points 
of a genus $g$ hyperelliptic curve, and
$\kappa_1 < \kappa_2 < \kappa_3$. Then theta derivative vectors $\partial_{v} \theta[\I\cup \{\kappa_1\}]$,
$\partial_{v} \theta[\I\cup \{\kappa_2\}]$, and $\partial_{v} \theta[\I\cup \{\kappa_3\}]$ are linearly dependent,
and the following relation between these three holds
 \begin{multline}\label{D1thetaLD}
  \theta[\J^{(j_n\to \kappa_1)}]\theta[\J^{(j_m\to \kappa_1)}]\theta[\J^{(j_m,j_n \to \kappa_2,\kappa_3)}]
  \partial_{v} \theta[\I\cup \{\kappa_1\}] \\
 - \theta[\J^{(j_n\to \kappa_2)}]\theta[\J^{(j_m\to \kappa_2)}]\theta[\J^{(j_m,j_n \to \kappa_1,\kappa_3)}]
 \partial_{v} \theta[\I\cup \{\kappa_2\}] \\
 + \theta[\J^{(j_n\to \kappa_3)}]\theta[\J^{(j_m\to \kappa_3)}]\theta[\J^{(j_m,j_n \to \kappa_1,\kappa_2)}]
  \partial_{v} \theta[\I\cup \{\kappa_3\}] = 0.
\end{multline}
\end{prop}
Proof is given in Appendix~\ref{A:D1thetaLD3}.

In the context of Proposition~\ref{P:D1thetaLD3} with a partition 
$\I \cup \{\kappa_1$, $\kappa_2$, $0\} \cup \J$, where $0$ is the index of infinity, 
$\I =\{i_1,\dots,i_{g-2}\}$ and $\J =\{j_1,\dots,j_{g+1}\}$, relation \eqref{ThGInf} acquires the form
 \begin{multline}\label{D1thetaLD2P}
  \theta[\J_0^{(j_n \to 0)}]\theta[\J_0^{(j_m \to 0)}]\theta[\J_0^{(j_m,j_n \to \kappa_1,\kappa_2)}]
  \partial_{v} \theta[\I\cup \{0\}] \\
 - \theta[\J_0^{(j_n\to \kappa_1)}]\theta[\J_0^{(j_m\to \kappa_1)}]\theta[\J_0^{(j_m,j_n \to \kappa_2, 0)}]
 \partial_{v} \theta[\I\cup \{\kappa_1\}] \\
 + \theta[\J_0^{(j_n\to \kappa_2)}]\theta[\J_0^{(j_m\to \kappa_2)}]\theta[\J_0^{(j_m,j_n \to \kappa_1, 0)}]
  \partial_{v} \theta[\I\cup \{\kappa_2\}] = 0.
\end{multline}
One should take into account that complement partitions produce the same characteristic, 
for example $[\I\cup\{\kappa_2,j,0\}] = [\J_0^{(j\to \kappa_1)}]$.

Assertions of Proposition~\ref{P:thomaeK2} and Proposition~\ref{P:D1thetaLD3} are combined into 
\newtheorem*{prop2}{Proposition \ref{P:D1thetaLD3}'}
\begin{prop2}
Let $\I \cup \{\kappa_1$, $\kappa_2$, $\kappa_3\} \cup \J$ with $\I=\{i_1$, $\dots$, $i_{g-2}\}$ 
and $\J=\{j_1$, $\dots$, $j_{g+1}\}$
be a partition of the set $\{0$, $1$, $\dots$, $2g+1\}$ of indices of all branch points 
of a genus $g\geqslant 2$ hyperelliptic curve, and
$\kappa_1 < \kappa_2 < \kappa_3$, taking into account that the index of infinity is $0$. 
Then theta derivative vectors $\partial_{v} \theta[\I\cup \{\kappa_1\}]$,
$\partial_{v} \theta[\I\cup \{\kappa_2\}]$, and $\partial_{v} \theta[\I\cup \{\kappa_3\}]$ are linearly dependent,
and the following relation between these three holds
 \begin{multline}\label{D1thetaLDG}
  \theta[\J^{(j_n\to \kappa_1)}]\theta[\J^{(j_m\to \kappa_1)}]\theta[\J^{(j_m,j_n \to \kappa_2,\kappa_3)}]
  \partial_{v} \theta[\I\cup \{\kappa_1\}] \\
 - \theta[\J^{(j_n\to \kappa_2)}]\theta[\J^{(j_m\to \kappa_2)}]\theta[\J^{(j_m,j_n \to \kappa_1,\kappa_3)}]
 \partial_{v} \theta[\I\cup \{\kappa_2\}] \\
 + \theta[\J^{(j_n\to \kappa_3)}]\theta[\J^{(j_m\to \kappa_3)}]\theta[\J^{(j_m,j_n \to \kappa_1,\kappa_2)}]
  \partial_{v} \theta[\I\cup \{\kappa_3\}] = 0,
\end{multline}
where $\J^{(j_m,j_n \to \kappa_1,\kappa_2)}$ denotes the set $\J$ with indices $j_m$, $j_n$
replaced by $\kappa_1$, $\kappa_2$ etc.
\end{prop2}

Introduce the \emph{co-lexicographic order of sets} of indices. Only sets of equal cardinality are compared 
(the index of infinity should not be omitted). Two sets are compared by highest indices, 
if the highest indices are equal then a smaller ones should be compared, and so on. 

The result of Proposition~\ref{P:D1thetaLD3}' can be extended to the case of three
linearly independent vectors of theta derivatives.

\begin{prop}\label{P:D1thetaLD4}
Let $\I\cup \{\kappa_1$, $\kappa_2$, $\kappa_3$, $\kappa_4$, $\kappa_5\} \cup \J$ with 
$\I =\{i_1$, $\dots$, $i_{g-3}\}$ and $\J=\{j_1$, $\dots$, $j_{g}\}$
be a partition of the set $\{0$, $1$, $\dots$, $2g+1\}$ of indices of all branch points of 
a genus $g\geqslant 3$ hyperelliptic curve.
Then any collection of three theta derivative vectors of the form 
$\partial_{v} \theta[\I\cup \{\kappa_1,\kappa_2\}]$,
$\partial_{v} \theta[\I\cup \{\kappa_1,\kappa_3\}]$, and $\partial_{v} \theta[\I\cup \{\kappa_2,\kappa_3\}]$
is linearly independent. And vector $\partial_{v} \theta[\I\cup \{\kappa_4,\kappa_5\}]$ is spanned by 
these three, at that the following relation holds
\begin{multline}\label{D1ThetaLD4a}
 \theta[\J^{(j_n\to \kappa_1,\kappa_2)}]\theta[\J^{(j_m\to \kappa_1,\kappa_2)}]
 \theta[\J^{(j_m,j_n \to \kappa_3,\kappa_4,\kappa_5)}]
  \partial_{v} \theta[\I\cup \{\kappa_1,\kappa_2\}] \\
 - \theta[\J^{(j_n\to \kappa_1,\kappa_3)}]\theta[\J^{(j_m\to \kappa_1,\kappa_3)}]
 \theta[\J^{(j_m,j_n \to \kappa_2,\kappa_4,\kappa_5)}]
 \partial_{v} \theta[\I\cup \{\kappa_1,\kappa_3\}] \\
 + \theta[\J^{(j_n\to \kappa_2,\kappa_3)}]\theta[\J^{(j_m\to \kappa_2,\kappa_3)}] 
 \theta[\J^{(j_m,j_n \to \kappa_1,\kappa_4,\kappa_5)}]
 \partial_{v} \theta[\I\cup \{\kappa_2,\kappa_3\}] \\
 - \theta[\J^{(j_n\to \kappa_4,\kappa_5)}]\theta[\J^{(j_m\to \kappa_4,\kappa_5)}] 
 \theta[\J^{(j_m,j_n \to \kappa_1,\kappa_2,\kappa_3)}]
  \partial_{v} \theta[\I\cup \{\kappa_4,\kappa_5\}] = 0.
\end{multline}
where $\J^{(j_m,j_n \to \kappa_1,\kappa_2,\kappa_3)}$ denotes the set $\J$ with indices $j_m$, $j_n$
replaced by $\kappa_1$, $\kappa_2$, $\kappa_3$ etc.
Every set $\I\cup \{\kappa_i,\kappa_k\}$ is arranged in the ascending order, and
$\kappa_1 < \kappa_2 < \kappa_3 < \kappa_4 < \kappa_5$ is supposed, 
the index of infinity is $0$. 
\end{prop}
\emph{A brief proof}.
According to Proposition~\ref{P:D1thetaLD3}'
vector $\partial_v \theta[\I\cup\{\kappa_4,\kappa_5\}]$  is a linear combination of
$\partial_v \theta[\I\cup\{\kappa_4,\kappa_1\}]$ and $\partial_v \theta[\I\cup\{\kappa_4,\kappa_2\}]$. 
Next, $\partial_v \theta[\I\cup\{\kappa_4,\kappa_1\}]$ is a linear combination of 
$\partial_v \theta[\I\cup\{\kappa_1,\kappa_2\}]$  and  $\partial_v \theta[\I\cup\{\kappa_1,\kappa_3\}]$,
and $\partial_v \theta[\I\cup\{\kappa_4,\kappa_2\}]$ is a linear combination of 
$\partial_v \theta[\I\cup\{\kappa_1,\kappa_2\}]$  and  $\partial_v \theta[\I\cup\{\kappa_2,\kappa_3\}]$.
This leads to three relations of the form \eqref{D1thetaLDG} which reduce to \eqref{D1ThetaLD4a}.
A more detailed proof is given in Appendix~\ref{A:D1thetaLD4}.

Therefore, a collection of vectors of the form $\partial_v \theta[\I\cup\{\kappa_i,\kappa_j\}]$ with $\card \I = g-3$
and $i$, $j\in\{1,2,3,4,5\}$ contains three linear independent vectors. One can choose 
$\partial_v \theta[\I\cup\{\kappa_1,\kappa_2\}]$, $\partial_v \theta[\I\cup\{\kappa_1,\kappa_3\}]$, 
 and  $\partial_v \theta[\I\cup\{\kappa_2,\kappa_3\}]$ as a basis. Then $\partial_v \theta[\I\cup\{\kappa_4,\kappa_5\}]$
is spanned by all three vectors due to the cardinality of the intersection of the four sets: $\I\cup\{\kappa_1,\kappa_2\}$,
$\I\cup\{\kappa_1,\kappa_3\}$, $\I\cup\{\kappa_2,\kappa_3\}$ and $\I\cup\{\kappa_4,\kappa_5\}$ is $g-3$. 
Any other vector from the collection  is spanned by only two basis vectors due to 
the intersection of the sets defining characteristics has cardinality $g-2$.
Indeed, if $\{\kappa_i,\kappa_j\}$ is not $\{\kappa_4,\kappa_5\}$ then one of indices $i$, $j$ belongs to $\{1$, $2$, $3\}$.
For example, vector $\partial_v \theta[\I\cup\{\kappa_1,\kappa_4\}]$ is linearly dependent on 
$\I\cup\{\kappa_1,\kappa_2\}$, $\I\cup\{\kappa_1,\kappa_3\}$,
according to Proposition~\ref{P:D1thetaLD3}', since the intersection of sets
$\I\cup\{\kappa_1,\kappa_2\}$,
$\I\cup\{\kappa_1,\kappa_3\}$ and $\I\cup\{\kappa_1,\kappa_4\}$ has cardinality $g-2$.

One can choose another linearly independent set, say
$\partial_{v} \theta[\I\cup \{\kappa_2,\kappa_3\}]$,
$\partial_{v} \theta[\I\cup \{\kappa_2,\kappa_5\}]$, and $\partial_{v} \theta[\I\cup \{\kappa_3,\kappa_5\}]$.
Then the vector spanned by all these three vectors is $\partial_{v} \theta[\I\cup \{\kappa_1,\kappa_4\}]$, and
the following relation holds
\begin{multline}\label{D1ThetaLD4b}
 \theta[\J^{(j_n\to \kappa_2,\kappa_3)}]\theta[\J^{(j_m\to \kappa_2,\kappa_3)}]
 \theta[\J^{(j_m,j_n \to \kappa_1,\kappa_4,\kappa_5)}]
  \partial_{v} \theta[\I\cup \{\kappa_2,\kappa_3\}] \\
 - \theta[\J^{(j_n\to \kappa_1,\kappa_4)}]\theta[\J^{(j_m\to \kappa_1,\kappa_4)}]
 \theta[\J^{(j_m,j_n \to \kappa_2,\kappa_3,\kappa_5)}]
 \partial_{v} \theta[\I\cup \{\kappa_1,\kappa_4\}] \\
 + \theta[\J^{(j_n\to \kappa_2,\kappa_5)}]\theta[\J^{(j_m\to \kappa_2,\kappa_5)}] 
 \theta[\J^{(j_m,j_n \to \kappa_1,\kappa_3,\kappa_4)}]
 \partial_{v} \theta[\I\cup \{\kappa_2,\kappa_5\}] \\
 - \theta[\J^{(j_n\to \kappa_3,\kappa_5)}]\theta[\J^{(j_m\to \kappa_3,\kappa_5)}] 
 \theta[\J^{(j_m,j_n \to \kappa_1,\kappa_2,\kappa_4)}]
  \partial_{v} \theta[\I\cup \{\kappa_3,\kappa_5\}] = 0.
\end{multline}
\begin{rem}
Note that summands in relations \eqref{D1thetaLD}--\eqref{D1ThetaLD4b} 
have alternating signs when all vectors $\partial_{v}\theta[\I\cup \{\kappa_i,\kappa_k\}]$ 
are arranged in the ascending order of sets $\I\cup \{\kappa_i,\kappa_k\}$.
We suppose that  $\kappa_1 < \kappa_2 < \kappa_3 < \kappa_4 < \kappa_5$.
\end{rem}

The above is generalised in the following
\begin{teo}\label{T:D1thetaLD}
Let $\{\partial_v \theta[\B_n]\}$ be a collection of $g$-component vectors,
where $\B_n$ denotes a subset of the set $\{0$, $1$, $\dots$, $2g+1\}$ of indices
of all branch points of a genus~$g$ hyperelliptic curve,
such that $[\B_n]$ is a characteristic of multiplicity $1$. 

If $\cap_{n} \B_n$ has cardinality $g-\rFr$, then the collection of vectors has rank not greater than $\rFr$. 
The rank equals $\rFr$ if and only if for every $\nFr$ from $\rFr$ to $2$ 
there exists at least one subcollection $\{\partial_v \theta[\B_{k_1}]$, \ldots,
$\partial_v \theta[\B_{k_{\nFr+1}}]\}$ of $\nFr+1$ vectors
such that $\cap_{i=1}^{\nFr+1} \B_{k_i}$ is of cardinality $g-\nFr$.
\end{teo}
\begin{proof}
As stated in Propositions~\ref{P:D1thetaLD3}', a collection of $g$-component vectors
$\{\partial_v \theta[\B_n]\}$ with $\cap_{n} \B_n =  \B^{[g-2]}$, 
where $\B^{[g-2]}$ denotes a set of cardinality $g-2$,
contains at most two linearly independent vectors. 
Every $\B_n$ in the collection contains $\B^{[g-2]}$, and is obtained from $\B^{[g-2]}$ 
by joining one index (due to multiplicity $1$ of the corresponding characteristic, 
see Subsection \ref{ss:CharPart} for a detailed explanation). 
Any pair of vectors $\partial_v \theta[\B_{n_1}]$ and $\partial_v \theta[\B_{n_2}]$
such that $\B_{n_1}=\B^{[g-2]}\cup\{\kappa_1\}$ and $\B_{n_2}=\B^{[g-2]}\cup\{\kappa_2\}$
with $\kappa_1$, $\kappa_2 \not\in \B^{[g-2]}$ serves as a spanning set (or a basis) of the collection,
and any other theta derivative $\partial_v \theta [\B_n]$  from the collection 
is a linear combination of these two as given by \eqref{D1thetaLDG}.

Let $\{\partial_v \theta[\B_n]\}$ be a collection of $g$-component vectors such that
$\cap_{n} \B_n$ equals a set $\B^{[g-3]}$ of cardinality $g\,{-}\,3$. 
Then every $\B_n$ is obtained from $\B^{[g-3]}$ by joining two indices. 
As stated in Proposition~\ref{P:D1thetaLD4},
vectors $\partial_v \theta[\B_{n_1}]$, $\partial_v \theta[\B_{n_2}]$ and
$\partial_v \theta[\B_{n_3}]$ such that $\B_{n_1}=\B^{[g-3]}\cup\{\kappa_1,\kappa_2\}$, 
$\B_{n_2}=\B^{[g-3]}\cup\{\kappa_1,\kappa_3\}$
and $\B_{n_3}=\B^{[g-3]}\cup\{\kappa_2,\kappa_3\}$ with $\kappa_1$, $\kappa_2$, $\kappa_3 \not\in \B^{[g-3]}$
form a linearly independent set.
So the collection containing such three vectors has rank $3$, 
and all other vectors of the collection are linear combinations of these three as given by \eqref{D1ThetaLD4a}. 
Note that a collection of rank $2$ (if two spanning vectors are sufficient)
necessarily has $\cap_{n} \B_n = \B^{[g-2]}$, where $\B ^{[g-2]}$ is of cardinality $g\,{-}\,2$.

Next, consider a collection $\{\partial_v \theta[\B_n]\}$ of $g$-component vectors with 
$\cap_{n} \B_n$ equal to a set $\B^{[g-4]}$ of cardinality $g\,{-}\,4$. Then
every $\B_n$ is obtained from $\B^{[g-4]}$ by joining three indices,
let $\K_{i,j,k}=\{\kappa_i,\kappa_j,\kappa_k\}$ denote a set of indices joined to~$\B^{[g-4]}$,
that is $\B_n =\B^{[g-4]}\cup\K_{i,j,k}$.
Among all vectors of the collection at most four constitute a spanning set, 
for example with $\K_{1,2,3}$, $\K_{1,2,4}$, $\K_{1,3,4}$, $\K_{2,3,4}$.
Indeed, by Proposition~\ref{P:D1thetaLD3}' a vector $\partial_v \theta[\B^{[g-4]}\cup\K_{i,j,k}]$ 
is spanned by  two: $\partial_v \theta[\B^{[g-4]}\cup\K_{1,j,k}]$ and $\partial_v \theta[\B^{[g-4]}\cup\K_{2,j,k}]$.
Then the former is spanned by $\partial_v \theta[\B^{[g-4]}\cup\K_{1,2,k}]$ and
$\partial_v \theta[\B^{[g-4]}\cup\K_{1,3,k}]$,
and the latter by $\partial_v \theta[\B^{[g-4]}\cup\K_{1,2,k}]$
and $\partial_v \theta[\B^{[g-4]}\cup\K_{2,3,k}]$.
By Proposition~\ref{P:D1thetaLD4} three vectors $\partial_v \theta[\B^{[g-4]}\cup\K_{1,2,k}]$,
$\partial_v \theta[\B^{[g-4]}\cup\K_{1,3,k}]$ and $\partial_v \theta[\B^{[g-4]}\cup\K_{2,3,k}]$
form a spanning set for any vector $\partial_v \theta[\B^{[g-4]}\cup\K_{l,m,k}]$ 
with $l$, $m \not\in \{1$, $2$, $3\}$. Similarly to the proof of  Proposition~\ref{P:D1thetaLD4}
one can prove that vectors $\partial_v \theta[\B^{[g-4]}\cup\K_{1,2,3}]$,
$\partial_v \theta[\B^{[g-4]}\cup\K_{1,2,4}]$, $\partial_v \theta[\B^{[g-4]}\cup\K_{1,3,4}]$, 
$\partial_v \theta[\B^{[g-4]}\cup\K_{2,3,4}]$ form a spanning set for  any
$\partial_v \theta[\B^{[g-4]}\cup\K_{i,j,k}]$ with $i$, $j$, $k \not\in \{1$, $2$, $3$, $4\}$. 

Note that a collection $\{\partial_v \theta[\B_n]\}$ with an intersection $\cap_{n} \B_n$ 
of cardinality $g-4$ does not necessarily contain 
four linearly independent vectors. Suppose that all vectors but one  of the collection are spanned by two vectors, 
say $\partial_v \theta[\B^{[g-4]}\cup\K_{1,2,3}]$ and $\partial_v \theta[\B^{[g-4]}\cup\K_{1,2,4}]$,
that is the intersection of the corresponding partitions is $\B^{[g-4]}\cup \{\kappa_1,\kappa_2\}$ of cardinality $g-2$. 
These two vectors with the remaining one form a basis. At the same time, 
by the assumption the intersection of all partitions in the collection is of cardinality $g-4$.
Thus, the collection has only three spanning vectors, so the rank is~$3$ (less than $4$). In this case 
there is no subcollection with an intersection of cardinality $g-3$. 

Now, let $\{\partial_v \theta[\B_n]\}$ be a collection of $g$-component vectors such that
$\cap_{n} \B_n =\B^{[g-\rFr]}$ is of cardinality $g\,{-}\,\rFr$. 
Every $\B_n$ is obtained from $\B^{[g-\rFr]}$ by joining $\rFr-1$ indices,
due to multiplicity~$1$ of characteristic $[\B_n]$. By induction
one can prove that  a maximal spanning set can be composed from $\rFr$ vectors 
of the form $\partial_{v} \theta[\B^{[g-\rFr]}\cup \K^{(\kappa)}]$,
where $\K$ is a set of $\rFr$ indices, say $\{\kappa_1,\kappa_2,\dots,\kappa_\rFr\}$,
and $\K^{(\kappa)}=\K\backslash \{\kappa\}$.
However, collections with smaller number of spanning vectors also exist.
\end{proof}

\begin{conj}\label{C:D1thetaLDH}
Let $\I \cup \B \cup \J$ with 
$\I =\{i_1$, $\dots$, $i_{g-\rFr}\}$, $\B = \{\kappa_1$, $\dots$, $\kappa_{2\rFr-1}\}$,
and $\J=\{j_1$, $\dots$, $j_{g+3-\rFr}\}$
be a partition of the set $\{0,$ $1$, $\dots$, $2g+1\}$ of indices of branch points
 of a genus $g$ hyperelliptic curve.
Let $\K = \{\kappa_1,\kappa_2,\dots,\kappa_\rFr\} \subset \B$, 
and $\K^{(\kappa)}$ denote $\K \backslash \{\kappa\}$. 
Then vectors $\{\partial_{v} \theta[\I\cup \K^{(\kappa_l)}]\mid l=1,\dots,\rFr\}$ are linearly independent,
and vector $\partial_{v} \theta[\I\cup (\B\backslash \K)]$ is spanned by these $\mathfrak{r}$ vectors,
at that the following relation holds
\begin{multline}\label{D1ThetaLDN}
 \sum_{l=1}^\rFr (-1)^{l-1}\theta[\J^{(j_n)}\cup \K^{(\kappa_l)}]\theta[\J^{(j_m)}\cup \K^{(\kappa_l)}] 
 \theta[\J^{(j_m,j_n)}\cup (\B\backslash\K^{(\kappa_l)})]
  \partial_{v} \theta[\I\cup \K^{(\kappa_l)}] \\
 + (-1)^{\rFr} \theta[\J^{(j_n)}\cup (\B\backslash \K)]\theta[\J^{(j_m)}\cup (\B\backslash \K)]  
 \theta[\J^{(j_m,j_n)}\cup \K] \partial_{v} \theta[\I\cup (\B\backslash \K)] = 0,
\end{multline}
where $\J^{(j)}=\J \backslash \{j\}$ and $\J^{(j_m,j_n)} = \J\backslash \{j_m,j_n\}$.
All sets $\{\I\cup \K^{(\kappa_l)}\}$ and $\I\cup (\B\backslash \K)$ are arranged in the ascending order, and
$\kappa_1 < \kappa_2 < \dots < \kappa_{2\rFr-1}$ is supposed. 
\end{conj}

\begin{rem}
Under the assumptions of Conjecture~\ref{C:D1thetaLDH} 
any vector $\partial_{v} \theta[\I\cup \M]$ with $\M$ consisting of $\rFr-1$
indices from $\B$ and $\M\neq \B\backslash \K$ is spanned by $\rFr-\cFr$ vectors, where
$\cFr$ is the cardinality of $\M \cap \K$.
\end{rem}

\subsection{Second order  theta derivatives}\label{ss:2derThetaC}
Let $\I_0 \cup \J_0$ be a partition of the set of $2g+1$ finite branch points with 
$\I_0=\{i_1,\dots i_g\}$, and $\J_0=\{j_1,\dots j_{g+1}\}$, where index $0$ of infinity belongs to $\I_0$ but omitted. 
Characteristics of multiplicity $2$ arise when $3$ or $4$ indices drop from $\I_0$,
see Subsection \ref{ss:CharPart} for a detailed explanation. Let $\K$ denote the set of indices which drop,
then $\I_2 = \I_0 \backslash \K$, and $\J_2 = \J_0 \cup \K$.

With the help of the general and the second Thomae formulas
the following decomposition of the Hesse matrix $\partial^2_v \theta[\I_2]$
can be found.

\begin{teo}\label{T:D2thetaGradRepr}
 Let $\I_0=\{i_1,\,\dots,\,i_{g}\}$ and $\J_0 = \{j_1,\,\dots,\,j_{g+1}\}$ form a partition of the set 
 $\{1$, $2$, $\dots$, $2g+1\}$ of indices of finite branch points of a genus $g$ hyperelliptic curve, 
 and $\I_2 = \I_0\backslash \K$, where $\K$ is a set of cardinality $\kFr=3$ or 4. Then
\begin{gather}\label{TD2Expr}
 \partial^2_{v_{n_1}, v_{n_2}} \theta[\I_2] = 
 \frac{1}{\theta[\I_0]} \sum_{k,l=1}^\kFr R_{k,l} \partial_{v_{n_1}} \theta[\I_0^{(p_k)}] 
 \partial_{v_{n_2}} \theta[\I_0^{(p_l)}],
\end{gather} 
or in the matrix form
\begin{gather}\label{TD2ExprM}
 \partial^2_v \theta[\I_2] = \frac{1}{\theta[\I_0]} \partial_v \theta[\I_1]^t \hat{R} \partial_v \theta[\I_1],
\end{gather}
where $\partial_v \theta[\I_1]$ denotes a $\kFr\times g$ matrix with $(i,j)$-entry equal to 
$\partial_{v_j} \theta[\I_0^{(p_i)}]$, $p_i \in \K$, 
and $\kFr\times \kFr$ matrix $\hat{R}$ has vanishing diagonal entries $R_{k,k}=0$, 
and off-diagonal entries are defined as follows
with arbitrary $j_n$, $j_m\in\J_0$ and $k$, $l$ running from $1$ to $\kFr$
\begin{itemize}
 \item when $\kFr=3$, so that $\K\backslash\{p_k,p_l\}=\{q\}$
\begin{gather}\label{R3M} 
 R_{k,l} = (-1)^{k+l}
 \frac{ \theta[\I_0^{(p_k,p_l \to j_n,j_m)}]
  \theta[\J_0^{(j_n,j_m \to q)}]\theta[\I_0^{(q\to j_m)}] \theta[\I_0^{(q\to j_n)}]}
 {\theta[\J_0^{(j_m)}] \theta[\J_0^{(j_n)}]
 \theta[\I_0^{(p_k,q \to j_n,j_m)}]\theta[\I_0^{(p_l,q \to j_n,j_m)}]},
\end{gather}
\item when $\kFr=4$, so that $\K\backslash\{p_k,p_l\}=\{q_1,q_2\}$,
\begin{multline}\label{R4M} 
 R_{k,l} = (-1)^{k+l}
 \theta[\I_0^{(p_k,p_l \to j_n,j_m)}] \theta[\I_0^{(q_1,q_2 \to j_n,j_m)}]
 \times \\ \times
 \frac{\theta[\J_0^{(j_n,j_m \to p_k)}]\theta[\J_0^{(j_n,j_m \to p_l)}]}
 {\big(\theta[\J_0^{(j_m)}] \theta[\J_0^{(j_n)}]\big)^2} 
 \prod_{q \in \K\backslash\{p_k,p_l\}} 
 \frac{\theta[\I_0^{(q\to j_m)}] \theta[\I_0^{(q\to j_n)}]}
 {\theta[\I_0^{(p_k,q \to j_n,j_m)}]\theta[\I_0^{(p_l,q \to j_n,j_m)}]}.
\end{multline}
\end{itemize}
\end{teo}

\begin{rem}
Matrix $\hat{R}$ has one of the following forms
\begin{itemize}
 \item in the case of three dropped indices ($\card\K=3$)
\begin{gather}\label{RK3}
 \hat{R} = \begin{pmatrix} 0 & R_{1,2} & R_{1,3} \\ R_{1,2} & 0 & R_{2,3} \\
 R_{1,3} & R_{2,3} & 0 \end{pmatrix},
\end{gather}
\item in the case of four dropped indices ($\card\K=4$)
\begin{gather}\label{RK4}
 \hat{R} = \begin{pmatrix} 
 0 & R_{1,2} & R_{1,3} & R_{1,4} \\
 R_{1,2} & 0 & R_{2,3} & R_{2,4} \\
 R_{1,3} & R_{2,3} & 0 & R_{3,4} \\
 R_{1,4} & R_{2,4} & R_{3,4} & 0
 \end{pmatrix},
\end{gather}
\end{itemize}
the sign of $R_{k,l}$ is $(-1)^{k+l}$ in the case of ascending order of indices in $\K$.
\end{rem}

\begin{proof}
In the case of multiplicity $\mFr=2$ formula \eqref{thomaeNR} acquires the form
\begin{multline}\label{thomae3Inv}
  \frac{\partial^2_{v_{n_1}, v_{n_2}} \theta\big[\I_2\big]}{\theta[\I_0]}
  = \epsilon \prod_{\kappa\in \K}\frac{\big(\prod_{j \in \J_0} (e_\kappa - e_j)\big)^{1/4}}
 {\big(\prod_{\iota \in \I_{2}} (e_\kappa - e_{\iota}) \big)^{1/4}} \times \\ \times
 \sum_{\substack{p_1,p_2 \in \K \\ p_1\neq p_2}}
 \frac{\sum_{j=1}^g(-1)^{j-1} s_{j-1}^{(p_1)} \omega_{jn_1} 
 \sum_{j=1}^g(-1)^{j-1} s_{j-1}^{(p_2)} \omega_{jn_2}}
 {\prod_{\kappa\in \K\backslash\{p_1,p_2\}} (e_{p_1}-e_\kappa)(e_{p_2}-e_\kappa)},
\end{multline}
which holds when $\K$ contains $3$ or $4$ indices.
Applying the second Thomae theorem in the form \eqref{Th2Inv}, obtain
\begin{multline}\label{thomae3Grad}
  \frac{\partial^2_{v_{n_1}, v_{n_2}} \theta\big[\I_2\big]} {\theta[\I_0]}
= \sum_{\substack{p_1,p_2 \in \K \\ p_1\neq p_2}} 
 \frac{\epsilon_{p_1,p_2} (e_{p_1}-e_{p_2})^{1/2}}
 {\prod_{\kappa\in \K\backslash\{p_1,p_2\}} (e_{p_1}-e_\kappa)^{1/2}(e_{p_2}-e_\kappa)^{1/2}} \times \\ \times
 \prod_{\kappa\in \K\backslash\{p_1,p_2\}} \frac{\big(\prod_{j \in \J_0} (e_\kappa - e_j)\big)^{1/4}}
 {\big(\prod_{\iota \in \I_{2}\cup\{p_1,p_2\}} (e_\kappa - e_{\iota}) \big)^{1/4}}
 \frac{1}{\theta[\I_0]^2} \partial_{v_{n_1}} \theta[\I_0^{(p_1)}] 
 \partial_{v_{n_2}} \theta[\I_0^{(p_2)}],
\end{multline}
where $\epsilon_{p_1,p_2}^8=1$. Next, consider separately 
the cases of $\kFr$ equal to $3$ and $4$.

Let $\kFr=3$, so $\K=\{\kappa_1,\kappa_2,\kappa_3\}$.
For a pair $\{\kappa_1,\kappa_2\} \in \K$ with the help of FTT Corollaries~\ref{C:eklm} and \ref{C:eJI}
with the normal order of elements $\{e_k\}$
one finds the multiple at $\partial_{v_{n_1}} \theta[\I_0^{(\kappa_1)}] \partial_{v_{n_2}} \theta[\I_0^{(\kappa_2)}]$,
which is
\begin{multline}\label{thomae3GradK3}
 \Ord \frac{ (e_{\kappa_1} - e_{\kappa_2})^{1/2}}
 {(e_{\kappa_1} - e_{\kappa_3})^{1/2}(e_{\kappa_2} - e_{\kappa_3})^{1/2}}
 \frac{\big(\prod_{j \in \J_0} (e_{\kappa_3} - e_j)\big)^{1/4}}
 {\big(\prod_{\iota \in \I_{0}^{(\kappa_3)}} (e_{\kappa_3} - e_{\iota}) \big)^{1/4}} \\
 = \frac{\theta[\I_0^{(\kappa_1,\kappa_2 \to j_n,j_m)}]
 \theta[\I_0^{(\kappa_3\to j_m)}] \theta[\I_0^{(\kappa_3\to j_n)}] \theta[\J_0^{(j_n,j_m \to \kappa_3)}]}
 {\theta[\I_0^{(\kappa_2,\kappa_3 \to j_n,j_m)}]\theta[\I_0^{(\kappa_3,\kappa_1 \to j_n,j_m)}] 
 \theta[\J_0^{(j_m)}] \theta[\J_0^{(j_n)}]}.
\end{multline}
Here $\I_{2}\cup\{\kappa_1,\kappa_2\} = \I_0^{(\kappa_3)}$. 
Finally, with an arbitrary pair $j_n,j_m \in \J_0$ the following holds
\begin{multline}\label{D2thetaK3}
 \partial^2_{v_{n_1}, v_{n_2}} \theta\big[\I_2\big] = 
 \frac{1}{\theta[\I_0] \theta[\J_0^{(j_m)}] \theta[\J_0^{(j_n)}]} \times \\ \times 
 \sum_{\substack{p_1,p_2 \in \K \\ p_1\neq p_2}}  \epsilon_{p_1,p_2}
 \frac{\theta[\I_0^{(p_1,p_2 \to j_n,j_m)}]
 \theta[\I_0^{(p_3\to j_m)}] \theta[\I_0^{(p_3\to j_n)}] \theta[\J_0^{(j_n,j_m \to p_3)}]}
 {\theta[\I_0^{(p_2,p_3 \to j_n,j_m)}]\theta[\I_0^{(p_3,p_1 \to j_n,j_m)}]}
 \times \\ \times
 \partial_{v_{n_1}} \theta[\I_0^{(p_1)}] \partial_{v_{n_2}} \theta[\I_0^{(p_2)}],
\end{multline}
where $p_1$, $p_2$, $p_3$ denote different elements of $\K$ such that $\{p_3\} = \K \backslash \{p_1,p_2\}$.
In fact, at $\kappa_1 \leqslant \kappa_2 \leqslant \kappa_3$ the multiplier $\epsilon_{p_1,p_2}$  can be taken as 
follows: $\epsilon_{\kappa_1,\kappa_2}=\epsilon_{\kappa_2,\kappa_3}=-1$, and
$\epsilon_{\kappa_1,\kappa_3}=1$.

Let $\kFr=4$, so $\K=\{\kappa_1,\kappa_2,\kappa_3,\kappa_4\}$.
Denote $\I_{2}\cup\{\kappa_1,\kappa_2\} = \I_0^{(\kappa_3,\kappa_4)}$ and so on. 
Then apply FTT Corollaries~\ref{C:eklm} and \ref{C:eJI} with the normal order. 
In particular, with a pair $\{\kappa_1,\kappa_2\}\in \K$
the multiple of $\partial_{v_{n_1}} \theta[\I_0^{(\kappa_1)}] \partial_{v_{n_2}} \theta[\I_0^{(\kappa_2)}]$
gets the form
\begin{multline}\label{thomae3GradK4}
 \Ord \frac{\big((e_{\kappa_2} - e_{\kappa_1})(e_{\kappa_4} - e_{\kappa_3})\big)^{1/2}}
 {\big((e_{\kappa_3} - e_{\kappa_2})(e_{\kappa_4} - e_{\kappa_2})\big)^{1/2}}
 \frac{\big(\prod_{j \in \J_0} (e_{\kappa_3} - e_j)\big)^{1/4}}
 {\big((e_{\kappa_3} - e_{\kappa_1})^2\prod_{\iota \in \I_{0}^{(\kappa_3)}} (e_{\kappa_3} - e_{\iota}) \big)^{1/4}} 
  \times \\ \times
 \frac{\big(\prod_{j \in \J_0} (e_{\kappa_4} - e_j)\big)^{1/4}}
 {\big((e_{\kappa_4} - e_{\kappa_1})^2\prod_{\iota \in \I_{0}^{(\kappa_4)}} (e_{\kappa_4} - e_{\iota}) \big)^{1/4}} \\
 = \frac{\theta[\I_0^{(\kappa_1,\kappa_2 \to j_n,j_m)}]\theta[\I_0^{(\kappa_3,\kappa_4 \to j_n,j_m)}]
 \theta[\I_0^{(\kappa_3\to j_m)}] \theta[\I_0^{(\kappa_3\to j_n)}]}
 { \theta[\I_0^{(\kappa_1,\kappa_3 \to j_n,j_m)}] \theta[\I_0^{(\kappa_2,\kappa_3 \to j_n,j_m)}] } \times \\ \times
 \frac{\theta[\I_0^{(\kappa_4\to j_m)}] \theta[\I_0^{(\kappa_4\to j_n)}]
 \theta[\J_0^{(j_n,j_m \to \kappa_1)}] \theta[\J_0^{(j_n,j_m \to \kappa_2)}]}
 {\theta[\I_0^{(\kappa_1,\kappa_4 \to j_n,j_m)}] \theta[\I_0^{(\kappa_2,\kappa_4 \to j_n,j_m)}]
 \theta[\J_0^{(j_m)}]^2 \theta[\J_0^{(j_n)}]^2},
\end{multline}
where the following relations are used
\begin{gather}
 \bigg( \Ord \frac{e_{\kappa_2} - e_{\kappa_1}}{e_{\kappa_3} - e_{\kappa_2}} \bigg)^{1/2} = 
 \frac{\theta[\I_0^{(\kappa_1,\kappa_2 \to j_n,j_m)}] \theta[\J_0^{(j_n,j_m \to \kappa_3)}]}
 {\theta[\I_0^{(\kappa_2,\kappa_3 \to j_n,j_m)}] \theta[\J_0^{(j_n,j_m \to \kappa_1)}]},\notag \\
  \bigg( \Ord \frac{e_{\kappa_4} - e_{\kappa_3}}{e_{\kappa_4} - e_{\kappa_2}} \bigg)^{1/2} = 
 \frac{\theta[\I_0^{(\kappa_3,\kappa_4 \to j_n,j_m)}] \theta[\J_0^{(j_n,j_m \to \kappa_2)}]}
 {\theta[\I_0^{(\kappa_2,\kappa_4 \to j_n,j_m)}] \theta[\J_0^{(j_n,j_m \to \kappa_3)}] },\label{eRels} \\
 \bigg( \Ord \frac{\prod_{j \in \J_0} (e_{\kappa_3} - e_j)}
 {(e_{\kappa_3} - e_{\kappa_1})^2\prod_{\iota \in \I_{0}^{(\kappa_3)}} (e_{\kappa_3} - e_{\iota}) } \bigg)^{1/4}
 = \frac{\theta[\I_0^{(\kappa_3\to j_n)}] \theta[\I_0^{(\kappa_3\to j_m)}] \theta[\J_0^{(j_n,j_m \to \kappa_1)}] }
 {\theta[\I_0^{(\kappa_1,\kappa_3\to j_n,j_m)}] \theta[\J_0^{(j_m)}] \theta[\J_0^{(j_n)}]}, \notag \\
  \bigg( \Ord \frac{\prod_{j \in \J_0} (e_{\kappa_4} - e_j)}
 {(e_{\kappa_4} - e_{\kappa_1})^2\prod_{\iota \in \I_{0}^{(\kappa_4)}} (e_{\kappa_4} - e_{\iota}) }  \bigg)^{1/4}
 = \frac{\theta[\I_0^{(\kappa_4\to j_n)}] \theta[\I_0^{(\kappa_4\to j_m)}]\theta[\J_0^{(j_n,j_m \to \kappa_1)}]}
 {\theta[\I_0^{(\kappa_1,\kappa_4\to j_n,j_m)}] \theta[\J_0^{(j_m)}] \theta[\J_0^{(j_n)}]}, \notag 
\end{gather}
and $j_m$, $j_n$ are chosen the same in all the relations.
Finally,
\begin{multline}\label{D2thetaK4}
  \partial^2_{v_{n_1}, v_{n_2}} \theta\big[\I_2\big] = 
  \frac{\epsilon}{\theta[\I_0]\theta[\J_0^{(j_m)}]^2 \theta[\J_0^{(j_n)}]^2} \times \\  \times
  \sum_{\substack{p_1,p_2 \in \K \\ p_1\neq p_2}}  
 (-1)^{p_1+ p_2} \frac{\theta[\I_0^{(p_1,p_2 \to j_n,j_m)}]\theta[\I_0^{(p_3,p_4 \to j_n,j_m)}]
 \theta[\I_0^{(p_3\to j_m)}] \theta[\I_0^{(p_3\to j_n)}]}
 { \theta[\I_0^{(p_1,p_3 \to j_n,j_m)}] \theta[\I_0^{(p_2,p_3 \to j_n,j_m)}] } \times \\ \times
 \frac{\theta[\I_0^{(p_4\to j_m)}] \theta[\I_0^{(p_4\to j_n)}]
 \theta[\J_0^{(j_n,j_m \to p_1)}] \theta[\J_0^{(j_n,j_m \to p_2)}]}
 {\theta[\I_0^{(p_1,p_4 \to j_n,j_m)}] \theta[\I_0^{(p_2,p_4 \to j_n,j_m)}]} \times \\ \times
 \partial_{v_{n_1}} \theta[\I_0^{(p_1)}] \partial_{v_{n_2}} \theta[\I_0^{(p_2)}] . 
\end{multline}
where $\{p_3,p_4\} = \K \backslash \{p_1,p_2\}$ for each pair $\{p_1,p_2\}$.
In fact, the multiplier $\epsilon$ equals~$1$, and remains the same for all partitions $\I_2 \cup \J_2$ 
in all genera. This accords with the statement of Remark~\ref{R:eOrd} that multiplier $\epsilon$
in the general Thomae formula for second order theta derivatives is the same in all genera.
\end{proof}

\begin{rem}\label{R:2Form}
Theorem~\ref{T:D2thetaGradRepr} provides expressions for second order theta derivatives 
in terms of theta constants and first order theta derivatives. This can be considered as a \emph{generalization
of Jacobi's derivative formula}. 
Each entry of $\partial^2 \theta[\I_2]$ with $\I_2=\I_0\backslash \K$ equals a symmetric bilinear form
with matrix $\hat{R}$, namely:
\begin{gather}\label{D2thetaEntry}
 \partial_{v_{n_1},v_{n_2}}^2 \theta[\I_2] = \partial_{v_{n_1}} \theta[\I_1]^t \hat{R}\, \partial_{v_{n_2}} \theta[\I_1],
\end{gather}
where in the case of $\K=\{i_1,i_2,i_3\} \subset \I_0$
\begin{align*}
 &\partial_{v_{n}} \theta[\I_1] = \big(\partial_{v_{n}} \theta[\I_0^{(i_1)}],
 \partial_{v_{n}} \theta[\I_0^{(i_2)}], \partial_{v_{n}} \theta[\I_0^{(i_3)}]\big)^t& 
\intertext{or in the case of $\K=\{i_1,i_2,i_3,i_4\}\subset \I_0$}
 &\partial_{v_{n}} \theta[\I_1] = \big(\partial_{v_{n}} \theta[\I_0^{(i_1)}],
 \partial_{v_{n}} \theta[\I_0^{(i_2)}], \partial_{v_{n}} \theta[\I_0^{(i_3)}], \partial_{v_{n}} \theta[\I_0^{(i_4)}]\big)^t.&
\end{align*}
Note that first order theta derivatives $\{\partial_v \theta[\I_0^{(i)}]\}_{i\in\K}$
involved into \eqref{D2thetaEntry} form a set of linearly independent vectors.
\end{rem}

\begin{exam}
In the case of genus $3$ hyperelliptic curve there is a unique partition $\I_2=\emptyset$ of multiplicity $2$,
which is obtained by $\binom{7}{3}=35$ ways from partitions $\I_0\cup \J_0$ of $7$ indices.
With $\I_0=\{i_1,i_2,i_3\}$, $i_1<i_2<i_3$, and $\J_0=\{j_1,j_2,j_3,j_4\}$, and $j_m=j_1$, $j_n=j_2$
relation \eqref{TD2ExprM} reads as
\begin{multline}\label{DI2RelG3}
 \partial_{v_{n_1},v_{n_2}}^2 \theta^{\emptyset} = 
 \frac{1}{\theta^{\{i_1,i_2,i_3\}}\theta^{\{j_2,j_3,j_4\}}\theta^{\{j_1,j_3,j_4\}}} \times \\ \times
 \Big(- \theta^{\{i_3,j_1,j_2\}}\theta^{\{i_3,j_3,j_4\}}\theta^{\{i_1,i_2,j_1\}}\theta^{\{i_1,i_2,j_2\}}
 \big(\theta^{\{i_1,j_1,j_2\}}\theta^{\{i_2,j_1,j_2\}}\big)^{-1} \times \\ \times
 \big( \partial_{v_{n_1}} \theta^{\{i_2,i_3\}} \partial_{v_{n_2}} \theta^{\{i_1,i_3\}} 
 + \partial_{v_{n_2}} \theta^{\{i_2,i_3\}} \partial_{v_{n_1}} \theta^{\{i_1,i_3\}}\big)\\
 + \theta^{\{i_2,j_1,j_2\}}\theta^{\{i_2,j_3,j_4\}}\theta^{\{i_1,i_3,j_1\}}\theta^{\{i_1,i_3,j_2\}}
 \big(\theta^{\{i_1,j_1,j_2\}}\theta^{\{i_3,j_1,j_2\}}\big)^{-1} \times \\ \times
 \big( \partial_{v_{n_1}} \theta^{\{i_2,i_3\}} \partial_{v_{n_2}} \theta^{\{i_1,i_2\}} 
 + \partial_{v_{n_2}} \theta^{\{i_2,i_3\}} \partial_{v_{n_1}} \theta^{\{i_1,i_2\}}\big)\\
 - \theta^{\{i_1,j_1,j_2\}}\theta^{\{i_1,j_3,j_4\}}\theta^{\{i_2,i_3,j_1\}}\theta^{\{i_2,i_3,j_2\}}
 \big(\theta^{\{i_2,j_1,j_2\}}\theta^{\{i_3,j_1,j_2\}}\big)^{-1} \times \\ \times
 \big( \partial_{v_{n_1}} \theta^{\{i_1,i_3\}} \partial_{v_{n_2}} \theta^{\{i_1,i_2\}} 
 + \partial_{v_{n_2}} \theta^{\{i_1,i_3\}} \partial_{v_{n_1}} \theta^{\{i_1,i_2\}}\big).
\end{multline}
Evidently, second order theta derivatives are represented by expressions quadratic in 
first order theta derivatives. 

Moreover, each entry $\partial_{v_{n_1},v_{n_2}}^2 \theta^{\emptyset}$ 
is represented by $35$ expressions produced from different $\I_0$,
and these expressions are all equal to each other. Some relations of this form are given in Appendix~\ref{A:D2thetaLD4Ex}.
\end{exam}

\begin{exam}
In the case of genus $4$ among half-period characteristics of multiplicity~$2$ there is one
whose partition is obtained by dropping $4$ indices from $\I_0$, that is $\I_2=\emptyset$. Other  
characteristics of multiplicity $2$ are obtained by dropping $3$ indices from $\I_0$, 
then $\I_2 = \{\iota\}$.

Let $\I_0=\{i_1$, $i_2$, $i_3$, $\iota\}$ with $i_1<i_2<i_3$, and $\J_0=\{j_1$, $j_2$, $j_3$, $j_4$, $j_5\}$.
By dropping indices $\{i_1,i_2,i_3\}=\K$ one gets 
a representation for $\partial_v^2 \theta^{\{\iota\}}$ of the form
\eqref{TD2ExprM} (here $j_m=j_1$, $j_n=j_2$)
\begin{multline}\label{DI2RelG4}
 \partial_{v_{n_1},v_{n_2}}^2 \theta^{\{\iota\}} = 
 \frac{1}{\theta^{\{i_1,i_2,i_3,\iota\}}\theta^{\{j_2,j_3,j_4,j_5\}}\theta^{\{j_1,j_3,j_4,j_5\}}} \times \\ \times
 \Big(- \theta^{\{i_3,j_1,j_2,\iota\}}\theta^{\{i_3,j_3,j_4,j_5\}}\theta^{\{i_1,i_2,j_1,\iota\}}\theta^{\{i_1,i_2,j_2,\iota\}}
 \big(\theta^{\{i_1,j_1,j_2,\iota\}}\theta^{\{i_2,j_1,j_2,\iota\}}\big)^{-1} \times \\ \times
 \big( \partial_{v_{n_1}} \theta^{\{i_2,i_3,\iota\}} \partial_{v_{n_2}} \theta^{\{i_1,i_3,\iota\}} 
 + \partial_{v_{n_2}} \theta^{\{i_2,i_3,\iota\}} \partial_{v_{n_1}} \theta^{\{i_1,i_3,\iota\}}\big)\\
 + \theta^{\{i_2,j_1,j_2,\iota\}}\theta^{\{i_2,j_3,j_4,j_5\}}\theta^{\{i_1,i_3,j_1,\iota\}}\theta^{\{i_1,i_3,j_2,\iota\}}
 \big(\theta^{\{i_1,j_1,j_2,\iota\}}\theta^{\{i_3,j_1,j_2,\iota\}}\big)^{-1} \times \\ \times
 \big( \partial_{v_{n_1}} \theta^{\{i_2,i_3,\iota\}} \partial_{v_{n_2}} \theta^{\{i_1,i_2,\iota\}} 
 + \partial_{v_{n_2}} \theta^{\{i_2,i_3,\iota\}} \partial_{v_{n_1}} \theta^{\{i_1,i_2,\iota\}}\big)\\
 - \theta^{\{i_1,j_1,j_2,\iota\}}\theta^{\{i_1,j_3,j_4,j_5\}}\theta^{\{i_2,i_3,j_1,\iota\}}\theta^{\{i_2,i_3,j_2,\iota\}}
 \big(\theta^{\{i_2,j_1,j_2,\iota\}}\theta^{\{i_3,j_1,j_2,\iota\}}\big)^{-1} \times \\ \times
 \big( \partial_{v_{n_1}} \theta^{\{i_1,i_3,\iota\}} \partial_{v_{n_2}} \theta^{\{i_1,i_2,\iota\}} 
 + \partial_{v_{n_2}} \theta^{\{i_1,i_3,\iota\}} \partial_{v_{n_1}} \theta^{\{i_1,i_2,\iota\}}\big).
\end{multline}
Next, let $\I_0=\{i_1$, $i_2$, $i_3$, $i_4\}$ with $i_1<i_2<i_3<i_4$, and $\J_0=\{j_1$, $j_2$, $j_3$, $j_4$, $j_5\}$.
A representation for $\partial^2_v\theta^{\emptyset}$ is obtained from \eqref{TD2ExprM} when all four indices of $\I_0$
are dropped.
Examples of formulas for these second derivative theta constants are given in Appendix~\ref{A:D2thetaLD4Ex}.
\end{exam}

\begin{rem}\label{R:multrepr}
 Each entry in \eqref{TD2ExprM} has many representations depending on a choice of $\K$ 
 which is subtracted from $\I_0$, and
 a choice of $j_n$ and $j_m$. All these expressions are equivalent due to Propositions~\ref{P:D1thetaLD4} and \ref{P:D1thetaLD3}'.
 In the case of three dropped indices this is stated by the following
\end{rem}
\begin{prop}\label{P:D2thetaLD3}
Let $\I=\{i_1$, $i_2$, $\dots$, $i_{g-3}\}$, $\J=\{j_1$, $j_2$, $\dots$, $j_{g}\}$,
and $\I\cup \J \cup\{p_1$, $p_2$, $p_3$, $p_4\}$ be a partition 
of the set $\{1$, $2$, $\dots$, $2g+1\}$ of indices of finite branch points.
Then the right hand sides of \eqref{TD2ExprM} with two partitions
$\big(\I\cup \{p_1,p_2,p_3\}\big) \cup \big(\J \cup \{p_4\}\big)$, 
and $\big(\I\cup \{p_1,p_2,p_4\}\big) \cup \big(\J \cup \{p_3\}\big)$ are equal.
\end{prop}

In the case of four dropped indices the following holds
\begin{prop}\label{P:D2thetaLD4}
Let $\I=\{i_1$, $i_2$, $\dots$, $i_{g-4}\}$, $\J=\{j_1$, $j_2$, $\dots$, $j_{g}\}$,
and $\I\cup \J \cup\{p_1$, $p_2$, $p_3$, $p_4$, $p_5\}$ be a partition 
of the set $\{1$, $2$, $\dots$, $2g+1\}$ of indices of finite branch points.
Then the right hand sides of \eqref{TD2ExprM} with two partitions
$\big(\I\cup \{p_1,p_2,p_3,p_4\}\big) \cup \big(\J \cup \{p_5\}\big)$, 
and $\big(\I\cup \{p_1,p_2,p_3,p_5\}\big) \cup \big(\J \cup \{p_4\}\big)$ are equal.
\end{prop}

\subsection{Third order theta derivatives}\label{ss:3derThetaC}
Recall that $\I_0 \cup \J_0$ denotes a partition of the set of $2g+1$ finite branch points with 
$\I_0=\{i_1,\dots i_g\}$, and $\J_0=\{j_1,\dots j_{g+1}\}$. 
A characteristic $[\I_3]$ of multiplicity $3$ corresponds to $\I_3$
obtained from $\I_0$ by dropping $5$ or $6$ indices.
Let $\K$ denote the set of indices which drop,
then $\I_3 = \I_0 \backslash \K$, and $\J_3 = \J_0 \cup \K$.

Third order theta derivatives are expressed in terms of theta constants and first order theta derivatives
as shown in the following
\begin{teo}\label{T:D3thetaGradRepr}
 Let $\I_0=\{i_1$, $\dots$, $i_{g}\}$ and $\J_0 = \{j_1$, $\dots$, $j_{g+1}\}$ form a partition of 
 the set $\{1$, $2$, $\dots$, $2g+1\}$ of indices of finite branch points of a genus $g$ hyperelliptic curve, 
 and $\I_3 = \I_0\backslash \K$, where $\K$ is a set of cardinality $\kFr=5$ or $6$. Then
\begin{multline}\label{Th3InvM}
 \partial^3_{v_{n_1},v_{n_2},v_{n_3}} \theta[\I_3] \\ = \frac{-1}{\theta[\I_0]^2} \sum_{k_1,k_2,k_3=1}^\kFr 
  R_{k_1,k_2,k_3} \partial_{v_{n_1}} \theta[\I_0^{(p_{k_1})}] 
  \partial_{v_{n_2}} \theta[\I_0^{(p_{k_2})}]
  \partial_{v_{n_3}} \theta[\I_0^{(p_{k_3})}],
\end{multline}
where $\I_0^{(p)}=\I_0 \backslash \{p\}$, and $R_{k_1,k_2,k_3}$ form a symmetric tensor of order~$3$ 
with vanishing diagonal entries $R_{k,k,k}=R_{k,k,l}=R_{k,l,k}=R_{l,k,k}=0$, 
and off-diagonal entries defined as follows
with arbitrary $j_n$, $j_m\in\J_0$
\begin{itemize}
 \item when $\kFr=5$, so that $\{q_1, q_2\}=\K\backslash\{p_{k_1},$
 $p_{k_2},$ $p_{k_3}\}$
\begin{multline}\label{R5M} 
 R_{k_1,k_2,k_3} = (-1)^{p_{k_1}+p_{k_2}+p_{k_3}}
 \frac{\theta[\I_0^{(q_1,q_2 \to j_n,j_m)}] }{\big(\theta[\J_0^{(j_m)}]\theta[\J_0^{(j_n)}]\big)^{2}} \times \\ \times 
 \theta[\I_0^{(p_{k_1},p_{k_2} \to j_n,j_m)}]\theta[\I_0^{(p_{k_1},p_{k_3} \to j_n,j_m)}]
 \theta[\I_0^{(p_{k_2},p_{k_3} \to j_n,j_m)}] \times \\ \times
 \prod_{q\in \K\backslash\{p_{k_1},p_{k_2},p_{k_3}\}} 
 \frac{\theta[\I_0^{(q\to j_m)}] \theta[\I_0^{(q\to j_n)}] \theta[\J_0^{(j_n,j_m \to q)}]}
 {\theta[\I_0^{(p_{k_1},q \to j_n,j_m)}]\theta[\I_0^{(p_{k_2},q \to j_n,j_m)}]\theta[\I_0^{(p_{k_3},q \to j_n,j_m)}]},
\end{multline}
\item when $\kFr=6$, so that $\{q_1,q_2,q_3\}=\K\backslash\{p_{k_1},p_{k_2},p_{k_3}\}$,
\begin{multline}\label{R6M} 
 R_{k_1,k_2,k_3} = (-1)^{p_{k_1}+p_{k_2}+p_{k_3}}
 \frac{1}{\big(\theta[\J_0^{(j_m)}] \theta[\J_0^{(j_n)}]\big)^3}  
 \theta[\I_0^{(p_{k_1},p_{k_2} \to j_n,j_m)}]  \times \\ \times
 \theta[\I_0^{(p_{k_1},p_{k_3} \to j_n,j_m)}]
 \theta[\I_0^{(p_{k_2},p_{k_3} \to j_n,j_m)}]  \theta[\I_0^{(q_1,q_2 \to j_n,j_m)}] \times \\ \times
 \theta[\I_0^{(q_1,q_3 \to j_n,j_m)}] \theta[\I_0^{(q_2,q_3 \to j_n,j_m)}] 
 \prod_{i=1}^3 \theta[\J_0^{(j_n,j_m \to p_{k_i})}] \times \\ \times
 \prod_{q\in \K\backslash\{p_{k_1},p_{k_2},p_{k_3}\}} 
 \frac{\theta[\I_0^{(q\to j_m)}] \theta[\I_0^{(q\to j_n)}]}
 {\theta[\I_0^{(p_{k_1},q \to j_n,j_m)}] \theta[\I_0^{(p_{k_2},q \to j_n,j_m)}]
 \theta[\I_0^{(p_{k_3},q \to j_n,j_m)}]}.
\end{multline}
\end{itemize}
\end{teo}

\begin{proof}
With $\mFr=3$ formula \eqref{thomaeNR} gets the form
\begin{multline}\label{thomae5Inv}
  \frac{\partial^3_{v_{n_1},v_{n_2},v_{n_3}} \theta\big[\I_3\big]}{\theta[\I_0]}
  = \epsilon \prod_{\kappa\in \K}\frac{\big(\prod_{j \in \J_0} (e_\kappa - e_j)\big)^{1/4}}
 {\big(\prod_{\iota \in \I_{3}} (e_\kappa - e_{\iota}) \big)^{1/4}} \times \\ \times
 \sum_{\substack{p_1,p_2,p_3 \in \K \\ \text{all different}}}
 \frac{\sum_{j=1}^g(-1)^{j-1} s_{j-1}^{(p_1)} \omega_{jn_1} 
 \sum_{j=1}^g(-1)^{j-1} s_{j-1}^{(p_2)} \omega_{jn_2} }
 {\prod_{\kappa\in \K\backslash\{p_1,p_2,p_3\}} (e_{p_1}-e_\kappa)(e_{p_2}-e_\kappa)}
 \times \\ \times \frac{\sum_{j=1}^g(-1)^{j-1} s_{j-1}^{(p_3)} \omega_{jn_3}}
 {\prod_{\kappa\in \K\backslash\{p_1,p_2,p_3\}}(e_{p_3}-e_\kappa)},
\end{multline}
where $\K$ consists of $5$ or $6$ indices.
After substitution of \eqref{Th2Inv} the right hand side of \eqref{thomae5Inv} transforms into
\begin{multline}\label{RthomaeM3}
 \frac{1}{\theta[\I_0]^3}\sum_{\substack{p_1,p_2,p_3 \in \K \\ \text{all different}}}
 \frac{\epsilon_{p_1,p_2,p_3} \Delta[\{e_{p_1},e_{p_2},e_{p_3}\}]^{1/2}}
 {\big(\prod_{\kappa\in \K\backslash\{p_1,p_2,p_3\}} \prod_{i=1}^3 (e_{p_i}-e_\kappa)\big)^{1/2}} \times \\ \times
 \prod_{\kappa\in \K\backslash\{p_1,p_2,p_3\}} \frac{\big(\prod_{j \in \J_0} (e_\kappa - e_j)\big)^{1/4}}
 {\big(\prod_{\iota \in \I_{3}\cup\{p_1,p_2,p_3\}} (e_\kappa - e_{\iota}) \big)^{1/4}} \times \\ \times
  \partial_{v_{n_1}} \theta[\I_0^{(p_1)}] 
 \partial_{v_{n_2}} \theta[\I_0^{(p_2)}] \partial_{v_{n_3}} \theta[\I_0^{(p_3)}] ,
\end{multline}
where $\epsilon_{p_1,p_2,p_3}^8=1$, and $\Delta[\{e_{p_1},e_{p_2},e_{p_3}\}]$ denotes the Vandermonde determinant
built from elements $\{e_{p_1},e_{p_2},e_{p_3}\}$.

First, let the cardinality of $\K$ be $5$, so $\K=\{\kappa_1,\kappa_2,\kappa_3, \kappa_4,\kappa_5\}$.
For a set of three $\{\kappa_1,\kappa_2,\kappa_3\} \in \K$, applying FTT Corollaries~\ref{C:eklm} and \ref{C:eJI}
with normal ordering,
one finds  the coefficient of $\partial_{v_{n_1}} \theta[\I_0^{(\kappa_1)}] \partial_{v_{n_2}} \theta[\I_0^{(\kappa_2)}]
\partial_{v_{n_3}} \theta[\I_0^{(\kappa_3)}]$:
\begin{multline*}
 \Ord \frac{\big((e_{\kappa_2} - e_{\kappa_1})(e_{\kappa_3} - e_{\kappa_1})
  (e_{\kappa_3} - e_{\kappa_2})(e_{\kappa_5} - e_{\kappa_4})\big)^{1/2}}
  {\big((e_{\kappa_4} - e_{\kappa_2})(e_{\kappa_4} - e_{\kappa_3})
  (e_{\kappa_5} - e_{\kappa_2})(e_{\kappa_5} - e_{\kappa_3})\big)^{1/2}}  \times \\ \times
 \frac{\big(\prod_{j \in \J_0} (e_{\kappa_4} - e_j)\big)^{1/4}}
 {\big((e_{\kappa_4}-e_{\kappa_1})^2 \prod_{\iota \in \I_{0}^{(\kappa_4)}} (e_{\kappa_4} - e_{\iota})\big)^{1/4}} 
 \frac{\big(\prod_{j \in \J_0} (e_{\kappa_5} - e_j)\big)^{1/4}}
 {\big((e_{\kappa_5}-e_{\kappa_1})^2 \prod_{\iota \in \I_{0}^{(\kappa_5)}} (e_{\kappa_5} - e_{\iota})\big)^{1/4}} \\
 = \frac{\theta[\I_0^{(\kappa_1,\kappa_2 \to j_n,j_m)}]\theta[\I_0^{(\kappa_1,\kappa_3 \to j_n,j_m)}]
 \theta[\I_0^{(\kappa_2,\kappa_3 \to j_n,j_m)}] \theta[\I_0^{(\kappa_4,\kappa_5 \to j_n,j_m)}]}
 { \theta[\I_0^{(\kappa_1,\kappa_4 \to j_n,j_m)}] \theta[\I_0^{(\kappa_2,\kappa_4 \to j_n,j_m)}]
 \theta[\I_0^{(\kappa_3,\kappa_4 \to j_n,j_m)}] \theta[\I_0^{(\kappa_1,\kappa_5 \to j_n,j_m)}] } \times \\ \times
 \frac{\theta[\I_0^{(\kappa_4\to j_m)}] \theta[\I_0^{(\kappa_4\to j_n)}]
 \theta[\I_0^{(\kappa_5\to j_m)}] \theta[\I_0^{(\kappa_5\to j_n)}]
 \theta[\J_0^{(j_n,j_m \to \kappa_4)}] \theta[\J_0^{(j_n,j_m \to \kappa_5)}]}
 {\theta[\I_0^{(\kappa_2,\kappa_5 \to j_n,j_m)}]\theta[\I_0^{(\kappa_3,\kappa_5 \to j_n,j_m)}]
 \theta[\J_0^{(j_m)}]^2 \theta[\J_0^{(j_n)}]^2},
\end{multline*}
where relations similar to \eqref{eRels} are used.
Finally, with an arbitrary pair $j_n,j_m \in \J_0$ and elements of $\K$ denoted by
$p_1$, $p_2$, $p_3$, $q_1$, $q_2$ the following holds
\begin{multline}\label{D3thetaK5}
\partial^3_{v_{n_1},v_{n_2},v_{n_3}} \theta [\I_3] = 
 \frac{\epsilon}{\theta[\I_0]^2 \theta[\J_0^{(j_m)}]^2 \theta[\J_0^{(j_n)}]^2} 
 \sum_{\substack{p_1,p_2,p_3 \in \K \\ \text{all different}}}  (-1)^{p_1+p_2+p_3} \times \\ \times
 \theta[\I_0^{(p_1,p_2 \to j_n,j_m)}]\theta[\I_0^{(p_1,p_3 \to j_n,j_m)}]
 \theta[\I_0^{(p_2,p_3 \to j_n,j_m)}] \theta[\I_0^{(q_1,q_2 \to j_n,j_m)}] \times \\ \times
 \bigg(\prod_{q\in \K\backslash\{p_1,p_2,p_3\}} 
 \frac{\theta[\I_0^{(q\to j_m)}] \theta[\I_0^{(q\to j_n)}] \theta[\J_0^{(j_n,j_m \to q)}]}
 {\theta[\I_0^{(p_1,q \to j_n,j_m)}]\theta[\I_0^{(p_2,q \to j_n,j_m)}]\theta[\I_0^{(p_3,q \to j_n,j_m)}]} \bigg)
 \times \\ \times 
 \partial_{v_{n_1}} \theta[\I_0^{(p_1)}] \partial_{v_{n_2}} \theta[\I_0^{(p_2)}]
 \partial_{v_{n_3}} \theta[\I_0^{(p_3)}].
\end{multline}

Next, let  $\K=\{\kappa_1,\kappa_2,\kappa_3,\kappa_4,\kappa_5,\kappa_6\}$.
Then the coefficient of $\partial_{v_{n_1}} \theta[\I_0^{(p_1)}]$
$\partial_{v_{n_2}} \theta[\I_0^{(p_2)}]
\partial_{v_{n_3}} \theta[\I_0^{(p_3)}]$ in \eqref{RthomaeM3} is expressed in terms of theta constants with 
the help of FTT Corollaries~\ref{C:eklm} and \ref{C:eJI} with normal ordering, namely:
\begin{multline*}
 \Ord \frac{ \big((e_{\kappa_2} - e_{\kappa_1})(e_{\kappa_3} - e_{\kappa_1})
  (e_{\kappa_3} - e_{\kappa_2})(e_{\kappa_5} - e_{\kappa_4})\big)^{1/2}}
  { \prod_{i=4}^6 \big((e_{\kappa_i} - e_{\kappa_2})(e_{\kappa_i} - e_{\kappa_3})\big)^{1/2}}  \times \\ \times
 \big((e_{\kappa_6} - e_{\kappa_4})(e_{\kappa_6} - e_{\kappa_5})\big)^{1/2}
  \prod_{i=4}^6 \frac{\big(\prod_{j \in \J_0} (e_{\kappa_i} - e_j)\big)^{1/4}}
 {\big((e_{\kappa_i}-e_{\kappa_1})^2 \prod_{\iota \in \I_{0}^{(\kappa_i)}} (e_{\kappa_i} - e_{\iota})\big)^{1/4}} \\
 = \theta[\I_0^{(\kappa_1,\kappa_2 \to j_n,j_m)}]\theta[\I_0^{(\kappa_1,\kappa_3 \to j_n,j_m)}]
 \theta[\I_0^{(\kappa_2,\kappa_3 \to j_n,j_m)}] \theta[\I_0^{(\kappa_4,\kappa_5 \to j_n,j_m)}]\times \\ \times 
 \theta[\I_0^{(\kappa_4,\kappa_6 \to j_n,j_m)}]\theta[\I_0^{(\kappa_5,\kappa_6 \to j_n,j_m)}]
 \frac{\prod_{i=1}^3\theta[\J_0^{(j_n,j_m \to \kappa_i)}]}{\theta[\J_0^{(j_m)}]^3 \theta[\J_0^{(j_n)}]^3} 
 \times \\ \times
 \prod_{i=4}^6 \frac{\theta[\I_0^{(\kappa_i\to j_m)}] \theta[\I_0^{(\kappa_i\to j_n)}]} 
 {\theta[\I_0^{(\kappa_1,\kappa_i \to j_n,j_m)}]
 \theta[\I_0^{(\kappa_2,\kappa_i \to j_n,j_m)}]\theta[\I_0^{(\kappa_3,\kappa_i \to j_n,j_m)}]}.
\end{multline*}
Finally,
\begin{multline}\label{D3thetaK6}
  \partial^3_{v_{n_1},v_{n_2},v_{n_3}} \theta\big[\I_3\big] = 
  \frac{\epsilon}{\theta[\I_0]^2\theta[\J_0^{(j_m)}]^3 \theta[\J_0^{(j_n)}]^3}  
  \sum_{\substack{p_1,p_2,p_3 \in \K \\ \text{all different}}}  (-1)^{p_1+p_2+p_3} \times \\ \times
 \theta[\I_0^{(p_1,p_2 \to j_n,j_m)}]\theta[\I_0^{(p_1,p_3 \to j_n,j_m)}]
 \theta[\I_0^{(p_2,p_3 \to j_n,j_m)}] \theta[\I_0^{(q_1,q_2 \to j_n,j_m)}] \times \\ \times
 \theta[\I_0^{(q_1,q_3 \to j_n,j_m)}] \theta[\I_0^{(q_2,q_3 \to j_n,j_m)}] 
 \theta[\J_0^{(j_n,j_m \to p_1)}]\theta[\J_0^{(j_n,j_m \to p_2)}]\theta[\J_0^{(j_n,j_m \to p_3)}] \times \\ \times
 \bigg(\prod_{q\in \K\backslash\{p_1,p_2,p_3\}} 
 \frac{\theta[\I_0^{(q\to j_m)}] \theta[\I_0^{(q\to j_n)}]}
 {\theta[\I_0^{(p_1,q \to j_n,j_m)}]\theta[\I_0^{(p_2,q \to j_n,j_m)}]\theta[\I_0^{(p_3,q \to j_n,j_m)}]} \bigg)
 \times \\ \times 
 \partial_{v_{n_1}} \theta[\I_0^{(p_1)}] \partial_{v_{n_2}} \theta[\I_0^{(p_2)}]
 \partial_{v_{n_3}} \theta[\I_0^{(p_3)}],
\end{multline}
where $\{q_1,q_2,q_3\} = \K \backslash \{p_1,p_2,p_3\}$ for each set of three $\{p_1,p_2,p_3\}$. In fact,
the multiplier $\epsilon$ equals $-1$,  and remains the same for all partitions $\I_3 \cup \J_3$ 
in all genera. This accords with the statement of Remark~\ref{R:eOrd} that multiplier $\epsilon$
in the general Thomae formula for third order theta derivatives is alternating as well as $\epsilon$ 
in the second Thomae formula.
\end{proof}

\begin{rem}\label{R:3Form}
Theorem~\ref{T:D3thetaGradRepr} gives a representation of third order theta derivatives 
in terms of theta constants and first order theta derivatives which also can be considered as a \emph{generalization
of Jacobi's derivative formula}. Note that \eqref{Th3InvM} is a symmetric trilinear form
with tensor $\hat{R}$ of order $3$ on the space of vectors $\partial_{v_{n}} \theta[\I_1] $
of the form
\begin{align*}
 &\partial_{v_{n}} \theta[\I_1] = \big(\partial_{v_{n}} \theta[\I_0^{(i_1)}],
 \partial_{v_{n}} \theta[\I_0^{(i_2)}], \partial_{v_{n}} \theta[\I_0^{(i_3)}], 
 \partial_{v_{n}} \theta[\I_0^{(i_4)}], \partial_{v_{n}} \theta[\I_0^{(i_5)}]\big)^t,& 
\intertext{where $\{i_1,i_2,i_3,i_4,i_5\}=\K$ such that $\I_3=\I_0\backslash \K$, or}
 &\partial_{v_{n}} \theta[\I_1] = \big(\partial_{v_{n}} \theta[\I_0^{(i_1)}],
 \partial_{v_{n}} \theta[\I_0^{(i_2)}], \partial_{v_{n}} \theta[\I_0^{(i_3)}], 
 \partial_{v_{n}} \theta[\I_0^{(i_4)}], \\ 
 &\phantom{mmmmmmmmmmmmmmmmmm}\partial_{v_{n}} \theta[\I_0^{(i_5)}], 
 \partial_{v_{n}} \theta[\I_0^{(i_6)}]\big)^t,&
\end{align*}
where $\{i_1,i_2,i_3,i_4,i_5,i_6\}=\K$.
Note that first order theta derivatives $\{\partial_v \theta[\I_0^{(i)}]\}_{i\in\K}$
involved into \eqref{Th3InvM} form a set of linearly independent vectors.
\end{rem}

\begin{rem}
There exist many representations of the same third order theta derivative  produced with
 different choices of $\K$ and  $j_n$, $j_m \in \J_3 \backslash \K$, all of them are equivalent, 
 cf. Remark~\ref{R:multrepr}
\end{rem}

Based onTheorems~\ref{T:D2thetaGradRepr} and \ref{T:D3thetaGradRepr} 
the following generalization to arbitrary multiplicity $\mFr$ arises.
\begin{conj}\label{C:DNthetaGradRepr}
 Let $\I_0=\{i_1$, $\dots$, $i_{g}\}$ and $\J_0 = \{j_1$, $\dots$, $j_{g+1}\}$ form a partition of 
 the set $\{1$, $2$, $\dots$, $2g+1\}$ of indices of finite branch points of a genus $g$ hyperelliptic curve, 
 and $\I_\mFr = \I_0\backslash \K$, where $\K$ is a set of cardinality $\kFr=2\mFr-1$ or $2\mFr$. Then
\begin{gather}\label{ThNInvM}
 \partial^\mFr_{v_{n_1},\dots,v_{n_\mFr}} \theta[\I_\mFr] 
 = \frac{\epsilon}{\theta[\I_0]^{\mFr-1}} \sum_{k_1,\dots,k_\mFr = 1}^\kFr 
  R_{k_1,\dots,k_\mFr} \prod_{i=1}^\mFr \partial_{v_{n_i}} \theta[\I_0^{(p_{k_i})}],
\end{gather}
where $\I_0^{(p)}=\I_0 \backslash \{p\}$, and $R_{k_1,\dots,k_\mFr}$ form a symmetric tensor of order~$\mFr$  
with vanishing diagonal entries, that is entries with two coinciding indices equal zero, 
and off-diagonal entries are defined as follows
with arbitrary $j_n$, $j_m\in\J_0$
\begin{itemize}
 \item when $\kFr=2\mFr-1$, so that $\{q_1,\dots,q_{\mFr-1}\}=\K\backslash\{p_{k_1}$, \ldots, $p_{k_\mFr}\}$
\begin{multline}\label{RNoM} 
 R_{k_1,\dots,k_\mFr} = 
 \frac{(-1)^{p_{k_1}+ \cdots + p_{k_\mFr}}}
 {\big(\theta[\J_0^{(j_m)}]\theta[\J_0^{(j_n)}]\big)^{\kFr-\mFr}} 
 \prod_{l>i=1}^{\mFr} \theta[\I_0^{(p_{k_i},p_{k_l} \to j_n,j_m)}] \times \\ \times 
 \prod_{l>i=1}^{\kFr-\mFr} \theta[\I_0^{(q_i,q_l \to j_n,j_m)}] 
 \prod_{l=1}^{\kFr-\mFr} \frac{\theta[\I_0^{(q_l\to j_m)}] \theta[\I_0^{(q_l\to j_n)}] \theta[\J_0^{(j_n,j_m \to q_l)}]}
 {\prod_{i=1}^\mFr \theta[\I_0^{(p_{k_i},q_l \to j_n,j_m)}]},
\end{multline}
\item when $\kFr=2\mFr$, so that $\{q_1,\dots,q_{\mFr}\}=\K\backslash\{p_{k_1}$, \ldots, $p_{k_\mFr}\}$,
\begin{multline}\label{RNeM} 
 R_{k_1,\dots,k_\mFr} = 
 \frac{(-1)^{p_{k_1}+ \cdots + p_{k_\mFr}}}
 {\big(\theta[\J_0^{(j_m)}]\theta[\J_0^{(j_n)}]\big)^{\kFr-\mFr}} 
  \prod_{l>i=1}^{\mFr} \theta[\I_0^{(p_{k_i},p_{k_l} \to j_n,j_m)}] \times \\ \times 
  \prod_{l>i=1}^{\kFr-\mFr} \theta[\I_0^{(q_i,q_l \to j_n,j_m)}] 
 \prod_{i=1}^\mFr \theta[\J_0^{(j_n,j_m \to p_{k_i})}] 
 \prod_{l=1}^{\kFr-\mFr} \frac{\theta[\I_0^{(q_l\to j_m)}] \theta[\I_0^{(q_l\to j_n)}]}
 {\prod_{i=1}^\mFr \theta[\I_0^{(p_{k_i},q_l \to j_n,j_m)}]} .
\end{multline}
\end{itemize}
The multiplier  $\epsilon$ equals $1$ or $-1$, and this value is the same for all partitions $\I_\mFr \cup \J_\mFr$
with fixed $\mFr$.
\end{conj}

\section{Schottky type relation}\label{s:SchottkyRel}
Returning to second order theta derivatives, we find out a relation similar to the Schottky identity. 

\begin{prop}\label{P:RelG}
Let $\I_0=\{i_1,\dots,i_{g-4}$, $p_1$, $p_2$, $p_3$, $p_4\}$, 
$\J_0 = \{j_1,\dots, j_{g+1}\}$, and $\I_0 \cup \J_0$ form a partition of 
the set $\{1$, $2$, $\dots$, $2g+1\}$ of indices of finite branch points of a genus $g\geqslant 4$ hyperelliptic curve. 
With an arbitrary choice of the partition  and an arbitrary pair $j_m$, $j_n\in \J_0$ 
the following relation holds
\begin{multline}\label{SchottkyR}
 \theta[\I_0^{(p_1,p_2\to j_m,j_n)}]^8 \theta[\I_0^{(p_3,p_4 \to j_m,j_n)}]^8 \\
 + \theta[\I_0^{(p_1,p_3\to j_m,j_n)}]^8 \theta[\I_0^{(p_2,p_4 \to j_m,j_n)}]^8 
 + \theta[\I_0^{(p_1,p_4\to j_m,j_n)}]^8 \theta[\I_0^{(p_2,p_3 \to j_m,j_n)}]^8 \\
 - 2\Big(\theta[\I_0^{(p_1,p_2\to j_m,j_n)}]^4 \theta[\I_0^{(p_3,p_4 \to j_m,j_n)}]^4
 \theta[\I_0^{(p_1,p_3\to j_m,j_n)}]^4 \theta[\I_0^{(p_2,p_4 \to j_m,j_n)}]^4 \\
 + \theta[\I_0^{(p_1,p_2\to j_m,j_n)}]^4 \theta[\I_0^{(p_3,p_4 \to j_m,j_n)}]^4
 \theta[\I_0^{(p_1,p_4\to j_m,j_n)}]^4 \theta[\I_0^{(p_2,p_3 \to j_m,j_n)}]^4 \\
 + \theta[\I_0^{(p_1,p_3\to j_m,j_n)}]^4 \theta[\I_0^{(p_2,p_4 \to j_m,j_n)}]^4
 \theta[\I_0^{(p_1,p_4\to j_m,j_n)}]^4 \theta[\I_0^{(p_2,p_3 \to j_m,j_n)}]^4 \Big) = 0.
\end{multline}
\end{prop}
\begin{proof}
This follows immediately from the fact that
\begin{gather}\label{R4vanish}
 \det \hat{R} = 0,
\end{gather}
where $\hat{R}$ is $4\times 4$ matrix with entries defined 
by \eqref{R4M} with the set $\K = \{p_1$, $p_2$, $p_3$, $p_4\}$ of dropped indices. 
To prove \eqref{R4vanish} we recall another representation of the entries $R_{k,l}$ of matrix $\hat{R}$,
namely the representation
in terms of branch points, which is given by \eqref{thomae3GradK4} 
where $\{\kappa_1,\kappa_2\}$ is the pair of indices $\{p_k,p_l\}$ 
from $\{p_1,p_2,p_3,p_4\}$, and $\{\kappa_3,\kappa_4\} = \K\backslash \{p_k,p_l\}$.
We suppose that $p_1<p_2<p_3<p_4$, and indices of branch points in factors $(e_l-e_k)$ 
are ordered normally. 

Substituting the mentioned representation \eqref{thomae3GradK4} for $R_{k,l}$  into
\begin{multline*}
 \det \hat{R} = R_{1,2}^2 R_{3,4}^2 + R_{1,3}^2 R_{2,4}^2 + R_{1,4}^2 R_{2,3}^2 \\ 
 - 2 \big(R_{1,2} R_{3,4} R_{1,3} R_{2,4}
 + R_{1,2} R_{3,4} R_{1,4} R_{2,3} + R_{1,3} R_{2,4} R_{1,4} R_{3,4}\big)
\end{multline*}
 one obtains
\begin{gather}\label{Ra}
c  \det \hat{R} = a_1^4 + a_2^4 + a_3^4
 - 2 \big(a_1^2 a_2^2 + a_1^2 a_3^2 + a_2^2 a_3^2 \big)
\end{gather}
with a constant multiple $c$ and
\begin{gather}\label{aDefs}
 \begin{split}
 &a_1 = (e_{p_2}-e_{p_1})(e_{p_4}-e_{p_3}),\\
 &a_2 = (e_{p_3}-e_{p_1})(e_{p_4}-e_{p_2}),\\
 &a_3 = (e_{p_4}-e_{p_1})(e_{p_3}-e_{p_2}).
 \end{split}
\end{gather}
The right hand side of \eqref{Ra} equals
\begin{equation*}
(a_1+a_2+a_3)(a_1+a_2-a_3)(a_1-a_2+a_3)(a_1-a_2-a_3)=0,
\end{equation*}
due to the identity
\begin{gather*}
 a_1-a_2+a_3 = 0.
\end{gather*}
\end{proof}

\begin{rem}
By straightforward computation one can check that all $3\times 3$ minors of $\hat{R}$ do not vanish.
Therefore, in the case of four dropped indices matrix $\hat{R}$ in representation \eqref{TD2ExprM} 
of $\partial_v^2 \theta[\I_2]$ has rank $3$. By the Cauchy-Binet formula this implies that 
the Hesse matrix $\partial_v^2 \theta[\I_2]$ also has rank $3$, that 
coincides with the statement of Theorem~\ref{T:DetD2theta}. 
\end{rem}

The reader can notice that \eqref{SchottkyR} has the form of the \emph{Schottky invariant}~$J$, 
see \cite[p.\,341--342]{Sch1888}. According to \cite{Sch1888}, 
\emph{in the abelian group of all half-period characteristics every syzygetic group of rank $3$} (that is of order~$2^3$) 
\emph{possesses three constants $\{r_1,r_2,r_3\}$ 
such that 
\begin{gather}\label{SchottkyEqlt}
 \sqrt{r_1} \pm \sqrt{r_2} \pm \sqrt{r_3} = 0
\end{gather}
with a particular choice of signs,
and the constant
\begin{gather}\label{SchottkyInv}
 J = r_1^2 + r_2^2 + r_3^2 - 2 r_1 r_2 - 2r_1 r_3 - 2 r_2 r_3
\end{gather}
is independent of the way of grouping characteristics.}

Recall that two characteristics $[\alpha]$ and $[\beta]$ are called syzygetic or azygetic according
as $(\alpha^t \beta' - \beta^t \alpha') \modR 2 =0$ or $1$. Characteristics in a set are called 
syzygetic (azygetic) if the characteristics are pairwise syzygetic (azygetic).
After \cite[ch.\,XVII]{bak897} we denote a syzygetic group (called G\"{o}pel group) by $(P)$.
With the help of this group and another three characteristics $A_1$, $A_2$, $A_3$, 
one produces three coset spaces (G\"{o}pel systems)
$A_i+(P) = (A_i P)$, $i=1,2,3$.
Each coset has the property that
every three of its characteristics are syzygetic \cite[ch.\,XVII p.\,490--491]{bak897}. 
Then three constants in the Schottky invariant are defined as follows, see \cite[Eq.\,(3) p.\,340]{Sch1888},
\begin{gather}\label{rSchottkyDef}
 r_i = \prod_{A\in (A_i P)} \theta[A].
\end{gather}

In \cite[Lemma 3 p.\,538]{Ig1982} Igusa gives a more accurate statement: \emph{We choose an 
even azygetic triplet $\{A_1$, $A_2$, $A_3\}$ in the space of all characteristics, a group $(P)$ of
rank $3$ such that elements of $(A_1 P)$, $(A_2 P)$, $(A_3 P)$  are all even and put \eqref{rSchottkyDef}
for $1\leqslant i \leqslant 3$; then $J$ as in \eqref{SchottkyInv} depends neither on the triplet $\{A_1$, $A_2$, $A_3\}$
nor on $(P)$.} In fact, group $(P)$ is not required to be syzygetic, 
but all characteristics of $(A_1 P)$, $(A_2 P)$, $(A_3 P)$ 
should be even. However, in the hyperelliptic case these characteristics should be even non-singular,
otherwise the products vanish, and relation \eqref{SchottkyEqlt} becomes trivial.

Relation \eqref{SchottkyR} has the form of the Schottky invariant $J$ with $\{r_i\}$ 
produced by a group $(P)$ of rank $1$ (and order $2$) 
consisting of zero characteristic $[\varepsilon_{0}]$, 
and characteristic $[\mathcal{A}(\K)]$ of point $\mathcal{A}(\K) \in \Jac$
with $\K = \{p_1$, $p_2$, $p_3$, $p_4\} \subset \I_0$. 
Three characteristics which give rise to the coset spaces are the following
\begin{gather}\label{Achars}
\begin{split}
 &A_1 = [\I_0^{(p_1,p_2 \to j_m, j_n)}],\\
 &A_2 = [\I_0^{(p_1,p_3 \to j_m, j_n)}],\\
 &A_3 = [\I_0^{(p_1,p_4 \to j_m, j_n)}]
\end{split}
\end{gather}
with the arbitrary choice of $j_m$,  $j_n \in \J_0$.
Since constants $r_i$ have degree $8$, the
products $\prod_{A\in (A_i P)} \theta[A]$  are taken to the $4$-th power.
Relation \eqref{SchottkyR} holds in any genus, and proves that 
the Schottky invariant $J$ produced by a subgroup $(P)$ of rank $1$ vanishes in the hyperelliptic case.

\begin{rem}\label{r:rSchottky}
From \eqref{aDefs} with the help of FTT Corollary~\ref{C:eklm}
one can produce products of $2^3$ theta functions as in \eqref{rSchottkyDef}, and
construct the true Schottky invariant. Let a group $(P)$ be generated by
three characteristics, one of which is necessarily
$[\mathcal{A}(\K)]$. The two others can be taken in the form $[\mathcal{A}(\{q_1,j_m\})]$ and 
$[\mathcal{A}(\{q_2,j_n\})]$, where $q_1,q_2 \in \J_0\backslash \{j_m,j_n\}$, $q_1\neq q_2$. 
One or two of the elements $[\mathcal{A}(\{q_1,j_m\})]$
and $[\mathcal{A}(\{q_2,j_n\})]$ can be replaced by 
$[\mathcal{A}(\{\iota_1,q_1\})]$ and $[\mathcal{A}(\{\iota_2,q_2\})]$ with  
$\iota_1, \iota_2 \in\I_0\backslash \K$, $\iota_1\neq \iota_2$ 
and $q_1,q_2 \in \J_0\backslash \{j_m,j_n\}$, $q_1\neq q_2$. 
This choice of generators of $(P)$ guarantees that
all elements of $(A_1P)$, $(A_2P)$, $(A_3P)$ with $A_1$, $A_2$, $A_3$ defined by \eqref{Achars} are even non-singular.
The idea of extension of group $(P)$ comes from the fact that
the ratio of differences of branch points on the left hand side of \eqref{eklm} depends only on three indices
and can be obtained from  different ratios of theta constants.

Then constants $\{r_1,r_2,r_3\}$ are constructed by \eqref{rSchottkyDef}.
With $(P)$ generated by $[\mathcal{A}(\{p_1,p_2,p_3,p_4\})]$, $[\mathcal{A}(\{q_1,j_m\})]$, 
$[\mathcal{A}(\{q_2,j_n\})]$, and $A_1 = [\I_0^{(p_1,p_2 \to j_m, j_n)}]$, where  $\I_0 \supset \K$, 
the following product over characteristics from $(A_1 P)$ is obtained
\begin{multline}\label{TrueSchottky}
 r_1 = \theta[\I_0^{(p_1,p_2\to j_m,j_n)}] \theta[\I_0^{(p_3,p_4\to j_m,j_n)}]
 \theta[\I_0^{(p_1,p_2\to q_1,j_n)}] \theta[\I_0^{(p_3,p_4\to q_1,j_n)}]\times \\ \times
 \theta[\I_0^{(p_1,p_2\to j_m,q_2)}] \theta[\I_0^{(p_3,p_4\to j_m,q_2)}] 
 \theta[\I_0^{(p_1,p_2\to q_1,q_2)}] \theta[\I_0^{(p_3,p_4\to q_1,q_2)}].
\end{multline}
Similar products for $r_2$ and $r_3$ are taken over characteristics from $(A_2 P)$ and $(A_3 P)$,
where $A_2$ and $A_3$  are given by \eqref{Achars}.  
It is straightforward to verify that 
\begin{gather*}
 \sqrt{r_1} - \sqrt{r_2} + \sqrt{r_3} = 0
\end{gather*}
with $p_1<p_2<p_3<p_4$, due to $\sqrt{r_i}=a_i$, where $\{a_1,a_2,a_3\}$ are defined by \eqref{aDefs}.
Thus, $J=0$.

This way of obtaining the Schottky relation is applicable to arbitrary genus in the hyperelliptic case.
A comparison with examples of the Schottky relations from \cite{FR1969Schtt} and \cite{FR1970Schtt} 
are given in Appendix~\ref{A:ShottkyComp}.
At the same time, the present result differs from the proposed in \cite{FR1970Schtt}  
generalisation on the base
of equality of ratios of the Schottky and the Riemann theta constants (in genera $g-1$ and $g$),
and provides another way of generalisations of the Schottky relation, see Appendices~\ref{A:SchottkyRel5} 
and \ref{A:SchottkyRank4}.
\end{rem}

\begin{rem}
In \cite[Lemma 3 p.\,538]{Ig1982} the three characteristics $A_1$, $A_2$, $A_3$
are specified to be an azygetic triplet.
So we can deduce that $[\I_0^{(p_1,p_2 \to j_m, j_n)}]$, $[\I_0^{(p_1,p_3 \to j_m, j_n)}]$, 
$[\I_0^{(p_1,p_4 \to j_m, j_n)}]$, as defined by \eqref{Achars},
 form an azygetic triplet.
Note that these characteristics are even non-singular by the construction.
\end{rem}

\section{Conclusion and discussion}
The relations obtained in the present paper shed light on the heap of theta derivative identities.
From Subsection~\ref{ss:FDT} one can see that among all first order theta derivatives 
there exists a minimal set of gradient vectors serving as a basis.
As shown in Subsetions~\ref{ss:2derThetaC} and \ref{ss:3derThetaC}
all higher order theta derivatives are expressed in terms of
 first order theta derivatives and theta constants.

The results are summarised as follows. Let $\I_0\cup \J_0$ 
with $\I_0=\{i_1$, $\dots$, $i_g\}$ and $\J_0=\{j_1$, $\dots$, $j_{g+1}\}$ be a partition 
of the set $\{1$, $2$, $\dots$, $2g+1\}$  of indices of finite branch points of a
genus $g$ hyperelliptic curve. Among all gradient vectors $\partial_v \theta$ of null values of 
theta functions with half-period characteristics a basis of $g$ linear independent vectors can be chosen.
For example, vectors $\{\partial_v \theta [\I_0^{(\kappa_i)}] \mid \kappa_i\in \I_0,\, i=1,\dots,g \}$, 
where $\I_0^{(\kappa)} = \I_0 \backslash \{\kappa\}$, form such a basis.
All other gradients $\partial_v \theta[\I_1]$ with characteristics of multiplicity $1$ 
are expressed as linear combinations of the basis vectors, and coefficients
have the form of ratios of theta constants (see Propositions~\ref{P:thomaeK2}, \ref{P:D1thetaLD3}', 
\ref{P:D1thetaLD4} and Conjecture~\ref{C:D1thetaLDH}).  Also the question about the rank of 
a collection of such vectors is elucidated (Theorem~\ref{T:D1thetaLD}).

Next, second order theta derivatives are expressed as
symmetric bilinear forms on the vector space of gradients (Theorem~\ref{T:D2thetaGradRepr}).
Third order theta derivatives are expressed as
symmetric trilinear forms on the vector space of gradients (Theorem~\ref{T:D3thetaGradRepr}).
Conjecture~\ref{C:DNthetaGradRepr} extends this result to 
 higher order theta derivatives.
All these expressions are called here  \emph{generalizations of Jacobi's derivative formula}.

However, Jacobi's derivative formula \eqref{JDF} does not appear 
as a particular case in this stream of relations. Jacobi's derivative formula arises as
genus $1$ case of the Riemann-Jacobi derivative formula. And the latter follows from the second Thomae formula. 
On the other hand,  new relations are derived from the general Thomae formula for
higher order theta derivatives, which express higher order theta derivatives through 
first order theta derivatives and theta constants. So the obtained relation allow to decrease 
the order of theta derivatives similar to
Jacobi's derivative formula.

The proposed relations serve as  the basic material to produce further relations,
known and unknown. This is an objective for a future research. As an example, the Schottky relation
in the hyperelliptic case is derived (Proposition~\ref{P:RelG} and Remark~\ref{r:rSchottky}). 
It is obtained in genus $4$ and generalised to an arbitrary genus higher than~$4$.
This result is in a good correspondence with the examples given in \cite{FR1969Schtt} and \cite{FR1970Schtt}.
Equalities of ratios of the Schottky and the Riemann theta constants are confirmed in genera $3$-$4$ and $4$-$5$,
see Apendix~\ref{A:Schottky}.
In addition, a way of constructing the related Schottky and Riemann theta constants 
(in genera $g-1$ and $g$ respectively) is clarified and explained in terms of partitions corresponding to characteristics.

\appendix
\section{First order theta derivative relations}\label{A:G2ThetaRelK2}
\subsection{Genus $2$}
Below a list of relations between vectors of first order theta derivatives in genus $2$
is presented
(for brevity the notation $\theta^{\{\kappa_1,\kappa_2\}}$ 
for $\theta[\{\kappa_1,\kappa_2\}](0;\tau)$ is used)
\begin{align*}
 \partial_{v} \theta^{\emptyset} 
 &= \frac{\theta^{\{2,3\}}\theta^{\{2,4\}}\theta^{\{2,5\}}}
 {\theta^{\{3,4\}}\theta^{\{3,5\}}\theta^{\{4,5\}}} \partial_{v} \theta^{\{1\}} 
 - \frac{\theta^{\{1,3\}}\theta^{\{1,4\}}\theta^{\{1,5\}}}
 {\theta^{\{3,4\}}\theta^{\{3,5\}}\theta^{\{4,5\}}} \partial_{v} \theta^{\{2\}} \\
 &= \frac{\theta^{\{1,3\}}\theta^{\{3,4\}}\theta^{\{3,5\}}}
 {\theta^{\{1,4\}}\theta^{\{1,5\}}\theta^{\{4,5\}}} \partial_{v} \theta^{\{2\}} 
 - \frac{\theta^{\{1,2\}}\theta^{\{2,4\}}\theta^{\{2,5\}}}
 {\theta^{\{1,4\}}\theta^{\{1,5\}}\theta^{\{4,5\}}} \partial_{v} \theta^{\{3\}} \\
  &= \frac{\theta^{\{1,4\}}\theta^{\{2,4\}}\theta^{\{4,5\}}}
 {\theta^{\{1,2\}}\theta^{\{1,5\}}\theta^{\{2,5\}}} \partial_{v} \theta^{\{3\}} 
 - \frac{\theta^{\{1,3\}}\theta^{\{2,3\}}\theta^{\{3,5\}}}
 {\theta^{\{1,2\}}\theta^{\{1,5\}}\theta^{\{2,5\}}} \partial_{v} \theta^{\{4\}} \\
 &= \frac{\theta^{\{1,5\}}\theta^{\{2,5\}}\theta^{\{3,5\}}}
 {\theta^{\{1,2\}}\theta^{\{1,3\}}\theta^{\{2,3\}}} \partial_{v} \theta^{\{4\}} 
 - \frac{\theta^{\{1,4\}}\theta^{\{2,4\}}\theta^{\{3,4\}}}
 {\theta^{\{1,2\}}\theta^{\{1,3\}}\theta^{\{2,3\}}} \partial_{v} \theta^{\{5\}},
 \end{align*}
 and the same with the standard representation of characteristics
\begin{align*}
  \partial_{v} \theta[{}^{11}_{01}] 
 &= \big(\theta[{}^{11}_{00}]\theta[{}^{10}_{00}]\theta[{}^{10}_{01}]\big)^{-1} \Big(
  \theta[{}^{00}_{01}]\theta[{}^{00}_{00}]\theta[{}^{01}_{00}]\partial_{v} \theta[{}^{01}_{01}] - 
  \theta[{}^{00}_{11}]\theta[{}^{00}_{10}]\theta[{}^{01}_{10}]\partial_{v} \theta[{}^{01}_{11}] \Big) \\
 &= \big(\theta[{}^{00}_{10}]\theta[{}^{01}_{10}]\theta[{}^{10}_{01}]\big)^{-1} \Big(
  \theta[{}^{00}_{11}]\theta[{}^{11}_{00}]\theta[{}^{10}_{00}]\partial_{v} \theta[{}^{01}_{11}] - 
  \theta[{}^{11}_{11}]\theta[{}^{00}_{00}]\theta[{}^{01}_{00}]\partial_{v} \theta[{}^{10}_{11}] \Big) \\
 &= \big(\theta[{}^{11}_{11}]\theta[{}^{01}_{10}]\theta[{}^{01}_{00}]\big)^{-1} \Big(
  \theta[{}^{00}_{10}]\theta[{}^{00}_{00}]\theta[{}^{10}_{01}]\partial_{v} \theta[{}^{10}_{11}] - 
  \theta[{}^{00}_{11}]\theta[{}^{00}_{01}]\theta[{}^{10}_{00}]\partial_{v} \theta[{}^{10}_{10}] \Big) \\
 &= \big(\theta[{}^{11}_{11}]\theta[{}^{00}_{11}]\theta[{}^{00}_{01}]\big)^{-1} \Big(
  \theta[{}^{01}_{10}]\theta[{}^{01}_{00}]\theta[{}^{10}_{00}]\partial_{v} \theta[{}^{10}_{10}] - 
  \theta[{}^{00}_{10}]\theta[{}^{00}_{00}]\theta[{}^{11}_{00}]\partial_{v} \theta[{}^{11}_{10}] \Big).
  \end{align*}
Note that the matrices of characteristics are constructed from the homology basis 
given in Subsection \ref{ss:CharHyper}.
Any two vectors of first order theta derivatives could serve as a basis, and all other 
vectors are expressed in terms of these two.

\subsection{Genus $3$}
A list of relations between $\partial_{v} \theta^{\{1\}}$ and other 
vectors of first order theta derivatives in genus $3$ is presented below
\begin{align*}
 \partial_{v} \theta^{\{1\}} 
 &= \frac{\theta^{\{1,3,4\}}\theta^{\{1,3,5\}}\theta^{\{3,6,7\}}}
 {\theta^{\{1,4,5\}}\theta^{\{5,6,7\}}\theta^{\{4,6,7\}}}  \partial_{v} \theta^{\{1,2\}} 
 - \frac{\theta^{\{1,2,4\}}\theta^{\{1,2,5\}}\theta^{\{2,6,7\}}}
 {\theta^{\{1,4,5\}}\theta^{\{5,6,7\}}\theta^{\{4,6,7\}}}  \partial_{v} \theta^{\{1,3\}} \\
 &= \frac{\theta^{\{1,4,2\}}\theta^{\{1,4,5\}}\theta^{\{4,6,7\}}}
 {\theta^{\{1,2,5\}}\theta^{\{2,6,7\}}\theta^{\{5,6,7\}}} \partial_{v} \theta^{\{1,3\}} 
 - \frac{\theta^{\{1,3,2\}}\theta^{\{1,3,5\}}\theta^{\{3,6,7\}}}
 {\theta^{\{1,2,5\}}\theta^{\{2,6,7\}}\theta^{\{5,6,7\}}} \partial_{v} \theta^{\{1,4\}} \\
 &= \frac{\theta^{\{1,5,2\}}\theta^{\{1,5,3\}}\theta^{\{5,6,7\}}}
 {\theta^{\{1,2,3\}}\theta^{\{2,6,7\}}\theta^{\{3,6,7\}}} \partial_{v} \theta^{\{1,4\}} 
 - \frac{\theta^{\{1,4,2\}}\theta^{\{1,4,3\}}\theta^{\{4,6,7\}}}
 {\theta^{\{1,2,3\}}\theta^{\{2,6,7\}}\theta^{\{3,6,7\}}} \partial_{v} \theta^{\{1,5\}} \\
 &= \frac{\theta^{\{1,6,2\}}\theta^{\{1,6,3\}}\theta^{\{6,4,7\}}}
 {\theta^{\{1,2,3\}}\theta^{\{2,4,7\}}\theta^{\{3,4,7\}}} \partial_{v} \theta^{\{1,5\}} 
 - \frac{\theta^{\{1,5,2\}}\theta^{\{1,5,3\}}\theta^{\{5,4,7\}}}
 {\theta^{\{1,2,3\}}\theta^{\{2,4,7\}}\theta^{\{3,4,7\}}} \partial_{v} \theta^{\{1,6\}} \\
 &= \frac{\theta^{\{1,7,2\}}\theta^{\{1,7,3\}}\theta^{\{7,4,5\}}}
 {\theta^{\{1,2,3\}}\theta^{\{2,4,5\}}\theta^{\{3,4,5\}}} \partial_{v} \theta^{\{1,6\}} 
 - \frac{\theta^{\{1,6,2\}}\theta^{\{1,6,3\}}\theta^{\{6,4,5\}}}
 {\theta^{\{1,2,3\}}\theta^{\{2,4,5\}}\theta^{\{3,4,5\}}} \partial_{v} \theta^{\{1,7\}}, 
\end{align*}
and the same with the standard representation of characteristics
 \begin{align*}
 &\partial_{v} \theta[{}^{011}_{101}] \\
 &= \big(\theta[{}^{000}_{101}]\theta[{}^{111}_{011}]\theta[{}^{100}_{011}]\big)^{-1}
 \Big(\theta[{}^{011}_{111}]\theta[{}^{000}_{111}]\theta[{}^{100}_{001}]
 \partial_{v} \theta[{}^{111}_{001}] 
 - \theta[{}^{101}_{111}]\theta[{}^{110}_{111}]\theta[{}^{010}_{001}]
 \partial_{v} \theta[{}^{001}_{001}] \Big)\\
 &= \big(\theta[{}^{000}_{110}]\theta[{}^{111}_{000}]\theta[{}^{100}_{011}]\big)^{-1} 
 \Big(\theta[{}^{101}_{111}]\theta[{}^{000}_{101}]\theta[{}^{100}_{011}]
 \partial_{v} \theta[{}^{001}_{001}] 
 - \theta[{}^{101}_{101}]\theta[{}^{110}_{110}]\theta[{}^{001}_{000}]
 \partial_{v} \theta[{}^{001}_{011}] \Big)\\
 &= \big(\theta[{}^{101}_{101}]\theta[{}^{010}_{001}]\theta[{}^{100}_{001}]\big)^{-1} 
 \Big(\theta[{}^{110}_{111}]\theta[{}^{000}_{111}]\theta[{}^{111}_{011}]
 \partial_{v} \theta[{}^{001}_{011}] 
 - \theta[{}^{101}_{111}]\theta[{}^{011}_{111}]\theta[{}^{100}_{011}]
 \partial_{v} \theta[{}^{010}_{011}] \Big)\\
 &= \big(\theta[{}^{101}_{101}]\theta[{}^{001}_{000}]\theta[{}^{111}_{000}]\big)^{-1} 
 \Big(\theta[{}^{110}_{110}]\theta[{}^{000}_{110}]\theta[{}^{100}_{011}]
 \partial_{v} \theta[{}^{010}_{011}] 
 - \theta[{}^{110}_{111}]\theta[{}^{000}_{111}]\theta[{}^{100}_{010}]
 \partial_{v} \theta[{}^{010}_{010}] \Big)\\
 &= \big(\theta[{}^{110}_{111}]\theta[{}^{010}_{001}]\theta[{}^{111}_{011}]\big)^{-1} 
 \Big(\theta[{}^{111}_{110}]\theta[{}^{001}_{110}]\theta[{}^{100}_{010}]
 \partial_{v} \theta[{}^{010}_{010}] 
 - \theta[{}^{110}_{110}]\theta[{}^{000}_{110}]\theta[{}^{101}_{010}]
 \partial_{v} \theta[{}^{011}_{010}] \Big).
\end{align*}

\section{Proof of Proposition~\ref{P:D1thetaLD3}}\label{A:D1thetaLD3}
\begin{proof}
Applying \eqref{ThGInf} to decompositions of $\partial_v \theta[\I]$ into three pairs:
1) $\partial_{v} \theta[\I\cup \{\kappa_1\}]$ and $\partial_{v} \theta[\I\cup \{\kappa_2\}]$, 
2) $\partial_{v} \theta[\I\cup \{\kappa_1\}]$ and $\partial_{v} \theta[\I\cup \{\kappa_3\}]$,
3) $\partial_{v} \theta[\I\cup \{\kappa_2\}]$ and $\partial_{v} \theta[\I\cup \{\kappa_3\}]$
one obtains three equalities:
\begin{multline*}
 \partial_v \theta[\I] = 
 \big(\theta[\I\cup\{j_m,j_n\}]\theta[\J^{(j_n\to \kappa_3)}]\theta[\J^{(j_m\to \kappa_3)}]\big)^{-1} \times \\
 \times \Big( \theta[\I\cup\{\kappa_2,j_m\}]\theta[\I\cup\{\kappa_2,j_n\}]
 \theta[\J^{(j_m,j_n \to \kappa_2,\kappa_3)}]  \partial_{v} \theta[\I\cup \{\kappa_1\}] \\
 - \theta[\I\cup\{\kappa_1,j_m\}]\theta[\I\cup\{\kappa_1,j_n\}]
 \theta[\J^{(j_m,j_n \to \kappa_1,\kappa_3)}] \partial_{v} \theta[\I\cup \{\kappa_2\}] \Big) \\
 = \big(\theta[\I\cup\{j_m,j_n\}]\theta[\J^{(j_n\to \kappa_2)}]\theta[\J^{(j_m\to \kappa_2)}]\big)^{-1} \times \\
 \times \Big( \theta[\I\cup\{\kappa_3,j_m\}]\theta[\I\cup\{\kappa_3,j_n\}]
 \theta[\J^{(j_m,j_n \to \kappa_2,\kappa_3)}]  \partial_{v} \theta[\I\cup \{\kappa_1\}] \\
 - \theta[\I\cup\{\kappa_1,j_m\}]\theta[\I\cup\{\kappa_1,j_n\}]
 \theta[\J^{(j_m,j_n \to \kappa_1,\kappa_2)}] \partial_{v} \theta[\I\cup \{\kappa_3\}] \Big) \\
 = \big(\theta[\I\cup\{j_m,j_n\}]\theta[\J^{(j_n\to \kappa_1)}]\theta[\J^{(j_m\to \kappa_1)}]\big)^{-1} \times \\
 \times \Big( \theta[\I\cup\{\kappa_3,j_m\}]\theta[\I\cup\{\kappa_3,j_n\}]
 \theta[\J^{(j_m,j_n \to \kappa_1,\kappa_3)}]  \partial_{v} \theta[\I\cup \{\kappa_2\}] \\
 - \theta[\I\cup\{\kappa_2,j_m\}]\theta[\I\cup\{\kappa_2,j_n\}]
 \theta[\J^{(j_m,j_n \to \kappa_1,\kappa_2)}] \partial_{v} \theta[\I\cup \{\kappa_3\}] \Big).
\end{multline*} 
They produce two relations between vectors $\partial_{v} \theta[\I\cup \{\kappa_1\}]$, 
$\partial_{v} \theta[\I\cup \{\kappa_2\}]$, 
and $\partial_{v} \theta[\I\cup \{\kappa_3\}]$, and these two relations are equivalent, so one gets \eqref{D1thetaLD}
from Proposition~\ref{P:D1thetaLD3}. Indeed, extract the two relations and solve them for $\partial_{v} \theta[\I_1\cup \{\kappa_3\}]$
\begin{multline}\label{D1ThetaLDi}
 \partial_{v} \theta[\I\cup \{\kappa_3\}]  
 =  \frac{\theta[\J^{(j_n\to \kappa_1)}]\theta[\J^{(j_m\to \kappa_1)}]
 \theta[\J^{(j_m,j_n \to \kappa_2,\kappa_3)}]}
 {\theta[\J^{(j_n\to \kappa_3)}]\theta[\J^{(j_m\to \kappa_3)}]\theta[\J^{(j_m,j_n \to \kappa_1,\kappa_2)}]} \\
 \bigg(\frac{\theta[\J^{(j_n\to \kappa_3)}]\theta[\J^{(j_m\to \kappa_3)}]
 \theta[\I\cup\{\kappa_3,j_m\}]\theta[\I\cup\{\kappa_3,j_n\}]}
 {\theta[\J^{(j_n\to \kappa_1)}]\theta[\J^{(j_m\to \kappa_1)}]
 \theta[\I\cup\{\kappa_1,j_m\}]\theta[\I\cup\{\kappa_1,j_n\}]} \\
 - \frac{\theta[\J^{(j_n\to \kappa_2)}]\theta[\J^{(j_m\to \kappa_2)}]
 \theta[\I\cup\{\kappa_2,j_m\}]\theta[\I\cup\{\kappa_2,j_n\}]}
 {\theta[\J^{(j_n\to \kappa_1)}]\theta[\J^{(j_m\to \kappa_1)}]
 \theta[\I\cup\{\kappa_1,j_m\}] \theta[\I\cup\{\kappa_1,j_n\}]} \bigg) 
 \partial_{v} \theta[\I\cup \{\kappa_1\}] \\
 +  \frac{\theta[\J^{(j_n\to \kappa_2)}]\theta[\J^{(j_m\to \kappa_2)}]\theta[\J^{(j_m,j_n \to \kappa_1,\kappa_3)}]}
 {\theta[\J^{(j_n\to \kappa_3)}]\theta[\J^{(j_m\to \kappa_3)}]\theta[\J^{(j_m,j_n \to \kappa_1,\kappa_2)}]}
 \partial_{v} \theta[\I\cup \{\kappa_2\}] \\
 = - \frac{\theta[\J^{(j_n\to \kappa_1)}]\theta[\J^{(j_m\to \kappa_1)}]\theta[\J^{(j_m,j_n \to \kappa_2,\kappa_3)}]}
 {\theta[\J^{(j_n\to \kappa_3)}]\theta[\J^{(j_m\to \kappa_3)}]\theta[\J^{(j_m,j_n \to \kappa_1,\kappa_2)}]}
  \partial_{v} \theta[\I\cup \{\kappa_1\}] \\
 + \frac{\theta[\J^{(j_n\to \kappa_2)}]\theta[\J^{(j_m\to \kappa_2)}]\theta[\J^{(j_m,j_n \to \kappa_1,\kappa_3)}]}
 {\theta[\J^{(j_n\to \kappa_3)}]\theta[\J^{(j_m\to \kappa_3)}]\theta[\J^{(j_m,j_n \to \kappa_1,\kappa_2)}]} \\
 \bigg(\frac{\theta[\J^{(j_n\to \kappa_3)}]\theta[\J^{(j_m\to \kappa_3)}]
 \theta[\I\cup\{\kappa_3,j_m\}]\theta[\I\cup\{\kappa_3,j_n\}]}
 {\theta[\J^{(j_n\to \kappa_2)}]\theta[\J^{(j_m\to \kappa_2)}]
 \theta[\I\cup\{\kappa_2,j_m\}]\theta[\I\cup\{\kappa_2,j_n\}]} \\
 + \frac{\theta[\J^{(j_n\to \kappa_1)}]\theta[\J^{(j_m\to \kappa_1)}]
 \theta[\I\cup\{\kappa_1,j_m\}]\theta[\I\cup\{\kappa_1,j_n\}]}
 {\theta[\J^{(j_n\to \kappa_2)}]\theta[\J^{(j_m\to \kappa_2)}]
 \theta[\I\cup\{\kappa_2,j_m\}]\theta[\I\cup\{\kappa_2,j_n\}]} \bigg) 
 \partial_{v} \theta[\I\cup \{\kappa_2\}].
\end{multline}
This implies a relation between $\partial_{v} \theta[\I\cup \{\kappa_1\}]$ and $\partial_{v} \theta[\I\cup \{\kappa_2\}]$
\begin{multline*}
 \frac{\theta[\J^{(j_n\to \kappa_1)}]\theta[\J^{(j_m\to \kappa_1)}]
 \theta[\J^{(j_m,j_n \to \kappa_2,\kappa_3)}]}
 {\theta[\J^{(j_n\to \kappa_3)}]\theta[\J^{(j_m\to \kappa_3)}]\theta[\J^{(j_m,j_n \to \kappa_1,\kappa_2)}]} \\
 \bigg(\frac{\theta[\J^{(j_n\to \kappa_3)}]\theta[\J^{(j_m\to \kappa_3)}]
 \theta[\I\cup\{\kappa_3,j_m\}]\theta[\I\cup\{\kappa_3,j_n\}]}
 {\theta[\J^{(j_n\to \kappa_1)}]\theta[\J^{(j_m\to \kappa_1)}]
 \theta[\I\cup\{\kappa_1,j_m\}]\theta[\I\cup\{\kappa_1,j_n\}]} \\
 - \frac{\theta[\J^{(j_n\to \kappa_2)}]\theta[\J^{(j_m\to \kappa_2)}]
 \theta[\I\cup\{\kappa_2,j_m\}]\theta[\I\cup\{\kappa_2,j_n\}]}
 {\theta[\J^{(j_n\to \kappa_1)}]\theta[\J^{(j_m\to \kappa_1)}]
 \theta[\I\cup\{\kappa_1,j_m\}] \theta[\I\cup\{\kappa_1,j_n\}]} + 1\bigg) 
 \partial_{v} \theta[\I\cup \{\kappa_1\}] \\
 =  \frac{\theta[\J^{(j_n\to \kappa_2)}]\theta[\J^{(j_m\to \kappa_2)}]\theta[\J^{(j_m,j_n \to \kappa_1,\kappa_3)}]}
 {\theta[\J^{(j_n\to \kappa_3)}]\theta[\J^{(j_m\to \kappa_3)}]\theta[\J^{(j_m,j_n \to \kappa_1,\kappa_2)}]} \\
 \bigg(\frac{\theta[\J^{(j_n\to \kappa_3)}]\theta[\J^{(j_m\to \kappa_3)}]
 \theta[\I\cup\{\kappa_3,j_m\}]\theta[\I\cup\{\kappa_3,j_n\}]}
 {\theta[\J^{(j_n\to \kappa_2)}]\theta[\J^{(j_m\to \kappa_2)}]
 \theta[\I\cup\{\kappa_2,j_m\}]\theta[\I\cup\{\kappa_2,j_n\}]} \\
 + \frac{\theta[\J^{(j_n\to \kappa_1)}]\theta[\J^{(j_m\to \kappa_1)}]
 \theta[\I\cup\{\kappa_1,j_m\}]\theta[\I\cup\{\kappa_1,j_n\}]}
 {\theta[\J^{(j_n\to \kappa_2)}]\theta[\J^{(j_m\to \kappa_2)}]
 \theta[\I\cup\{\kappa_2,j_m\}]\theta[\I\cup\{\kappa_2,j_n\}]} - 1\bigg) 
 \partial_{v} \theta[\I\cup \{\kappa_2\}].
\end{multline*}
And this relation is an identity, since the expressions in parentheses vanish identically
due to FTT Corollary~\ref{C:eklm}, namely 
\begin{gather*}
 \frac{\theta[\J^{(j_n\to \kappa_3)}]^2\theta[\I\cup\{\kappa_3,j_n\}]^2}
 {\theta[\J^{(j_n\to \kappa_1)}]^2\theta[\I\cup\{\kappa_1,j_n\}]^2}
 = \frac{\theta[\J^{(j_m\to \kappa_3)}]^2\theta[\I\cup\{\kappa_3,j_m\}]^2}
 {\theta[\J^{(j_m\to \kappa_1)}]^2\theta[\I\cup\{\kappa_1,j_m\}]^2} 
 = \frac{e_{\kappa_2} - e_{\kappa_1}}{e_{\kappa_3} - e_{\kappa_2}},\\
 \frac{\theta[\J^{(j_n\to \kappa_2)}]^2\theta[\I\cup\{\kappa_2,j_n\}]^2}
 {\theta[\J^{(j_n\to \kappa_1)}]^2\theta[\I\cup\{\kappa_1,j_n\}]^2}
 = \frac{\theta[\J^{(j_m\to \kappa_2)}]^2\theta[\I\cup\{\kappa_2,j_m\}]^2}
 {\theta[\J^{(j_m\to \kappa_1)}]^2\theta[\I\cup\{\kappa_1,j_m\}]^2} 
 = \frac{e_{\kappa_3} - e_{\kappa_1}}{e_{\kappa_3} - e_{\kappa_2}},\\
 \frac{\theta[\J^{(j_n\to \kappa_3)}]^2\theta[\I\cup\{\kappa_3,j_n\}]^2}
 {\theta[\J^{(j_n\to \kappa_2)}]^2\theta[\I\cup\{\kappa_2,j_n\}]^2}
 = \frac{\theta[\J^{(j_m\to \kappa_3)}]^2\theta[\I\cup\{\kappa_3,j_m\}]^2}
 {\theta[\J^{(j_m\to \kappa_2)}]^2\theta[\I\cup\{\kappa_2,j_m\}]^2} 
 = \frac{e_{\kappa_2} - e_{\kappa_1}}{e_{\kappa_3} - e_{\kappa_1}}.
\end{gather*}
Thus,
\begin{multline*}
 \frac{\theta[\J^{(j_n\to \kappa_3)}]\theta[\J^{(j_m\to \kappa_3)}]
 \theta[\I\cup\{\kappa_3,j_m\}]\theta[\I\cup\{\kappa_3,j_n\}]}
 {\theta[\J^{(j_n\to \kappa_1)}]\theta[\J^{(j_m\to \kappa_1)}]
 \theta[\I\cup\{\kappa_1,j_m\}]\theta[\I\cup\{\kappa_1,j_n\}]} \\
 - \frac{\theta[\J^{(j_n\to \kappa_2)}]\theta[\J^{(j_m\to \kappa_2)}]
 \theta[\I\cup\{\kappa_2,j_m\}]\theta[\I\cup\{\kappa_2,j_n\}]}
 {\theta[\J^{(j_n\to \kappa_1)}]\theta[\J^{(j_m\to \kappa_1)}]
 \theta[\I\cup\{\kappa_1,j_m\}] \theta[\I\cup\{\kappa_1,j_n\}]} + 1 \\
 = \frac{e_{\kappa_2} - e_{\kappa_1}}{e_{\kappa_3} - e_{\kappa_2}}
 - \frac{e_{\kappa_3} - e_{\kappa_1}}{e_{\kappa_3} - e_{\kappa_2}} + 1 = 0,
\end{multline*}
\begin{multline*}
 \frac{\theta[\J^{(j_n\to \kappa_3)}]\theta[\J^{(j_m\to \kappa_3)}]
 \theta[\I\cup\{\kappa_3,j_m\}]\theta[\I\cup\{\kappa_3,j_n\}]}
 {\theta[\J^{(j_n\to \kappa_2)}]\theta[\J^{(j_m\to \kappa_2)}]
 \theta[\I\cup\{\kappa_2,j_m\}]\theta[\I\cup\{\kappa_2,j_n\}]} \\
 + \frac{\theta[\J^{(j_n\to \kappa_1)}]\theta[\J^{(j_m\to \kappa_1)}]
 \theta[\I\cup\{\kappa_1,j_m\}]\theta[\I\cup\{\kappa_1,j_n\}]}
 {\theta[\J^{(j_n\to \kappa_2)}]\theta[\J^{(j_m\to \kappa_2)}]
 \theta[\I\cup\{\kappa_2,j_m\}]\theta[\I\cup\{\kappa_2,j_n\}]} - 1 \\
 = \frac{e_{\kappa_2} - e_{\kappa_1}}{e_{\kappa_3} - e_{\kappa_1}}
 + \frac{e_{\kappa_3} - e_{\kappa_2}}{e_{\kappa_3} - e_{\kappa_1}} - 1 = 0.
\end{multline*}
Therefore, \eqref{D1ThetaLDi} contains one relation between vectors $\partial_{v} \theta[\I\cup \{\kappa_1\}]$, 
$\partial_{v} \theta[\I\cup \{\kappa_2\}]$, and $\partial_{v} \theta[\I\cup \{\kappa_3\}]$,
and every two of the vectors are linearly independent. The relation simplifies to the form
\begin{multline*}
  \theta[\J^{(j_n\to \kappa_1)}]\theta[\J^{(j_m\to \kappa_1)}]\theta[\J^{(j_m,j_n \to \kappa_2,\kappa_3)}]
  \partial_{v} \theta[\I\cup \{\kappa_1\}] \\
 - \theta[\J^{(j_n\to \kappa_2)}]\theta[\J^{(j_m\to \kappa_2)}]\theta[\J^{(j_m,j_n \to \kappa_1,\kappa_3)}]
 \partial_{v} \theta[\I\cup \{\kappa_2\}] \\
 + \theta[\J^{(j_n\to \kappa_3)}]\theta[\J^{(j_m\to \kappa_3)}]\theta[\J^{(j_m,j_n \to \kappa_1,\kappa_2)}]
  \partial_{v} \theta[\I\cup \{\kappa_3\}] = 0.
\end{multline*}
\end{proof}

As an example, in genus $2$ one finds
\begin{multline*}
  \theta^{\{1,4\}}\theta^{\{1,5\}}\theta^{\{2,3\}} \partial_{v} \theta^{\{1\}}
 - \theta^{\{2,4\}}\theta^{\{2,5\}}\theta^{\{1,3\}} \partial_{v} \theta^{\{2\}} \\
 + \theta^{\{3,4\}}\theta^{\{3,5\}}\theta^{\{1,2\}} \partial_{v} \theta^{\{3\}} = 0,
\end{multline*}
and in genus $3$ 
\begin{multline*}
  \theta^{\{2,6,7\}}\theta^{\{5,2,7\}}\theta^{\{3,4,7\}} \partial_{v} \theta^{\{1,2\}}
 - \theta^{\{3,6,7\}}\theta^{\{5,3,7\}}\theta^{\{2,4,7\}} \partial_{v} \theta^{\{1,3\}} \\
 + \theta^{\{4,6,7\}}\theta^{\{5,4,7\}}\theta^{\{2,3,7\}} \partial_{v} \theta^{\{1,4\}} = 0.
\end{multline*}
Recall that $\theta^{\{i,j\}}$ stands for $\theta[\{i,j\}](0;\tau)$, as well as
$\partial_{v} \theta^{\{i\}}$ stands for $\grad_v \theta[\{i\}](v;\tau) |_{v=0}$.

\section{Proof of Proposition~\ref{P:D1thetaLD4}}\label{A:D1thetaLD4}

\begin{proof}
The set $\I=\{i_1,\,\dots$, $i_{g-3}\}$ corresponds to a characteristic of multiplicity~$2$, 
then joining two indices from $\{\kappa_1,\,\kappa_2,\,\kappa_3,\,\kappa_4,\,\kappa_5\}$
results into a partition $\I\cup \{\kappa_i,\kappa_j\}$  corresponding to a characteristic of multiplicity $1$.
Write down equality \eqref{D1thetaLDG} with the following gradients:
\begin{enumerate}
\renewcommand{\theenumi}{\roman{enumi}}
 \item $\partial_{v} \theta[\I\cup \{\kappa_1,\kappa_2\}]$,
$\partial_{v} \theta[\I\cup \{\kappa_1,\kappa_3\}]$ and $\partial_{v} \theta[\I\cup \{\kappa_1,\kappa_5\}]$
with the common set $\I \cup \{\kappa_1\}$ denoted by $\I$ in \eqref{D1thetaLDG},
 \item $\partial_{v} \theta[\I\cup \{\kappa_1,\kappa_2\}]$,
$\partial_{v} \theta[\I\cup \{\kappa_2,\kappa_3\}]$ and $\partial_{v} \theta[\I\cup \{\kappa_2,\kappa_5\}]$
with the common set $\I \cup \{\kappa_2\}$,
 \item $\partial_{v} \theta[\I\cup \{\kappa_1,\kappa_5\}]$,
$\partial_{v} \theta[\I\cup \{\kappa_2,\kappa_5\}]$ and $\partial_{v} \theta[\I\cup \{\kappa_4,\kappa_5\}]$
with the common set $\I \cup \{\kappa_5\}$,
\end{enumerate}
and obtain three relations, namely:
 \begin{multline*}
  \theta[\J^{(j_n\to \kappa_2,\kappa_4)}]\theta[\J^{(j_m\to \kappa_2,\kappa_4)}]\theta[\J^{(j_m,j_n \to \kappa_3,\kappa_4,\kappa_5)}]
  \partial_{v} \theta[\I\cup \{\kappa_1,\kappa_2\}] \\
 - \theta[\J^{(j_n\to \kappa_3,\kappa_4)}]\theta[\J^{(j_m\to \kappa_3,\kappa_4)}]\theta[\J^{(j_m,j_n \to \kappa_2,\kappa_4,\kappa_5)}]
 \partial_{v} \theta[\I\cup \{\kappa_1,\kappa_3\}] \\
 + \theta[\J^{(j_n\to \kappa_4,\kappa_5)}]\theta[\J^{(j_m\to \kappa_4,\kappa_5)}]\theta[\J^{(j_m,j_n \to \kappa_2,\kappa_3,\kappa_4)}]
  \partial_{v} \theta[\I\cup \{\kappa_1,\kappa_5\}] = 0,
\end{multline*}
 \begin{multline*}
  \theta[\J^{(j_n\to \kappa_1,\kappa_4)}]\theta[\J^{(j_m\to \kappa_1,\kappa_4)}]\theta[\J^{(j_m,j_n \to \kappa_3,\kappa_4,\kappa_5)}]
  \partial_{v} \theta[\I\cup \{\kappa_1,\kappa_2\}] \\
 - \theta[\J^{(j_n\to \kappa_3,\kappa_4)}]\theta[\J^{(j_m\to \kappa_3,\kappa_4)}]\theta[\J^{(j_m,j_n \to \kappa_1,\kappa_4,\kappa_5)}]
 \partial_{v} \theta[\I\cup \{\kappa_2,\kappa_3\}] \\
 + \theta[\J^{(j_n\to \kappa_4,\kappa_5)}]\theta[\J^{(j_m\to \kappa_4,\kappa_5)}]\theta[\J^{(j_m,j_n \to \kappa_1,\kappa_3,\kappa_4)}]
  \partial_{v} \theta[\I\cup \{\kappa_2,\kappa_5\}] = 0,
\end{multline*} 
\begin{multline*}
  \theta[\J^{(j_n\to \kappa_1,\kappa_3)}]\theta[\J^{(j_m\to \kappa_1,\kappa_3)}]\theta[\J^{(j_m,j_n \to \kappa_2,\kappa_3,\kappa_4)}]
  \partial_{v} \theta[\I\cup \{\kappa_1,\kappa_5\}] \\
 - \theta[\J^{(j_n\to \kappa_2,\kappa_3)}]\theta[\J^{(j_m\to \kappa_2,\kappa_3)}]\theta[\J^{(j_m,j_n \to \kappa_1,\kappa_3,\kappa_4)}]
 \partial_{v} \theta[\I\cup \{\kappa_2,\kappa_5\}] \\
 + \theta[\J^{(j_n\to \kappa_3,\kappa_4)}]\theta[\J^{(j_m\to \kappa_3,\kappa_4)}]\theta[\J^{(j_m,j_n \to \kappa_1,\kappa_2,\kappa_3)}]
  \partial_{v} \theta[\I\cup \{\kappa_4,\kappa_5\}] = 0.
\end{multline*}
Next, eliminate vectors $\partial_{v} \theta[\I\cup \{\kappa_1,\kappa_5\}]$ and
$\partial_{v} \theta[\I\cup \{\kappa_2,\kappa_5\}]$ from the relations, and come to
\begin{multline}\label{D1ThetaLD4i}
 \Big( - \theta[\J^{(j_n\to \kappa_1,\kappa_3)}]\theta[\J^{(j_m\to \kappa_1,\kappa_3)}] 
  \theta[\J^{(j_n\to \kappa_2,\kappa_4)}]\theta[\J^{(j_m\to \kappa_2,\kappa_4)}] \\
 + \theta[\J^{(j_n\to \kappa_2,\kappa_3)}]\theta[\J^{(j_m\to \kappa_2,\kappa_3)}]
 \theta[\J^{(j_n\to \kappa_1,\kappa_4)}]\theta[\J^{(j_m\to \kappa_1,\kappa_4)}] \Big) \times \\ \times
 \frac{\theta[\J^{(j_m,j_n \to \kappa_3,\kappa_4,\kappa_5)}]}
 {\theta[\J^{(j_n\to \kappa_3,\kappa_4)}]\theta[\J^{(j_m\to \kappa_3,\kappa_4)}]} 
  \partial_{v} \theta[\I\cup \{\kappa_1,\kappa_2\}] \\
 + \theta[\J^{(j_n\to \kappa_1,\kappa_3)}]\theta[\J^{(j_m\to \kappa_1,\kappa_3)}]
 \theta[\J^{(j_m,j_n \to \kappa_2,\kappa_4,\kappa_5)}]
 \partial_{v} \theta[\I\cup \{\kappa_1,\kappa_3\}] \\
 - \theta[\J^{(j_n\to \kappa_2,\kappa_3)}]\theta[\J^{(j_m\to \kappa_2,\kappa_3)}] 
 \theta[\J^{(j_m,j_n \to \kappa_1,\kappa_4,\kappa_5)}]
 \partial_{v} \theta[\I\cup \{\kappa_2,\kappa_3\}] \\
 + \theta[\J^{(j_n\to \kappa_4,\kappa_5)}]\theta[\J^{(j_m\to \kappa_4,\kappa_5)}] 
 \theta[\J^{(j_m,j_n \to \kappa_1,\kappa_2,\kappa_3)}]
  \partial_{v} \theta[\I\cup \{\kappa_4,\kappa_5\}] = 0.
\end{multline}
With the help of FTT Corollary~\ref{C:eklm} the following holds
\begin{multline*}
  - \frac{\theta[\J^{(j_n\to \kappa_1,\kappa_3)}]\theta[\J^{(j_m\to \kappa_1,\kappa_3)}] 
  \theta[\J^{(j_n\to \kappa_2,\kappa_4)}]\theta[\J^{(j_m\to \kappa_2,\kappa_4)}]}
 {\theta[\J^{(j_n\to \kappa_3,\kappa_4)}]\theta[\J^{(j_m\to \kappa_3,\kappa_4)}]
 \theta[\J^{(j_n\to \kappa_1,\kappa_2)}]\theta[\J^{(j_m\to \kappa_1,\kappa_2)}]} \\
 + \frac{\theta[\J^{(j_n\to \kappa_2,\kappa_3)}]\theta[\J^{(j_m\to \kappa_2,\kappa_3)}]
 \theta[\J^{(j_n\to \kappa_1,\kappa_4)}]\theta[\J^{(j_m\to \kappa_1,\kappa_4)}]}
 {\theta[\J^{(j_n\to \kappa_3,\kappa_4)}]\theta[\J^{(j_m\to \kappa_3,\kappa_4)}]
 \theta[\J^{(j_n\to \kappa_1,\kappa_2)}]\theta[\J^{(j_m\to \kappa_1,\kappa_2)}]} + 1 \\
 = - \frac{\theta[\J^{(j_n\to \kappa_1,\kappa_3)}] \theta[\I\cup\{j_n,\kappa_1,\kappa_5\}] 
  \theta[\J^{(j_m\to \kappa_1,\kappa_3)}] \theta[\I\cup\{j_m,\kappa_1,\kappa_5\}]}
 {\theta[\J^{(j_n\to \kappa_3,\kappa_4)}] \theta[\I\cup\{j_n,\kappa_4,\kappa_5\}]
 \theta[\J^{(j_m\to \kappa_3,\kappa_4)}] \theta[\I\cup\{j_m,\kappa_4,\kappa_5\}]} \times \\
 \times \frac{\theta[\J^{(j_n\to \kappa_2,\kappa_4)}]\theta[\I\cup\{j_n,\kappa_4,\kappa_5\}]
 \theta[\J^{(j_m\to \kappa_2,\kappa_4)}]\theta[\I\cup\{j_m,\kappa_4,\kappa_5\}]}
 {\theta[\J^{(j_n\to \kappa_1,\kappa_2)}]\theta[\I\cup\{j_n,\kappa_1,\kappa_5\}]
 \theta[\J^{(j_m\to \kappa_1,\kappa_2)}]\theta[\I\cup\{j_m,\kappa_1,\kappa_5\}]}  \\
 + \frac{\theta[\J^{(j_n\to \kappa_2,\kappa_3)}]\theta[\I\cup\{j_n,\kappa_2,\kappa_5\}]
 \theta[\J^{(j_m\to \kappa_2,\kappa_3)}]\theta[\I\cup\{j_n,\kappa_2,\kappa_5\}]}
 {\theta[\J^{(j_n\to \kappa_3,\kappa_4)}]\theta[\I\cup\{j_n,\kappa_4,\kappa_5\}]
 \theta[\J^{(j_m\to \kappa_3,\kappa_4)}]\theta[\I\cup\{j_n,\kappa_4,\kappa_5\}]} \times \\
 \times \frac{\theta[\J^{(j_n\to \kappa_1,\kappa_4)}]\theta[\I\cup\{j_n,\kappa_4,\kappa_5\}]
 \theta[\J^{(j_m\to \kappa_1,\kappa_4)}]\theta[\I\cup\{j_n,\kappa_4,\kappa_5\}]}
 {\theta[\J^{(j_n\to \kappa_1,\kappa_2)}]\theta[\I\cup\{j_n,\kappa_2,\kappa_5\}]
 \theta[\J^{(j_m\to \kappa_1,\kappa_2)}]\theta[\I\cup\{j_n,\kappa_2,\kappa_5\}]} + 1 \\
 = - \frac{(e_{\kappa_4}-e_{\kappa_2})(e_{\kappa_3}-e_{\kappa_1})}{(e_{\kappa_2}-e_{\kappa_1})(e_{\kappa_4}-e_{\kappa_3})} 
 + \frac{(e_{\kappa_4}-e_{\kappa_1})(e_{\kappa_3}-e_{\kappa_2})}{(e_{\kappa_2}-e_{\kappa_1})(e_{\kappa_4}-e_{\kappa_3})} + 1
 = 0,
\end{multline*}
and \eqref{D1ThetaLD4i} simplifies to 
\begin{multline*}
 - \theta[\J^{(j_n\to \kappa_1,\kappa_2)}]\theta[\J^{(j_m\to \kappa_1,\kappa_2)}]
 \theta[\J^{(j_m,j_n \to \kappa_3,\kappa_4,\kappa_5)}]
  \partial_{v} \theta[\I\cup \{\kappa_1,\kappa_2\}] \\
 + \theta[\J^{(j_n\to \kappa_1,\kappa_3)}]\theta[\J^{(j_m\to \kappa_1,\kappa_3)}]
 \theta[\J^{(j_m,j_n \to \kappa_2,\kappa_4,\kappa_5)}]
 \partial_{v} \theta[\I\cup \{\kappa_1,\kappa_3\}] \\
 - \theta[\J^{(j_n\to \kappa_2,\kappa_3)}]\theta[\J^{(j_m\to \kappa_2,\kappa_3)}] 
 \theta[\J^{(j_m,j_n \to \kappa_1,\kappa_4,\kappa_5)}]
 \partial_{v} \theta[\I\cup \{\kappa_2,\kappa_3\}] \\
 + \theta[\J^{(j_n\to \kappa_4,\kappa_5)}]\theta[\J^{(j_m\to \kappa_4,\kappa_5)}] 
 \theta[\J^{(j_m,j_n \to \kappa_1,\kappa_2,\kappa_3)}]
  \partial_{v} \theta[\I\cup \{\kappa_4,\kappa_5\}] = 0,
\end{multline*}
which coincides with \eqref{D1ThetaLD4a}.

Linear independence of any three vectors in \eqref{D1ThetaLD4a}, 
say $\partial_v \theta[\mathcal{I}\cup \{\kappa_1,\kappa_2\}]$,
$\partial_v \theta[\mathcal{I}\cup \{\kappa_1,\kappa_3\}]$ and $\partial_v \theta[\mathcal{I}\cup \{\kappa_2,\kappa_3\}]$, 
follows from the fact that there is no relation between these theta derivatives due to Proposition~\ref{P:D1thetaLD3}',
according to which a relation between three theta derivatives 
exists only if characteristics of these theta derivatives are such that
the intersection of the corresponding sets of indices has cardinality $g-2$.
Instead, intersection $\mathcal{I}$ of the sets corresponding to characteristics
$[\mathcal{I}\cup \{\kappa_1,\kappa_2\}]$,
$[\mathcal{I}\cup \{\kappa_1,\kappa_3\}]$, $[\mathcal{I}\cup \{\kappa_2,\kappa_3\}]$ has 
 cardinality $n-3$.
\end{proof}

\section{Second derivative theta relations}\label{A:D2thetaLD4Ex}
\subsection{Genus $3$}
Formula \eqref{DI2RelG3} gives
$35$ representations (excluding different representations of ratios of theta constants
following from FTT Corollary~\ref{C:eklm}) for each entry of $\partial^2_v \theta^{\emptyset}$,
all the relations are equivalent.
In particular, with $\I_0=\{1,2,3\}$, $j_1=6$, $j_2=5$
\begin{multline*}
 \partial_{v_{n_1},v_{n_2}}^2 \theta^{\emptyset} = 
 \big(\theta^{\{1,2,3\}}\theta^{\{4,5,7\}}\theta^{\{4,6,7\}}\big)^{-1} \times \\ \times
 \Big({-} \theta^{\{3,5,6\}}\theta^{\{3,4,7\}}\theta^{\{1,2,5\}}\theta^{\{1,2,6\}}
 \big(\theta^{\{1,5,6\}}\theta^{\{2,5,6\}}\big)^{-1} \times \\ \times
 \big( \partial_{v_{n_1}} \theta^{\{2,3\}} \partial_{v_{n_2}} \theta^{\{1,3\}} 
 + \partial_{v_{n_2}} \theta^{\{2,3\}} \partial_{v_{n_1}} \theta^{\{1,3\}}\big)\\
 + \theta^{\{2,5,6\}}\theta^{\{2,4,7\}}\theta^{\{1,3,5\}}\theta^{\{1,3,6\}}
 \big(\theta^{\{1,5,6\}}\theta^{\{3,5,6\}}\big)^{-1} \times \\ \times
 \big( \partial_{v_{n_1}} \theta^{\{2,3\}} \partial_{v_{n_2}} \theta^{\{1,2\}} 
 + \partial_{v_{n_2}} \theta^{\{2,3\}} \partial_{v_{n_1}} \theta^{\{1,2\}}\big)\\
 - \theta^{\{1,5,6\}}\theta^{\{1,4,7\}}\theta^{\{2,3,5\}}\theta^{\{2,3,6\}}
 \big(\theta^{\{2,5,6\}}\theta^{\{3,5,6\}}\big)^{-1} \times \\ \times
 \big( \partial_{v_{n_1}} \theta^{\{1,3\}} \partial_{v_{n_2}} \theta^{\{1,2\}} 
 + \partial_{v_{n_2}} \theta^{\{1,3\}} \partial_{v_{n_1}} \theta^{\{1,2\}}\big)
\end{multline*}
or with characteristics in the standard form
\begin{multline*}
 \partial_{v_{n_1},v_{n_2}}^2 \theta[{}^{111}_{101}] = 
 \big(\theta[{}^{101}_{101}] \theta[{}^{100}_{010}]\theta[{}^{100}_{011}]\big)^{-1} \times \\ \times
 \Big({-}\theta[{}^{101}_{000}]\theta[{}^{111}_{000}]\theta[{}^{110}_{111}]\theta[{}^{110}_{110}]
 \big(\theta[{}^{011}_{100}]\theta[{}^{011}_{000}]\big)^{-1} \times \\ \times
 \big( \partial_{v_{n_1}} \theta[{}^{001}_{101}] \partial_{v_{n_2}} \theta[{}^{001}_{001}]
 + \partial_{v_{n_2}} \theta[{}^{001}_{101}] \partial_{v_{n_1}} \theta[{}^{001}_{001}]\big)\\
 + \theta[{}^{011}_{000}]\theta[{}^{001}_{000}]\theta[{}^{000}_{111}]\theta[{}^{000}_{110}]
 \big(\theta[{}^{011}_{100}]\theta[{}^{101}_{000}]\big)^{-1} \times \\ \times
 \big( \partial_{v_{n_1}} \theta[{}^{001}_{101}] \partial_{v_{n_2}} \theta[{}^{111}_{001}] 
 + \partial_{v_{n_2}} \theta[{}^{001}_{101}] \partial_{v_{n_1}} \theta[{}^{111}_{001}]\big)\\
 - \theta[{}^{011}_{100}]\theta[{}^{001}_{100}]\theta[{}^{000}_{011}]\theta[{}^{000}_{010}]
 \big(\theta[{}^{011}_{000}]\theta[{}^{101}_{000}]\big)^{-1} \times \\ \times
 \big( \partial_{v_{n_1}} \theta[{}^{001}_{001}] \partial_{v_{n_2}} \theta[{}^{111}_{001}] 
 + \partial_{v_{n_2}} \theta[{}^{001}_{001}] \partial_{v_{n_1}} \theta[{}^{111}_{001}]\big);
\end{multline*}
and with $\I_0=\{1,2,4\}$, $j_1=6$, $j_2=5$
\begin{multline*}
 \partial_{v_{n_1},v_{n_2}}^2 \theta^{\emptyset} = 
 \big(\theta^{\{1,2,4\}}\theta^{\{3,5,7\}}\theta^{\{3,6,7\}}\big)^{-1} \times \\ \times
 \Big({-} \theta^{\{4,5,6\}}\theta^{\{3,4,7\}}\theta^{\{1,2,5\}}\theta^{\{1,2,6\}}
 \big(\theta^{\{1,5,6\}}\theta^{\{2,5,6\}}\big)^{-1} \times \\ \times
 \big( \partial_{v_{n_1}} \theta^{\{2,4\}} \partial_{v_{n_2}} \theta^{\{1,4\}} 
 + \partial_{v_{n_2}} \theta^{\{2,4\}} \partial_{v_{n_1}} \theta^{\{1,4\}}\big)\\
 + \theta^{\{2,5,6\}}\theta^{\{2,3,7\}}\theta^{\{1,4,5\}}\theta^{\{1,4,6\}}
 \big(\theta^{\{1,5,6\}}\theta^{\{4,5,6\}}\big)^{-1} \times \\ \times
 \big( \partial_{v_{n_1}} \theta^{\{2,4\}} \partial_{v_{n_2}} \theta^{\{1,2\}} 
 + \partial_{v_{n_2}} \theta^{\{2,4\}} \partial_{v_{n_1}} \theta^{\{1,2\}}\big)\\
 - \theta^{\{1,5,6\}}\theta^{\{1,3,7\}}\theta^{\{2,4,5\}}\theta^{\{2,4,6\}}
 \big(\theta^{\{2,5,6\}}\theta^{\{4,5,6\}}\big)^{-1} \times \\ \times
 \big( \partial_{v_{n_1}} \theta^{\{1,4\}} \partial_{v_{n_2}} \theta^{\{1,2\}} 
 + \partial_{v_{n_2}} \theta^{\{1,4\}} \partial_{v_{n_1}} \theta^{\{1,2\}}\big)
\end{multline*}
or with characteristics in the standard form
\begin{multline*}
 \partial_{v_{n_1},v_{n_2}}^2 \theta[{}^{111}_{101}] = 
 \big(\theta[{}^{101}_{111}]\theta[{}^{100}_{000}] \theta[{}^{100}_{001}]\big)^{-1} \times \\ \times
 \Big({-} \theta[{}^{101}_{010}]\theta[{}^{111}_{000}]\theta[{}^{110}_{111}]\theta[{}^{110}_{110}]
 \big(\theta[{}^{011}_{100}]\theta[{}^{011}_{000}]\big)^{-1} \times \\ \times
 \big( \partial_{v_{n_1}} \theta[{}^{001}_{111}] \partial_{v_{n_2}} \theta[{}^{001}_{011}] 
 + \partial_{v_{n_2}} \theta[{}^{001}_{111}] \partial_{v_{n_1}} \theta[{}^{001}_{011}]\big)\\
 + \theta[{}^{011}_{000}]\theta[{}^{001}_{010}]\theta[{}^{000}_{101}]\theta[{}^{000}_{100}]
 \big(\theta[{}^{011}_{100}]\theta[{}^{101}_{010}]\big)^{-1} \times \\ \times
 \big( \partial_{v_{n_1}} \theta[{}^{001}_{111}] \partial_{v_{n_2}} \theta[{}^{111}_{001}] 
 + \partial_{v_{n_2}} \theta[{}^{001}_{111}] \partial_{v_{n_1}} \theta[{}^{111}_{001}]\big)\\
 - \theta[{}^{011}_{100}]\theta[{}^{001}_{110}]\theta[{}^{000}_{001}]\theta[{}^{000}_{000}]
 \big(\theta[{}^{011}_{000}]\theta[{}^{101}_{010}]\big)^{-1} \times \\ \times
 \big( \partial_{v_{n_1}} \theta[{}^{001}_{011}] \partial_{v_{n_2}} \theta[{}^{111}_{001}] 
 + \partial_{v_{n_2}} \theta[{}^{001}_{011}] \partial_{v_{n_1}} \theta[{}^{111}_{001}]\big).
\end{multline*}

\subsection{Genus $4$} There exist $9$ characteristics of multiplicity $2$ 
of the form $[\{\iota\}]$, and the corresponding second order theta derivatives 
are given by the formula \eqref{DI2RelG4}. 
Let $\I_0=\{1,2,3,4\}$ and $j_1=5$, $j_2=6$, then one obtains with $\iota=1$
\begin{multline*}
\partial_{v_{n_1},v_{n_2}}^2 \theta^{\{1\}} = 
 \frac{1}{\theta^{\{1,2,3,4\}}\theta^{\{6,7,8,9\}}\theta^{\{5,7,8,9\}}} \times \\ \times
 \Big(- \theta^{\{1,4,5,6\}}\theta^{\{4,7,8,9\}}\theta^{\{1,2,3,5\}}\theta^{\{1,2,3,6\}}
 \big(\theta^{\{1,2,5,6\}}\theta^{\{1,3,5,6\}}\big)^{-1} \times \\ \times
 \big( \partial_{v_{n_1}} \theta^{\{1,3,4\}} \partial_{v_{n_2}} \theta^{\{1,2,4\}} 
 + \partial_{v_{n_2}} \theta^{\{1,3,4\}} \partial_{v_{n_1}} \theta^{\{1,2,4\}}\big)\\
 + \theta^{\{1,3,5,6\}}\theta^{\{3,7,8,9\}}\theta^{\{1,2,4,5\}}\theta^{\{1,2,4,6\}}
 \big(\theta^{\{1,2,5,6\}}\theta^{\{1,4,5,6\}}\big)^{-1} \times \\ \times
 \big( \partial_{v_{n_1}} \theta^{\{1,3,4\}} \partial_{v_{n_2}} \theta^{\{1,2,3\}} 
 + \partial_{v_{n_2}} \theta^{\{1,3,4\}} \partial_{v_{n_1}} \theta^{\{1,2,3\}}\big)\\
 - \theta^{\{1,2,5,6\}}\theta^{\{2,7,8,9\}}\theta^{\{1,3,4,5\}}\theta^{\{1,3,4,6\}}
 \big(\theta^{\{1,3,5,6\}}\theta^{\{1,4,5,6\}}\big)^{-1} \times \\ \times
 \big( \partial_{v_{n_1}} \theta^{\{1,2,4\}} \partial_{v_{n_2}} \theta^{\{1,2,3\}} 
 + \partial_{v_{n_2}} \theta^{\{1,2,4\}} \partial_{v_{n_1}} \theta^{\{1,2,3\}}\big)
\end{multline*}
or in the standard form
\begin{multline*}
 \partial_{v_{n_1},v_{n_2}}^2 \theta[{}^{0111}_{0101}] = 
 \big(\theta[{}^{1111}_{1001}]\theta[{}^{1101}_{0101}] \theta[{}^{1101}_{0111}]\big)^{-1} \times \\ \times
 \Big({-} \theta[{}^{0011}_{1011}]\theta[{}^{1011}_{0111}]\theta[{}^{1001}_{1001}]\theta[{}^{1001}_{1011}]
 \big(\theta[{}^{1111}_{1111}]\theta[{}^{0011}_{1111}]\big)^{-1} \times \\ \times
 \big( \partial_{v_{n_1}} \theta[{}^{0111}_{0001}] \partial_{v_{n_2}} \theta[{}^{1011}_{0001}] 
 + \partial_{v_{n_2}} \theta[{}^{0111}_{0001}] \partial_{v_{n_1}} \theta[{}^{1011}_{0001}]\big)\\
 + \theta[{}^{0011}_{1111}]\theta[{}^{1011}_{0011}]\theta[{}^{1001}_{1101}]\theta[{}^{1001}_{1111}]
 \big(\theta[{}^{1111}_{1111}]\theta[{}^{0011}_{1011}]\big)^{-1} \times \\ \times
 \big( \partial_{v_{n_1}} \theta[{}^{0111}_{0001}] \partial_{v_{n_2}} \theta[{}^{1011}_{0101}] 
 + \partial_{v_{n_2}} \theta[{}^{0111}_{0001}] \partial_{v_{n_1}} \theta[{}^{1011}_{0101}]\big)\\
 - \theta[{}^{1111}_{1111}]\theta[{}^{0111}_{0011}]\theta[{}^{0101}_{1101}]\theta[{}^{0101}_{1111}]
 \big(\theta[{}^{0011}_{1111}]\theta[{}^{0011}_{1011}]\big)^{-1} \times \\ \times
 \big( \partial_{v_{n_1}} \theta[{}^{1011}_{0001}] \partial_{v_{n_2}} \theta[{}^{1011}_{0101}] 
 + \partial_{v_{n_2}} \theta[{}^{1011}_{0001}] \partial_{v_{n_1}} \theta[{}^{1011}_{0101}]\big);
\end{multline*}
and with $\iota=2$
\begin{multline*}
\partial_{v_{n_1},v_{n_2}}^2 \theta^{\{2\}} = 
 \frac{1}{\theta^{\{1,2,3,4\}}\theta^{\{6,7,8,9\}}\theta^{\{5,7,8,9\}}} \times \\ \times
 \Big(- \theta^{\{2,4,5,6\}}\theta^{\{4,7,8,9\}}\theta^{\{1,2,3,5\}}\theta^{\{1,2,3,6\}}
 \big(\theta^{\{1,2,5,6\}}\theta^{\{2,3,5,6\}}\big)^{-1} \times \\ \times
 \big( \partial_{v_{n_1}} \theta^{\{2,3,4\}} \partial_{v_{n_2}} \theta^{\{1,2,4\}} 
 + \partial_{v_{n_2}} \theta^{\{2,3,4\}} \partial_{v_{n_1}} \theta^{\{1,2,4\}}\big)\\
 + \theta^{\{2,3,5,6\}}\theta^{\{3,7,8,9\}}\theta^{\{1,2,4,5\}}\theta^{\{1,2,4,6\}}
 \big(\theta^{\{1,2,5,6\}}\theta^{\{2,4,5,6,\}}\big)^{-1} \times \\ \times
 \big( \partial_{v_{n_1}} \theta^{\{2,3,4\}} \partial_{v_{n_2}} \theta^{\{1,2,3\}} 
 + \partial_{v_{n_2}} \theta^{\{2,3,4\}} \partial_{v_{n_1}} \theta^{\{1,2,3\}}\big)\\
 - \theta^{\{1,2,5,6\}}\theta^{\{1,7,8,9\}}\theta^{\{2,3,4,5\}}\theta^{\{2,3,4,6\}}
 \big(\theta^{\{2,3,5,6\}}\theta^{\{2,4,5,6\}}\big)^{-1} \times \\ \times
 \big( \partial_{v_{n_1}} \theta^{\{1,2,4\}} \partial_{v_{n_2}} \theta^{\{1,2,3\}} 
 + \partial_{v_{n_2}} \theta^{\{1,2,4\}} \partial_{v_{n_1}} \theta^{\{1,2,3\}}\big)
\end{multline*}
or in the standard form
\begin{multline*}
 \partial_{v_{n_1},v_{n_2}}^2 \theta[{}^{0111}_{1101}] = 
 \big(\theta[{}^{1111}_{1001}]\theta[{}^{1101}_{0101}] \theta[{}^{1101}_{0111}]\big)^{-1} \times \\ \times
 \Big({-} \theta[{}^{0011}_{1011}]\theta[{}^{1011}_{0111}]\theta[{}^{1001}_{1001}]\theta[{}^{1001}_{1011}]
 \big(\theta[{}^{1111}_{1111}]\theta[{}^{0011}_{1111}]\big)^{-1} \times \\ \times
 \big( \partial_{v_{n_1}} \theta[{}^{0111}_{1001}] \partial_{v_{n_2}} \theta[{}^{1011}_{0001}] 
 + \partial_{v_{n_2}} \theta[{}^{0111}_{1001}] \partial_{v_{n_1}} \theta[{}^{1011}_{0001}]\big)\\
 + \theta[{}^{0011}_{1111}]\theta[{}^{1011}_{0011}]\theta[{}^{1001}_{1101}]\theta[{}^{1001}_{1111}]
 \big(\theta[{}^{1111}_{1111}]\theta[{}^{0011}_{1011}]\big)^{-1} \times \\ \times
 \big( \partial_{v_{n_1}} \theta[{}^{0111}_{1001}] \partial_{v_{n_2}} \theta[{}^{1011}_{0101}] 
 + \partial_{v_{n_2}} \theta[{}^{0111}_{1001}] \partial_{v_{n_1}} \theta[{}^{1011}_{0101}]\big)\\
 - \theta[{}^{1111}_{1111}]\theta[{}^{0111}_{0011}]\theta[{}^{0101}_{1101}]\theta[{}^{0101}_{1111}]
 \big(\theta[{}^{0011}_{1111}]\theta[{}^{0011}_{1011}]\big)^{-1} \times \\ \times
 \big( \partial_{v_{n_1}} \theta[{}^{1011}_{0001}] \partial_{v_{n_2}} \theta[{}^{1011}_{0101}] 
 + \partial_{v_{n_2}} \theta[{}^{1011}_{0001}] \partial_{v_{n_1}} \theta[{}^{1011}_{0101}]\big).
\end{multline*}

Let $\I_0=\{i_1$, $i_2$, $i_3$, $i_4\}$, $\J_0=\{j_1$, $j_2$, $j_3$, $j_4$, $j_5\}$ and $j_m=j_1$, $j_n=j_2$.
Second order theta derivatives with characteristic $\emptyset$ are given by the following formula
\begin{multline*}
  \partial_{v_{n_1},v_{n_2}}^2 \theta^{\emptyset} =
  \big(\theta^{\{i_1,i_2,i_3,i_4\}} (\theta^{\{j_2,j_3,j_4,j_5\}} \theta^{\{j_1,j_3,j_4,j_5\}})^2\big)^{-1} 
 \times \\ \times \bigg({-}
 \frac{\theta^{\{i_3,i_4,j_1,j_2\}} \theta^{\{i_1,i_2,j_1,j_2\}}\theta^{\{i_1,j_3,j_4,j_5\}}\theta^{\{i_2,j_3,j_4,j_5\}}}
 {\theta^{\{i_2,i_4,j_1,j_2\}} \theta^{\{i_1,i_4,j_1,j_2\}}\theta^{\{i_2,i_3,j_1,j_2\}}\theta^{\{i_1,i_3,j_1,j_2\}}} \times \\ \times
 \theta^{\{i_1,i_2,i_4,j_1\}} \theta^{\{i_1,i_2,i_4,j_2\}} 
 \theta^{\{i_1,i_2,i_3,j_1\}} \theta^{\{i_1,i_2,i_3,j_2\}} \times \\ \times
  \big( \partial_{v_{n_1}} \theta^{\{i_2,i_3,i_4\}} \partial_{v_{n_2}} \theta^{\{i_1,i_3,i_4\}} 
  + \partial_{v_{n_2}} \theta^{\{i_2,i_3,i_4\}} \partial_{v_{n_1}} \theta^{\{i_1,i_3,i_4\}}\big) \\
  + \frac{\theta^{\{i_2,i_4,j_1,j_2\}} \theta^{\{i_1,i_3,j_1,j_2\}}\theta^{\{i_1,j_3,j_4,j_5\}}\theta^{\{i_3,j_3,j_4,j_5\}}}
 {\theta^{\{i_3,i_4,j_1,j_2\}} \theta^{\{i_1,i_4,j_1,j_2\}}\theta^{\{i_2,i_3,j_1,j_2\}}\theta^{\{i_1,i_2,j_1,j_2\}}} \times \\ \times
 \theta^{\{i_1,i_3,i_4,j_1\}} \theta^{\{i_1,i_3,i_4,j_2\}} 
 \theta^{\{i_1,i_2,i_3,j_1\}} \theta^{\{i_1,i_2,i_3,j_2\}} \times \\ \times
  \big( \partial_{v_{n_1}} \theta^{\{i_2,i_3,i_4\}} \partial_{v_{n_2}} \theta^{\{i_1,i_2,i_4\}} 
  + \partial_{v_{n_2}} \theta^{\{i_2,i_3,i_4\}} \partial_{v_{n_1}} \theta^{\{i_1,i_2,i_4\}}\big)\\
  - \frac{\theta^{\{i_2,i_3,j_1,j_2\}} \theta^{\{i_1,i_4,j_1,j_2\}}\theta^{\{i_1,j_3,j_4,j_5\}}\theta^{\{i_4,j_3,j_4,j_5\}}}
 {\theta^{\{i_3,i_4,j_1,j_2\}} \theta^{\{i_2,i_4,j_1,j_2\}}\theta^{\{i_1,i_2,j_1,j_2\}}\theta^{\{i_1,i_3,j_1,j_2\}}} \times \\ \times
 \theta^{\{i_1,i_2,i_4,j_1\}} \theta^{\{i_1,i_2,i_4,j_2\}} 
 \theta^{\{i_1,i_3,i_4,j_1\}} \theta^{\{i_1,i_3,i_4,j_2\}} \times \\ \times
  \big( \partial_{v_{n_1}} \theta^{\{i_2,i_3,i_4\}} \partial_{v_{n_2}} \theta^{\{i_1,i_2,i_3\}} 
  + \partial_{v_{n_2}} \theta^{\{i_2,i_3,i_4\}} \partial_{v_{n_1}} \theta^{\{i_1,i_2,i_3\}}\big)\\
  - \frac{\theta^{\{i_1,i_4,j_1,j_2\}} \theta^{\{i_2,i_3,j_1,j_2\}}\theta^{\{i_2,j_3,j_4,j_5\}}\theta^{\{i_3,j_3,j_4,j_5\}}}
 {\theta^{\{i_3,i_4,j_1,j_2\}} \theta^{\{i_2,i_4,j_1,j_2\}}\theta^{\{i_1,i_2,j_1,j_2\}}\theta^{\{i_1,i_3,j_1,j_2\}}} \times \\ \times
 \theta^{\{i_2,i_3,i_4,j_1\}} \theta^{\{i_2,i_3,i_4,j_2\}} 
 \theta^{\{i_1,i_2,i_3,j_1\}} \theta^{\{i_1,i_2,i_3,j_2\}} \times \\ \times
  \big( \partial_{v_{n_1}} \theta^{\{i_1,i_3,i_4\}} \partial_{v_{n_2}} \theta^{\{i_1,i_2,i_4\}} 
  + \partial_{v_{n_2}} \theta^{\{i_1,i_3,i_4\}} \partial_{v_{n_1}} \theta^{\{i_1,i_2,i_4\}}\big)\\
  + \frac{\theta^{\{i_1,i_3,j_1,j_2\}} \theta^{\{i_2,i_4,j_1,j_2\}}\theta^{\{i_2,j_3,j_4,j_5\}}\theta^{\{i_4,j_3,j_4,j_5\}}}
 {\theta^{\{i_3,i_4,j_1,j_2\}} \theta^{\{i_1,i_4,j_1,j_2\}}\theta^{\{i_1,i_2,j_1,j_2\}}\theta^{\{i_2,i_3,j_1,j_2\}}} \times \\ \times
 \theta^{\{i_1,i_2,i_4,j_1\}} \theta^{\{i_1,i_2,i_4,j_2\}} 
 \theta^{\{i_2,i_3,i_4,j_1\}} \theta^{\{i_2,i_3,i_4,j_2\}} \times \\ \times
  \big( \partial_{v_{n_1}} \theta^{\{i_1,i_3,i_4\}} \partial_{v_{n_2}} \theta^{\{i_1,i_2,i_3\}} 
  + \partial_{v_{n_2}} \theta^{\{i_1,i_3,i_4\}} \partial_{v_{n_1}} \theta^{\{i_1,i_2,i_3\}}\big)\\
  - \frac{\theta^{\{i_1,i_2,j_1,j_2\}} \theta^{\{i_3,i_4,j_1,j_2\}}\theta^{\{i_3,j_3,j_4,j_5\}}\theta^{\{i_4,j_3,j_4,j_5\}}}
 {\theta^{\{i_2,i_4,j_1,j_2\}} \theta^{\{i_1,i_4,j_1,j_2\}}\theta^{\{i_1,i_3,j_1,j_2\}}\theta^{\{i_2,i_3,j_1,j_2\}}} \times \\ \times
 \theta^{\{i_2,i_3,i_4,j_1\}} \theta^{\{i_2,i_3,i_4,j_2\}} 
 \theta^{\{i_1,i_3,i_4,j_1\}} \theta^{\{i_1,i_3,i_4,j_2\}} \times \\ \times
  \big( \partial_{v_{n_1}} \theta^{\{i_1,i_2,i_4\}} \partial_{v_{n_2}} \theta^{\{i_1,i_2,i_3\}} 
  + \partial_{v_{n_2}} \theta^{\{i_1,i_2,i_4\}} \partial_{v_{n_1}} \theta^{\{i_1,i_2,i_3\}}\big)\bigg).
\end{multline*}
For example with $\I_0=\{1,2,3,4\}$, $j_1=5$, $j_2=6$ one gets
\begin{multline*}
  \partial_{v_{n_1},v_{n_2}}^2 \theta^{\emptyset} =
  \big(\theta^{\{1,2,3,4\}} (\theta^{\{6,7,8,9\}} \theta^{\{5,7,8,9\}})^2\big)^{-1} 
 \times \\ \times \bigg({-}
 \frac{\theta^{\{3,4,5,6\}} \theta^{\{1,2,5,6\}}\theta^{\{1,7,8,9\}}\theta^{\{2,7,8,9\}}}
 {\theta^{\{2,4,5,6\}} \theta^{\{1,4,5,6\}}\theta^{\{2,3,5,6\}}\theta^{\{1,3,5,6\}}} 
 \theta^{\{1,2,4,5\}} \theta^{\{1,2,4,6\}} 
 \theta^{\{1,2,3,5\}} \theta^{\{1,2,3,6\}} \times \\ \times
  \big( \partial_{v_{n_1}} \theta^{\{2,3,4\}} \partial_{v_{n_2}} \theta^{\{1,3,4\}} 
  + \partial_{v_{n_2}} \theta^{\{2,3,4\}} \partial_{v_{n_1}} \theta^{\{1,3,4\}}\big) \\
  + \frac{\theta^{\{2,4,5,6\}} \theta^{\{1,3,5,6\}}\theta^{\{1,7,8,9\}}\theta^{\{3,7,8,9\}}}
 {\theta^{\{3,4,5,6\}} \theta^{\{1,4,5,6\}}\theta^{\{2,3,5,6\}}\theta^{\{1,2,5,6\}}}
 \theta^{\{1,3,4,5\}} \theta^{\{1,3,4,6\}}\theta^{\{1,2,3,5\}} \theta^{\{1,2,3,6\}} \times \\ \times
  \big( \partial_{v_{n_1}} \theta^{\{2,3,4\}} \partial_{v_{n_2}} \theta^{\{1,2,4\}} 
  + \partial_{v_{n_2}} \theta^{\{2,3,4\}} \partial_{v_{n_1}} \theta^{\{1,2,4\}}\big)\\
  - \frac{\theta^{\{2,3,5,6\}} \theta^{\{1,4,5,6\}}\theta^{\{1,7,8,9\}}\theta^{\{4,7,8,9\}}}
 {\theta^{\{3,4,5,6\}} \theta^{\{2,4,5,6\}}\theta^{\{1,2,5,6\}}\theta^{\{1,3,5,6\}}} 
 \theta^{\{1,2,4,5\}} \theta^{\{1,2,4,6\}} \theta^{\{1,3,4,5\}} \theta^{\{1,3,4,6\}} \times \\ \times
  \big( \partial_{v_{n_1}} \theta^{\{2,3,4\}} \partial_{v_{n_2}} \theta^{\{1,2,3\}} 
  + \partial_{v_{n_2}} \theta^{\{2,3,4\}} \partial_{v_{n_1}} \theta^{\{1,2,3\}}\big)\\
  - \frac{\theta^{\{1,4,5,6\}} \theta^{\{2,3,5,6\}}\theta^{\{2,7,8,9\}}\theta^{\{3,7,8,9\}}}
 {\theta^{\{3,4,5,6\}} \theta^{\{2,4,5,6\}}\theta^{\{1,2,5,6\}}\theta^{\{1,3,5,6\}}} 
 \theta^{\{2,3,4,5\}} \theta^{\{2,3,4,6\}} \theta^{\{1,2,3,5\}} \theta^{\{1,2,3,6\}} \times \\ \times
  \big( \partial_{v_{n_1}} \theta^{\{1,3,4\}} \partial_{v_{n_2}} \theta^{\{1,2,4\}} 
  + \partial_{v_{n_2}} \theta^{\{1,3,4\}} \partial_{v_{n_1}} \theta^{\{1,2,4\}}\big)\\
  + \frac{\theta^{\{1,3,5,6\}} \theta^{\{2,4,5,6\}}\theta^{\{2,7,8,9\}}\theta^{\{4,7,8,9\}}}
 {\theta^{\{3,4,5,6\}} \theta^{\{1,4,5,6\}}\theta^{\{1,2,5,6\}}\theta^{\{2,3,5,6\}}} 
 \theta^{\{1,2,4,5\}} \theta^{\{1,2,4,6\}} \theta^{\{2,3,4,5\}} \theta^{\{2,3,4,6\}} \times \\ \times
  \big( \partial_{v_{n_1}} \theta^{\{1,3,4\}} \partial_{v_{n_2}} \theta^{\{1,2,3\}} 
  + \partial_{v_{n_2}} \theta^{\{1,3,4\}} \partial_{v_{n_1}} \theta^{\{1,2,3\}}\big)\\
  - \frac{\theta^{\{1,2,5,6\}} \theta^{\{3,4,5,6\}}\theta^{\{3,7,8,9\}}\theta^{\{4,7,8,9\}}}
 {\theta^{\{2,4,5,6\}} \theta^{\{1,4,5,6\}}\theta^{\{1,3,5,6\}}\theta^{\{2,3,5,6\}}} 
 \theta^{\{2,3,4,5\}} \theta^{\{2,3,4,6\}} \theta^{\{1,3,4,5\}} \theta^{\{1,3,4,6\}} \times \\ \times
  \big( \partial_{v_{n_1}} \theta^{\{1,2,4\}} \partial_{v_{n_2}} \theta^{\{1,2,3\}} 
  + \partial_{v_{n_2}} \theta^{\{1,2,4\}} \partial_{v_{n_1}} \theta^{\{1,2,3\}}\big)\bigg)
\end{multline*}
or with characteristics in the standard form
\begin{multline*}
 \partial_{v_{n_1},v_{n_2}}^2 \theta[{}^{1111}_{0101}] = 
 \big(\theta[{}^{1111}_{1001}] \big(\theta[{}^{1101}_{0101}] \theta[{}^{1101}_{0111}]\big)^2\big)^{-1} \times \\ \times
 \Big(- \frac{\theta[{}^{1111}_{0011}]\theta[{}^{1111}_{1111}]\theta[{}^{0111}_{1011}]\theta[{}^{0111}_{0011}]}
{\theta[{}^{0011}_{0011}]\theta[{}^{0011}_{1011}]\theta[{}^{0011}_{0111}]\theta[{}^{0011}_{1111}] }
\theta[{}^{1001}_{1101}]\theta[{}^{1001}_{1111}]\theta[{}^{1001}_{1001}]\theta[{}^{1001}_{1011}] \times \\ \times
 \big( \partial_{v_{n_1}} \theta[{}^{0111}_{1001}] \partial_{v_{n_2}} \theta[{}^{0111}_{0001}] 
 + \partial_{v_{n_2}} \theta[{}^{0111}_{1001}] \partial_{v_{n_1}} \theta[{}^{0111}_{0001}]\big)\\
+ \frac{\theta[{}^{0011}_{0011}]\theta[{}^{0011}_{1111}]\theta[{}^{0111}_{1011}]\theta[{}^{1011}_{0011}]}
{\theta[{}^{1111}_{0011}]\theta[{}^{0011}_{1011}]\theta[{}^{0011}_{0111}]\theta[{}^{1111}_{1111}] }
\theta[{}^{0101}_{1101}]\theta[{}^{0101}_{1111}]\theta[{}^{1001}_{1001}]\theta[{}^{1001}_{1011}]  \times \\ \times
 \big( \partial_{v_{n_1}} \theta[{}^{0111}_{1001}] \partial_{v_{n_2}} \theta[{}^{1011}_{0001}] 
 + \partial_{v_{n_2}} \theta[{}^{0111}_{1001}] \partial_{v_{n_1}} \theta[{}^{1011}_{0001}]\big)\\
 - \frac{\theta[{}^{0011}_{0111}]\theta[{}^{0011}_{1011}]\theta[{}^{0111}_{1011}]\theta[{}^{1011}_{0111}]}
{\theta[{}^{1111}_{0011}]\theta[{}^{0011}_{0011}]\theta[{}^{1111}_{1111}]\theta[{}^{0011}_{1111}] }
\theta[{}^{1001}_{1101}]\theta[{}^{1001}_{1111}]\theta[{}^{0101}_{1101}]\theta[{}^{0101}_{1111}] \times \\ \times
  \big( \partial_{v_{n_1}} \theta[{}^{0111}_{1001}] \partial_{v_{n_2}} \theta[{}^{1011}_{0101}] 
 + \partial_{v_{n_2}} \theta[{}^{0111}_{1001}] \partial_{v_{n_1}} \theta[{}^{1011}_{0101}]\big)\\
 - \frac{\theta[{}^{0011}_{1011}]\theta[{}^{0011}_{0111}]\theta[{}^{0111}_{0011}]\theta[{}^{1011}_{0011}]}
{\theta[{}^{1111}_{0011}]\theta[{}^{0011}_{0011}]\theta[{}^{1111}_{1111}]\theta[{}^{0011}_{1111}] }
\theta[{}^{0101}_{0101}]\theta[{}^{0101}_{0111}]\theta[{}^{1001}_{1001}]\theta[{}^{1001}_{1011}]  \times \\ \times
  \big( \partial_{v_{n_1}} \theta[{}^{0111}_{0001}] \partial_{v_{n_2}} \theta[{}^{1011}_{0001}] 
 + \partial_{v_{n_2}} \theta[{}^{0111}_{0001}] \partial_{v_{n_1}} \theta[{}^{1011}_{0001}]\big)\\
+ \frac{\theta[{}^{0011}_{1111}]\theta[{}^{0011}_{0011}]\theta[{}^{0111}_{0011}]\theta[{}^{1011}_{0111}]}
{\theta[{}^{1111}_{0011}]\theta[{}^{0011}_{1011}]\theta[{}^{1111}_{1111}]\theta[{}^{0011}_{0111}] }
\theta[{}^{1001}_{1101}]\theta[{}^{1001}_{1111}]\theta[{}^{0101}_{0101}]\theta[{}^{0101}_{0111}]  \times \\ \times
  \big( \partial_{v_{n_1}} \theta[{}^{0111}_{0001}] \partial_{v_{n_2}} \theta[{}^{1011}_{0101}] 
 + \partial_{v_{n_2}} \theta[{}^{0111}_{0001}] \partial_{v_{n_1}} \theta[{}^{1011}_{0101}]\big) \\
- \frac{\theta[{}^{1111}_{1111}]\theta[{}^{111}_{0011}]\theta[{}^{1011}_{0011}]\theta[{}^{1011}_{0111}]}
{\theta[{}^{0011}_{1011}]\theta[{}^{0011}_{0011}]\theta[{}^{0011}_{1111}]\theta[{}^{0011}_{0111}] }
\theta[{}^{0101}_{1101}]\theta[{}^{0101}_{1111}]\theta[{}^{0101}_{0101}]\theta[{}^{0101}_{0111}]  \times \\ \times
  \big( \partial_{v_{n_1}} \theta[{}^{1011}_{0001}] \partial_{v_{n_2}} \theta[{}^{1011}_{0101}] 
 + \partial_{v_{n_2}} \theta[{}^{1011}_{0001}] \partial_{v_{n_1}} \theta[{}^{1011}_{0101}]\big)\Big).
\end{multline*}

\section{Examples of the Schottky relations}\label{A:Schottky}
Here we analyse some examples of the Schottky invariants presented in the literature,
and propose some new examples derived on the base of ideas from Remark~\ref{r:rSchottky}.

\subsection{Examples from papers of Farkas \& Rauch}\label{A:ShottkyComp}
In genus $4$ hyperelliptic case the Schottky relation \eqref{SchottkyEqlt} is given by
formula ($\mathrm{R}_2$) \cite[p.\,685]{FR1969Schtt}, the same in \cite[p.\,459]{FR1970Schtt}.  
In our notation and choice of homology basis it gets the form
\begin{multline}\label{FarSchtt69}
 \big(\theta^{\{2,4,6,8\}} \theta^{\{1,4,6,8\}} \theta^{\{2,5,7,9\}} \theta^{\{1,5,7,9\}}
 \theta^{\{2,4,7,9\}}\theta^{\{1,4,7,9\}}\theta^{\{2,5,6,8\}}\theta^{\{1,5,6,8\}}\big)^{1/2}\\
 - \big(\theta^{\{2,4,6,9\}} \theta^{\{1,4,6,9\}} \theta^{\{2,5,7,8\}} \theta^{\{1,5,7,8\}}
 \theta^{\{2,4,7,8\}}\theta^{\{1,4,7,8\}}\theta^{\{2,5,6,9\}}\theta^{\{1,5,6,9\}}\big)^{1/2}\\
 - \big(\theta^{\{2,4,6,7\}} \theta^{\{1,4,6,7\}} \theta^{\{2,5,8,9\}} \theta^{\{1,5,8,9\}}
 \theta^{\{2,4,8,9\}}\theta^{\{1,4,8,9\}}\theta^{\{2,5,6,7\}}\theta^{\{1,5,6,7\}}\big)^{1/2} = 0.
\end{multline}
The signs are chosen according to the order of sets of indices.
The initial syzygetic group $(P)$ of rank $3$ is generated by $3$ elements
\begin{align*}
 &P_1= [\varepsilon_1]+[\varepsilon_2]& &\{1,2\}& \\
 &P_2 = [\varepsilon_1]+[\varepsilon_2]+[\varepsilon_3]&
 &\{1,2,3\}&\\
 &P_3 = [\varepsilon_1]+[\varepsilon_2]+[\varepsilon_3]+[\varepsilon_4]+[\varepsilon_5]&
 &\{1,2,3,4,5\},&
\end{align*}
given with the corresponding sets of indices. One can add sets by the union operation,
taking into account that an index occurring twice drops (two equal elements cancel each other out). 
Recall that characteristic $[\I]$ of theta constant
corresponds to the set $\I+\R$, where $\R=\{2,4,6,8\}$ is associated with  the vector of Riemann constants. 
For example, the set of indices in characteristic $[\{1,4,6,8\}]$ is obtained as $\{1,2\} + \{2,4,6,8\}$. 
In such notation it becomes obvious that all theta constants in \eqref{FarSchtt69} 
have non-singular even characteristics,
because all sets corresponding to characteristics contain $4=g$ indices.
Thus, characteristics of theta constants in the first summand of \eqref{FarSchtt69} correspond to 
the following elements of group $(P)$ (in the same order)
\begin{gather*}
 P_0=0,\quad P_1,\quad P_2, \quad P_1 P_2, \quad P_3, 
 \quad P_1 P_3, \quad P_2 P_3, \quad P_1 P_2 P_3.
\end{gather*}
Here the notation of \cite[ch.\,XVII]{bak897} is adopted, that is $P_1P_2$ 
denotes the sum of characteristics $P_1$ and $P_2$. 
Three cosets, which give rise to three products of theta constants in the Schottky invariant, 
are produced by the following azygetic triplet of characteristics
\begin{align*}
 &A_1= 0 = [\{2,4,6,8\}] && \{\} \\
 &A_2= [\varepsilon_8]+[\varepsilon_9] = [\{2,4,6,9\}] && \{8,9\} \\
 &A_3= [\varepsilon_7]+[\varepsilon_8] = [\{2,4,6,7\}] && \{7,8\}.
\end{align*}
A concise proof of \eqref{FarSchtt69} follows from FTT Corollary~\ref{C:eklm}:
\begin{multline*}
 \bigg(\frac{\theta^{\{2,4,6,8\}} \theta^{\{1,3,5,8\}} \theta^{\{1,3,5,7\}} \theta^{\{2,4,7,9\}}
  \theta^{\{1,4,7,9\}} \theta^{\{2,3,5,7\}} \theta^{\{2,3,5,8\}} \theta^{\{1,4,6,8\}} }
 {\theta^{\{2,4,6,7\}} \theta^{\{1,3,5,7\}} \theta^{\{1,3,5,8\}}  \theta^{\{2,4,8,9\}}
 \theta^{\{1,4,8,9\}} \theta^{\{2,3,5,8\}} \theta^{\{2,3,5,7\}}  \theta^{\{1,4,6,7\}} }
 \times \\ \times 
 \frac{\theta^{\{2,5,6,8\}} \theta^{\{1,3,4,8\}}\theta^{\{1,3,4,7\}} \theta^{\{2,5,7,9\}} 
 \theta^{\{1,5,7,9\}} \theta^{\{2,3,4,7\}} \theta^{\{2,3,4,8\}} \theta^{\{1,5,6,8\}} }
 {\theta^{\{2,5,6,7\}} \theta^{\{1,3,4,7\}}\theta^{\{1,3,4,8\}} \theta^{\{2,5,8,9\}} 
 \theta^{\{1,5,8,9\}} \theta^{\{2,3,4,8\}} \theta^{\{2,3,4,7\}} \theta^{\{1,5,6,7\}} } \bigg)^{1/2}\\
 - \bigg(\frac{\theta^{\{2,4,6,9\}} \theta^{\{1,3,5,9\}} \theta^{\{1,3,5,7\}} \theta^{\{2,4,7,8\}}
 \theta^{\{1,4,7,8\}} \theta^{\{2,3,5,7\}} \theta^{\{2,3,5,9\}} \theta^{\{1,4,6,9\}} }
 {\theta^{\{2,4,6,7\}} \theta^{\{1,3,5,7\}} \theta^{\{1,3,5,9\}} \theta^{\{2,4,8,9\}}
 \theta^{\{1,4,8,9\}} \theta^{\{2,3,5,9\}} \theta^{\{2,3,5,7\}} \theta^{\{1,4,6,7\}} } \times \\ 
 \times \frac{\theta^{\{2,5,6,9\}} \theta^{\{1,3,4,9\}} \theta^{\{1,3,4,7\}} \theta^{\{2,5,7,8\}}
 \theta^{\{1,5,7,8\}} \theta^{\{2,3,4,7\}} \theta^{\{2,3,4,9\}} \theta^{\{1,5,6,9\}} }
 {\theta^{\{2,5,6,7\}}\theta^{\{1,3,4,7\}} \theta^{\{1,3,4,9\}} \theta^{\{2,5,8,9\}}
 \theta^{\{1,5,8,9\}} \theta^{\{2,3,4,9\}}\theta^{\{2,3,4,7\}} \theta^{\{1,5,6,7\}} } \bigg)^{1/2} \\
 - 1 = \frac{(e_9-e_7)(e_8-e_6)}{(e_9-e_8)(e_7-e_6)} - \frac{(e_8-e_7)(e_9-e_6)}{(e_9-e_8)(e_7-e_6)} - 1 = 0.
\end{multline*}

In the context of Remark~\ref{r:rSchottky} the group $(P)$ 
is generated by three even pairwise syzygetic characteristics: $[\{2,4,7,9\}]=[\mathcal{A}(\{6,7,8,9\})]$,
$[\{1,4,6,8\}]=[\mathcal{A}(\{1,2\})]$, $[\{2,5,6,8\}]=[\mathcal{A}(\{4,5\})]$. In this case
$\I_0 = \{6,7,8,9\}$ and three cosets are produced by the even azygetic triplet:
$A_1 = [\I_0^{(7,9 \to 2,4)}]$, $A_2 = [\I_0^{(7,8 \to 2,4)}]$, 
$A_3 = [\I_0^{(8,9 \to 2,4)}]$.

At the same time, ratios from  \cite[Theorem~1, p.\,680]{FR1969Schtt} 
connect theta constants in genera $g-1$ and $g$. 
In the context of the above example one gets the following equalities of ratios
in genera $3$ and $4$
\begin{multline}\label{Ratio34Rels}
 \frac{\big(\theta^{\{2,4,6\}}\big)^2}{\theta^{\{2,4,6,8\}}\theta^{\{1,4,6,8\}}}
 = \frac{\big(\theta^{\{3,5,7\}}\big)^2}{\theta^{\{2,5,7,9\}}\theta^{\{1,5,7,9\}}}
 = \frac{\big(\theta^{\{2,5,7\}}\big)^2}{\theta^{\{2,4,7,9\}}\theta^{\{1,4,7,9\}}} \\
 = \frac{\big(\theta^{\{3,4,6\}}\big)^2}{\theta^{\{2,5,6,8\}}\theta^{\{1,5,6,8\}}}
 = \frac{\big(\theta^{\{2,4,7\}}\big)^2}{\theta^{\{2,4,6,9\}}\theta^{\{1,4,6,9\}}}
 = \frac{\big(\theta^{\{3,5,6\}}\big)^2}{\theta^{\{2,5,7,8\}}\theta^{\{1,5,7,8\}}}
 = \dots
\end{multline}
And similar relations between theta constants in genera $4$ and $5$
\begin{multline}\label{Ratio45Rels}
 \frac{\big(\theta^{\{2,4,6,8\}}\big)^2}{\theta^{\{2,4,6,8,10\}}\theta^{\{1,4,6,8,10\}}}
 = \frac{\big(\theta^{\{1,4,6,8\}}\big)^2}{\theta^{\{2,3,6,8,10\}}\theta^{\{1,3,6,8,10\}}} \\
 = \frac{\big(\theta^{\{2,5,7,9\}}\big)^2}{\theta^{\{2,4,7,9,11\}}\theta^{\{1,4,7,9,11\}}} 
 = \frac{\big(\theta^{\{1,5,7,9\}}\big)^2}{\theta^{\{2,3,7,9,11\}}\theta^{\{1,3,7,9,11\}}} \\
 = \frac{\big(\theta^{\{2,4,7,9\}}\big)^2}{\theta^{\{2,4,6,9,11\}}\theta^{\{1,4,6,9,11\}}}
 = \frac{\big(\theta^{\{1,4,7,9\}}\big)^2}{\theta^{\{2,3,6,9,11\}}\theta^{\{1,3,6,9,11\}}}
 = \dots
\end{multline}

As stated in \cite[p.\,685]{FR1969Schtt}, a genus $3$ curve is obtained from a  genus $4$ curve with $9$
branch points by dropping the first two branch points, that is the genus $3$ curve has branch points at
$e_3$, $e_4$, $e_5$, $e_6$, $e_7$, $e_8$, $e_9$ labeled by indices from $1$ to $7$.
Then the following relation
\begin{multline}\label{FarSchtt69G3}
 \theta^{\{2,4,6\}} \theta^{\{3,5,7\}} \theta^{\{2,5,7\}} \theta^{\{3,4,6\}}\\
 - \theta^{\{2,4,7\}} \theta^{\{3,5,6\}} \theta^{\{2,5,6\}} \theta^{\{3,4,7\}}
 - \theta^{\{2,4,5\}} \theta^{\{3,6,7\}} \theta^{\{2,6,7\}} \theta^{\{3,4,5\}} = 0,
\end{multline}
is derived from  \eqref{FarSchtt69} with the help of \eqref{Ratio34Rels}.
Alternatively, relation \eqref{FarSchtt69G3} can be proven
using FTT Corollary~\ref{C:eklm}.

Also the following extension of the Schottky relation \eqref{FarSchtt69} to genus $5$ 
is proposed in \cite{FR1969Schtt,FR1970Schtt}
\begin{multline}\label{FarSchtt70}
 \big(\theta^{\{2,4,6,8,10\}} \theta^{\{1,4,6,8,10\}} \theta^{\{2,5,7,9,11\}}\theta^{\{1,5,7,9,11\}} 
 \times \\ \times
 \theta^{\{2,4,7,9,11\}}\theta^{\{1,4,7,9,11\}} \theta^{\{2,5,6,8,10\}} \theta^{\{1,5,6,8,10\}}\big)^{1/2}\\
 -  \big(\theta^{\{2,4,6,8,11\}} \theta^{\{1,4,6,8,11\}} \theta^{\{2,5,7,9,10\}}\theta^{\{1,5,7,9,10\}} 
 \times \\ \times
 \theta^{\{2,4,7,9,10\}}\theta^{\{1,4,7,9,10\}} \theta^{\{2,5,6,8,11\}} \theta^{\{1,5,6,8,11\}}\big)^{1/2}\\
 - \big(\theta^{\{2,4,6,8,9\}} \theta^{\{1,4,6,8,9\}} \theta^{\{2,5,7,10,11\}}\theta^{\{1,5,7,10,11\}} 
 \times \\ \times
 \theta^{\{2,4,7,10,11\}}\theta^{\{1,4,7,10,11\}} \theta^{\{2,5,6,8,9\}} \theta^{\{1,5,6,8,9\}}\big)^{1/2}\\
 - \big(\theta^{\{2,4,6,7,8\}} \theta^{\{1,4,6,7,8\}} \theta^{\{2,5,9,10,11\}}\theta^{\{1,5,7,9,10,11\}} 
 \times \\ \times
 \theta^{\{2,4,7,9,10,11\}}\theta^{\{1,4,7,9,10,11\}} \theta^{\{2,5,6,7,8\}} \theta^{\{1,5,6,7,8\}}\big)^{1/2}= 0,
\end{multline}
which has the form
\begin{gather*}
 \sqrt{r_1} - \sqrt{r_2} - \sqrt{r_3} - \sqrt{r_4} = 0.
\end{gather*}
Here the group $(P)$ of rank $3$ is generated by the same $3$ elements
\begin{align*}
 &P_1= [\varepsilon_1]+[\varepsilon_2]& &\{1,2\}& \\
 &P_2 = [\varepsilon_1]+[\varepsilon_2]+[\varepsilon_3]&
 &\{1,2,3\}&\\
 &P_3 = [\varepsilon_1]+[\varepsilon_2]+[\varepsilon_3]+[\varepsilon_4]+[\varepsilon_5]&
 &\{1,2,3,4,5\},&
\end{align*}
and four products $\{r_1$, $r_2$, $r_3$, $r_4\}$ are constructed with the help of four characteristics azygetic in triplets,
namely:
\begin{align*}
 &A_1= 0 = [\{2,4,6,8,10\}] && \{\} \\
 &A_2= [\varepsilon_{10}]+[\varepsilon_{11}] = [\{2,4,6,8,11\}] && \{10,11\} \\
 &A_3= [\varepsilon_9]+[\varepsilon_{10}] = [\{2,4,6,8,9\}] && \{9,10\} \\
 &A_4= [\varepsilon_7]+[\varepsilon_{10}] = [\{2,4,6,7,8\}] && \{7,10\}.
\end{align*}
The way of producing \eqref{FarSchtt70} is explained in \cite[p.\,457--458]{FR1970Schtt}.

\subsection{Schottky relations in genus $5$}\label{A:SchottkyRel5}
With the help of results proposed in the present paper 
the standard Schottky relation can be obtained in an arbitrary genus.
In particular, in genus $5$ let $\I_0 = \{6,8,9,10,11\}$, $\K=\{6,9,10,11\}$ and $A_1$, $A_2$, $A_3$ be as above, that is
\begin{align*}
 &A_1 = [\I_0^{(9,11\to 2,4)}] = [\{2,4,6,8,10\}]& &\{\}&\\
 &A_2 = [\I_0^{(9,10\to 2,4)}] = [\{2,4,6,8,11\}]& &\{10,11\}&\\
 &A_3 = [\I_0^{(10,11\to 2,4)}] = [\{2,4,6,8,9\}]& &\{9,10\}.&
\end{align*}
Then relation \eqref{SchottkyR} follows from the Schottky relation
\begin{gather}\label{SchRelG5}
 \sqrt{r_1} - \sqrt{r_2} - \sqrt{r_3} = 0,
\end{gather}
which holds with the following products, derived from the group $(P)$ of rank $1$ generated by element $\K$,
\begin{align*}
 & r_1 = \big(\theta[\I_0^{(9,11\to 2,4)}]\theta[\I_0^{(6,10\to 2,4)}]\big)^4
 = \big(\theta^{\{2,4,6,8,10\}}\theta^{\{2,4,8,9,11\}}\big)^4,\\
 & r_2 = \big(\theta[\I_0^{(9,10\to 2,4)}]\theta[\I_0^{(6,11\to 2,4)}]\big)^4
 = \big(\theta^{\{2,4,6,8,11\}}\theta^{\{2,4,8,9,10\}}\big)^4,\\
 & r_3 = \big(\theta[\I_0^{(10,11\to 2,4)}]\theta[\I_0^{(6,9\to 2,4)}]\big)^4
 = \big(\theta^{\{2,4,6,8,9\}}\theta^{\{2,4,8,10,11\}}\big)^4.
\end{align*}
On the other hand, one can construct a group $(P)$ of rank $3$ generated by elements $\K$, $\{1,2\}$, $\{4,5\}$, 
that results into
\begin{align*}
 r_1 &= \theta^{\{2,4,6,8,10\}}\theta^{\{1,4,6,8,10\}}\theta^{\{2,5,6,8,10\}}\theta^{\{1,5,6,8,10\}} 
 \times \\ &\quad \times \theta^{\{2,4,8,9,11\}}\theta^{\{1,4,8,9,11\}}\theta^{\{2,4,8,9,11\}}\theta^{\{1,5,8,9,11\}},\\
 r_2 &= \theta^{\{2,4,6,8,11\}}\theta^{\{1,4,6,8,11\}} \theta^{\{2,5,6,8,11\}} \theta^{\{1,5,6,8,11\}} 
 \times \\ &\quad \times \theta^{\{2,4,8,9,10\}}\theta^{\{1,4,8,9,10\}}\theta^{\{2,5,8,9,10\}}\theta^{\{1,5,8,9,10\}},\\
 r_3 &= \theta^{\{2,4,6,8,9\}}\theta^{\{1,4,6,8,9\}}\theta^{\{2,5,6,8,9\}}\theta^{\{1,5,6,8,9\}}
 \times \\ &\quad \times \theta^{\{2,4,8,10,11\}}\theta^{\{1,4,8,10,11\}}\theta^{\{2,5,8,10,11\}}\theta^{\{1,5,8,10,11\}},
\end{align*}
or with characteristics in the standard form
\begin{align*}
 r_1 &= \theta[{}^{00000}_{00000}]\theta[{}^{00000}_{10000}]\theta[{}^{01100}_{00000}]\theta[{}^{01100}_{10000}] 
 \theta[{}^{00100}_{00010}]\theta[{}^{00100}_{10010}]\theta[{}^{01000}_{00010}]\theta[{}^{01000}_{10010}],\\
 r_2 &= \theta[{}^{00001}_{00000}]\theta[{}^{00001}_{10000}]\theta[{}^{01101}_{00000}]\theta[{}^{01101}_{10000}] 
 \theta[{}^{00101}_{00010}]\theta[{}^{00101}_{10010}]\theta[{}^{01001}_{00010}]\theta[{}^{01001}_{10010}],\\
 r_3 &= \theta[{}^{00000}_{00001}]\theta[{}^{00000}_{10001}]\theta[{}^{01100}_{00001}]\theta[{}^{01100}_{10001}] 
 \theta[{}^{00100}_{00011}]\theta[{}^{00100}_{10011}]\theta[{}^{01000}_{00011}]\theta[{}^{01000}_{10011}].
\end{align*}
The latter products $\{r_1,r_2,r_3\}$ also satisfy \eqref{SchRelG5}.

On the base of Remark~\ref{r:rSchottky} one can produce a group $(P)$ containing odd and even characteristics
which are syzygeric in pairs. Let characteristics $A_1$, $A_2$, $A_3$ be constructed by \eqref{Achars}, 
then all characteristics in $(A_1 P)$, $(A_2 P)$, $(A_3 P)$ are even non-singular. 
Then the Schottky relation \eqref{SchRelG5} holds.

For example, with the same $\I_0 = \{6,8,9,10,11\}$ and $\K=\{6,9,10,11\}$ as above, let $j_m=3$, $j_m=5$ and
group $(P)$ be generated by elements
\begin{align*}
 &P_1= [\varepsilon_1]+[\varepsilon_3]& &\{1,3\}& \\
 &P_2 = [\varepsilon_4]+[\varepsilon_5]& &\{4,5\}&\\
 &P_3 = [\varepsilon_6]+[\varepsilon_9]+[\varepsilon_{10}]+[\varepsilon_{11}]&
 &\{6,9,10,11\}=\K&
\end{align*}
With an azygetic triplet of characteristics 
\begin{align*}
 &A_1 = [\I_0^{(9,11\to 3,5)}] = [\{3,5,6,8,10\}]& &\{2,3,4,5\}&\\
 &A_2 = [\I_0^{(9,10\to 3,5)}] = [\{3,5,6,8,11\}]& &\{2,3,4,5,10,11\}&\\
 &A_3 = [\I_0^{(10,11\to 3,5)}] = [\{3,5,6,8,9\}]& &\{2,3,4,5,9,10\}&
\end{align*}
one obtains the following products of theta constants
\begin{align*}
 r_1 &= \theta^{\{3,5,6,8,10\}}\theta^{\{1,5,6,8,10\}}\theta^{\{3,4,6,8,10\}}\theta^{\{1,4,6,8,10\}} 
 \times \\ &\quad \times \theta^{\{3,5,8,9,11\}}\theta^{\{1,5,8,9,11\}}\theta^{\{3,4,8,9,11\}}\theta^{\{1,4,8,9,11\}},\\
 r_2 &= \theta^{\{3,5,6,8,11\}}\theta^{\{1,5,6,8,11\}} \theta^{\{3,4,6,8,11\}} \theta^{\{1,4,6,8,11\}} 
 \times \\ &\quad \times \theta^{\{3,5,8,9,10\}}\theta^{\{1,5,8,9,10\}}\theta^{\{3,4,8,9,10\}}\theta^{\{1,4,8,9,10\}},\\
 r_3 &= \theta^{\{3,5,6,8,9\}}\theta^{\{1,5,6,8,9\}}\theta^{\{3,4,6,8,9\}}\theta^{\{1,4,6,8,9\}}
 \times \\ &\quad \times \theta^{\{3,5,8,10,11\}}\theta^{\{1,5,8,10,11\}}\theta^{\{3,4,8,10,11\}}\theta^{\{1,4,8,10,11\}},
\end{align*}
or with characteristics in the standard form
\begin{align*}
 r_1 &= \theta[{}^{10100}_{00000}]\theta[{}^{01100}_{10000}]\theta[{}^{11000}_{00000}]\theta[{}^{00000}_{10000}] 
 \theta[{}^{10000}_{00010}]\theta[{}^{01000}_{10010}]\theta[{}^{11100}_{00010}]\theta[{}^{00100}_{10010}],\\
 r_2 &= \theta[{}^{10101}_{00000}]\theta[{}^{01101}_{10000}]\theta[{}^{11001}_{00000}]\theta[{}^{00001}_{10000}] 
 \theta[{}^{10001}_{00010}]\theta[{}^{01001}_{10010}]\theta[{}^{11101}_{00010}]\theta[{}^{00101}_{10010}],\\
 r_3 &= \theta[{}^{10100}_{00001}]\theta[{}^{01100}_{10001}]\theta[{}^{11000}_{00001}]\theta[{}^{00000}_{10001}] 
 \theta[{}^{10000}_{00011}]\theta[{}^{01000}_{10011}]\theta[{}^{11100}_{00011}]\theta[{}^{00100}_{10011}].
\end{align*}
This collection of products $\{r_1,r_2,r_3\}$ satisfies \eqref{SchRelG5}.

\subsection{Schottky relations from group of rank $4$}\label{A:SchottkyRank4}
From \eqref{FarSchtt69} with the help of \eqref{Ratio45Rels} one can derive a new version of the
Schottky relation on the base of a group $(P)$ of rank $4$, namely:
\begin{align*}
 r_1 &= \big(\theta^{\{2,4,6,8,10\}}\theta^{\{1,4,6,8,10\}}\theta^{\{2,3,6,8,10\}}\theta^{\{1,3,6,8,10\}} 
 \times \\ &\quad \times \theta^{\{2,4,7,9,11\}}\theta^{\{1,4,7,9,11\}}\theta^{\{2,3,7,9,11\}}\theta^{\{1,3,7,9,11\}}
 \times \\ &\quad \times \theta^{\{2,4,6,9,11\}}\theta^{\{1,4,6,9,11\}}\theta^{\{2,3,6,9,11\}}\theta^{\{1,3,6,9,11\}}
 \times \\ &\quad \times \theta^{\{2,4,7,8,10\}}\theta^{\{1,4,7,8,10\}}\theta^{\{2,3,7,8,10\}}\theta^{\{1,3,7,8,10\}}\big)^{1/2},\\
 r_2 &= \big(\theta^{\{2,4,6,8,11\}}\theta^{\{1,4,6,8,11\}}\theta^{\{2,3,6,8,11\}}\theta^{\{1,3,6,8,11\}} 
 \times \\ &\quad \times \theta^{\{2,4,7,9,10\}}\theta^{\{1,4,7,9,10\}}\theta^{\{2,3,7,9,10\}}\theta^{\{1,3,7,9,10\}}
 \times \\ &\quad \times \theta^{\{2,4,6,9,10\}}\theta^{\{1,4,6,9,10\}}\theta^{\{2,3,6,9,10\}}\theta^{\{1,3,6,9,10\}}
 \times \\ &\quad \times \theta^{\{2,4,7,8,11\}}\theta^{\{1,4,7,8,11\}}\theta^{\{2,3,7,8,11\}}\theta^{\{1,3,7,8,11\}}\big)^{1/2},\\
 r_3 &= \big(\theta^{\{2,4,6,8,9\}}\theta^{\{1,4,6,8,9\}}\theta^{\{2,3,6,8,9\}}\theta^{\{1,3,6,8,9\}} 
 \times \\ &\quad \times \theta^{\{2,4,7,10,11\}}\theta^{\{1,4,7,10,11\}}\theta^{\{2,3,7,10,11\}}\theta^{\{1,3,7,10,11\}}
 \times \\ &\quad \times \theta^{\{2,4,6,10,11\}}\theta^{\{1,4,6,10,11\}}\theta^{\{2,3,6,10,11\}}\theta^{\{1,3,6,10,11\}}
 \times \\ &\quad \times \theta^{\{2,4,7,8,9\}}\theta^{\{1,4,7,8,9\}}\theta^{\{2,3,7,8,9\}}\theta^{\{1,3,7,8,9\}}\big)^{1/2},
\end{align*}
or in the standard form 
\begin{align*}
 r_1 &= \big(\theta[{}^{00000}_{00000}]\theta[{}^{00000}_{10000}]\theta[{}^{00000}_{01000}]\theta[{}^{00000}_{11000}] 
 \theta[{}^{00110}_{00000}]\theta[{}^{00110}_{10000}]\theta[{}^{00110}_{01000}]\theta[{}^{00110}_{11000}] \times \\ &\quad \times
 \theta[{}^{00010}_{00000}]\theta[{}^{00010}_{10000}]\theta[{}^{00010}_{01000}]\theta[{}^{00010}_{11000}] 
 \theta[{}^{00100}_{00000}]\theta[{}^{00100}_{10000}]\theta[{}^{00100}_{01000}]\theta[{}^{00100}_{11000}]\big)^{1/2},\\
 r_2 &= \big(\theta[{}^{00001}_{00000}]\theta[{}^{00001}_{10000}]\theta[{}^{00001}_{01000}]\theta[{}^{00001}_{11000}] 
 \theta[{}^{00111}_{00000}]\theta[{}^{00111}_{10000}]\theta[{}^{00111}_{01000}]\theta[{}^{00111}_{11000}] \times \\ &\quad \times
 \theta[{}^{00011}_{00000}]\theta[{}^{00011}_{10000}]\theta[{}^{00011}_{01000}]\theta[{}^{00011}_{11000}] 
 \theta[{}^{00101}_{00000}]\theta[{}^{00101}_{10000}]\theta[{}^{00101}_{01000}]\theta[{}^{00101}_{11000}]\big)^{1/2},\\
 r_3 &= \big(\theta[{}^{00000}_{00001}]\theta[{}^{00000}_{10001}]\theta[{}^{00000}_{01001}]\theta[{}^{00000}_{11001}] 
 \theta[{}^{00110}_{00001}]\theta[{}^{00110}_{10001}]\theta[{}^{00110}_{01001}]\theta[{}^{00110}_{11001}] \times \\ &\quad \times
 \theta[{}^{00010}_{00001}]\theta[{}^{00010}_{10001}]\theta[{}^{00010}_{01001}]\theta[{}^{00010}_{11001}] 
 \theta[{}^{00100}_{00001}]\theta[{}^{00100}_{10001}]\theta[{}^{00100}_{01001}]\theta[{}^{00100}_{11001}]\big)^{1/2}.
\end{align*}
The new collection of products $\{r_1,r_2,r_3\}$ again satisfies \eqref{SchRelG5}.
In a similar way one can increase the rank of group $(P)$ in the Schottky relation.

\section*{Aknowledgement}
The author is grateful to Y.\;Kopeliovich and V.\;Enolski for fruitful discussions. 
The author owed to the referee for extending the results to hyperelliptic curves 
with arbitrary complex branch points,
comparatively to the first version where only real 
branch points were taken into account. Discussion with the referee stimulated
finding out exact values of the $8$-th root of unity $\epsilon$ in the first, the second, and the general Thomae formulas.


The datasets generated during the current study are available in the arXiv.


\begin{thebibliography}{99}

\bibitem{bak897}


H.~F. Baker, \emph{Abel's theorem and the allied theory of theta functions},
  Cambridge Univ. Press, Cambridge, (1897), Reprinted in 1995.
  
\bibitem{bak898}
Baker H. F.  \emph{On the Hyperelliptic Sigma Functions}, American Journal of Mathematics
Vol. 20, No. 4 (1898), pp. 301--384.

\bibitem{BerTF2020} J. Bernatska, General derivative Thomae formula for singular half-periods,
Lett. Math. Phys. (2020) (published online)
arXiv:1904.09333 [math.AG]

\bibitem{BEL14} Buchstaber V., Enolski V., Leikin D, Multi-dimensional sigma functions,
 (2014), arXiv: .

\bibitem{ER2008} Enolski V.Z., Richter P.H. Periods of hyperelliptic integrals expressed in terms of 
$\theta$-constants by means of Thomae formulae. 
Phil. Trans. London Math. Soc. A (2008), \textbf{366}, pp.1005--1024

\bibitem{FR1969Schtt} Farkas H.M., Rauch H.E., 
Tho kinds of theta constants  and period relations on a Riemann surface,
Proc. N.A.S., \textbf{62} (1969), pp. 679--686. 


\bibitem{FR1970Schtt} Farkas H.M., Rauch H.E., 
Period relations of Schottky type on Riemann surfaces,
Annals of Mathematics, 2nd Ser., \textbf{92}:3 (1970), pp. 434-461 


\bibitem{FK1980} Farkas H. M., Kra I. Riemann surfaces, Springer-Verlag New York, 1980
in Ser. Graduate Texts in Mathematics, \textbf{71}, pp.\,xi+340.


\bibitem{FK2} Farkas H.M., Kopeliovich Y., New theta constant identities II, 
Proc. Amer. Math. Soc. \textbf{123}:4 (1995), 1009--1020 

\bibitem{FK2001} Farkas H.M., Kra I. Theta Constants, Riemann Surfaces and the Modular Group.
Graduate Studies in Mathematics, V.\,37, AMS, Providence, Rhode Island, 2001

\bibitem{Fra1} Farkas,H.M., Theta functions in complex analysis and number theory. 
Dev. Math. \textbf{17} pp.~57--87 (2008)


\bibitem{fay973}
J.~D. Fay, \emph{Theta functions on Riemann surfaces}, Lectures Notes in
  Mathematics (Berlin), vol. 352, Springer, 1973.

\bibitem{Fro}  Frobenius F. G., Uber die constanten Factoren der Thetareihen, 
J. reine angew. Math. 98 (1885), pp. 244-263; Gesam. Abhand., Springer-Verlag, 1968, Bd. II, pp. 241--260.


\bibitem{Gra1988} Grant D. A generalization of Jacobi's derivative formula to dimension two
J. reine angew. Math, \textbf{392} (1988), pp. 125--136.

\bibitem{GM2005} Grushevsky S., Manni R. S. Two generalizations of Jacobi's derivative formula, 
Mathematical Research Letters, \textbf{12}, 2005, pp.~921-932.

\bibitem{Ig1979} Igusa J.-I. On the Nullwerte of Jacobians of odd theta functions,  
 Symposia Mathematica, 1979, \textbf{24}, pp.~83--95, Academic Press, London-New York, 1981.
 
\bibitem{Ig1980} Igusa J.-I.
On Jacobi's derivative formula and its generalizations,
 American Journal of Mathematics, \textbf{102}:2 (1980), pp. 409--446 

\bibitem{Ig1982Bull} Igusa J.-I. Problems on Abelian functions at the time of Poincar\'{e} and some at present,
Bulletin AMS, \textbf{6}:2, 1982, pp. 161--174.

\bibitem{Ig1982} Igusa J.-I. On the irreducibility of Schottky's divisor, J. Fac. Sci. Univ. Tokyo, 
\textbf{28} (1982), 531--595.


 
\bibitem{Mat2016} Matsuda K. Analogues of Jacobi's derivative formula, The Ramanujan Journal,
\textbf{39}:1, 2016, pp.\,31--47.

\bibitem{Mat2017} Matsuda K. Analogues of Jacobi's derivative formula II, The Ramanujan Journal,
\textbf{44}:1, 2017, pp.\,37--62. 
 
\bibitem{Mum} Mumford D., Tata Lectures on Theta. I. 
Modern Birkhauser Classics. Birkhauser, Boston (2007)



\bibitem{RF1974} Rauch H. E., Farkas H. M. Theta functions with applications to Riemann surfaces,
The Williams \& Wilkins Company, Baltimore, 1974, 232 p.

\bibitem{Rie}  Riemann B., Gesam. math. Werke; Nachtrage by M. Noether \& W. Wirtinger, 1902;
Dover Edition, 1953.

\bibitem{Ros1895} Rosenhain G., Abhandlung \:{u}ber die Functionen zweier Variablen 
mit vier Perioden welche die Inversion sind der ultra-elliptische Integrale erster Klasse, 
1851. Translation to German from Latin manuscript published in: 
Ostwald; Klassiker der Exacten Wissenschaften; Nr. 65, 1-96; Verlag von Wilhelm Engel- mann; 
Leipzig, 1895.

\bibitem{Sch1888}  Schottky F. Zur Theorie der Abelschen Functionen von vier Variablen,
Journal f\"{u}r die reine und angewandte Mathematik \textbf{102} (1888)
pp. 304--352.
 
 
\bibitem{Zemel2019} Zemel S. Evaluating theta derivatives with rational characteristics, Ramanujan J.,\textbf{50}
(2019) pp. 367--391.
 
\end{thebibliography}
\end{document}